\documentclass[12pt]{amsart}
\usepackage{geometry} 
\usepackage{amscd}
\usepackage{amssymb}
\usepackage{amsmath}
\usepackage{amsthm}

\newtheorem{thm}{Theorem}[section]
\newtheorem{lemma}[thm]{Lemma}
\newtheorem{corollary}[thm]{Corollary}
\newtheorem{prop}[thm]{Proposition}

\newtheorem{conjecture}[thm]{Conjecture}
\theoremstyle{definition}
\newtheorem{notation}[thm]{Notation}
\newtheorem{rem}[thm]{Remark}
\newtheorem{exa}[thm]{Example}
\newtheorem{defn}[thm]{Definition}

\makeatletter

\newcommand{\isom}{\overset{\sim}{\rightarrow}}

\title{A generalization of Kato's local $\varepsilon$-conjecture for $(\varphi,\Gamma)$-modules 
over the Robba ring.}
\author{Kentaro Nakamura*}
\date{} 

\begin{document}

\maketitle
\pagestyle{plain}
\footnote{2010 Mathematical Subject Classification 11F80 (primary), 11F85, 11S25 (secondary).
Keywords: $p$-adic Hodge theory, $(\varphi,\Gamma)$-module, $B$-pair.}
\footnote{Department of Mathematics, Hokkaido University, Kita 10, Nishi 8, Kita-Ku, Sapporo, Hokkaido, 060-0810, Japan.}
\footnote{e-mail address:kentaro@math.sci.hokudai.ac.jp}
\begin{abstract}
The aim of this article is to generalize Kato's ( commutative ) $p$-adic local $\varepsilon$-conjecture for families of $(\varphi,\Gamma)$-modules over the Robba rings. In particular, we prove the generalized local $\varepsilon$-conjecture 
for rank one $(\varphi,\Gamma)$-modules, which is a generalization of the Kato's theorem for rank one $p$-adic Galois representations.
The key ingredients are the author's previous work on the Bloch-Kato exponential map for $(\varphi,\Gamma)$-modules and the recent results of Kedlaya-Pottharst-Xiao on the finiteness of cohomology of $(\varphi,\Gamma)$-modules.
\end{abstract}
\setcounter{tocdepth}{2}
\tableofcontents

\section{Introduction.}
\subsection{Introduction}
Since the works of Kisin \cite{Ki03}, Colmez \cite{Co08}, and Bella\"iche- Chenevier \cite{BelCh09} etc, the theory of $(\varphi,\Gamma)$-modules over the ( relative ) Robba ring becomes one of the main streams in the theory of $p$-adic Galois representations. In particular, the recent works of Pottharst \cite{Po13a} and Kedlaya-Pottharst-Xiao \cite{KPX14} established the fundamental theorems (comparison with the Galois cohomology, finiteness, base change property, Tate duality, Euler-Poincar\'e formula ) in the theory of the cohomology of $(\varphi,\Gamma)$-modules over the relative Robba ring over $\mathbb{Q}_p$-affinoid algebras. As is suggested and actually given in \cite{KPX14}, \cite{Po13b}, their results are expected to have many applications in number theory (e.g. eigenvarietes, non-ordinary case of Iwasawa theory). On the other hand, in \cite{Na14a}, the author of this article generalized the theory of Bloch-Kato exponential map and Perrin-Riou's exponential map in the framework of $(\varphi,\Gamma)$-modules over the Robba ring. Since these maps are very important tools in Iwasawa theory, he expects that the results of \cite{Na14a} also have many applications.
As an application of the both theories, the purpose of this article is to generalize
Kato's $p$-adic local $\varepsilon$-conjecture \cite{Ka93b} in the framework of $(\varphi,\Gamma)$-modules over
the relative Robba ring over $\mathbb{Q}_p$-affinoid algebras, which we briefly explain in this
introduction: see $\S$3 for the precise definitions. Let $G_{\mathbb{Q}_p}$ be the absolute Galois
group of $\mathbb{Q}_p$. Let $\Lambda$ be a semi-local ring such that $\Lambda/\mathfrak{m}_{\Lambda}$ is a finite ring of the order
a power of $p$, where $\mathfrak{m}_{\Lambda}$ is the Jacobson radical of $\Lambda$. Let $T$ be a $\Lambda$-representation
of $G_{\mathbb{Q}_p}$ ,i.e. a finite projective $\Lambda$-module with a continuous $\Lambda$-linear $G_{\mathbb{Q}_p}$-action.
Let $C_{\mathrm{cont}}^{\bullet}(G_{\mathbb{Q}_p}, T)$ be the complex of continuous cochains of $G_{\mathbb{Q}_p}$ with the values in $T$. By the classical theory of Galois cohomology of $G_{\mathbb{Q}_p}$, this complex is a perfect complex of $\Lambda$-module which satisfies the base change property, Tate duality,$\cdots$. This fact enables us to define the determinant $\mathrm{Det}_{\Lambda}(C_{\mathrm{cont}}^{\bullet}(G_{\mathbb{Q}_p},T))$ which is a ( graded ) invertible $\Lambda$-module. Modifying this module by multiplying a kind of $\mathrm{det}_{\Lambda}(T)$, one can canonically attach a graded invertible $\Lambda$-module
$$\Delta_{\Lambda}(T)$$
called the fundamental line of the pair $(\Lambda,T)$, which is compatible with base change and Tate duality. Our main objects are the pairs $(A,M)$ where $A$ is a $\mathbb{Q}_p$-affinoid and $M$ is a $(\varphi,\Gamma)$-module over the relative Robba ring $\mathcal{R}_A$ over $A$. By the results of  \cite{KPX14}, then we can similarly attach a graded invertible $A$-module
$$\Delta_A(M)$$
such that, for a pair $(\Lambda,T)$ and a continuous homomorphism $f:\Lambda\rightarrow A$, there exists a canonical comparison isomorphism 
$$\Delta_{\Lambda}(T)\otimes_{\Lambda}A\isom \Delta_A(\bold{D}_{\mathrm{rig}}(T\otimes_{\Lambda}A))$$
by
the result of \cite{Po13a}. The following conjecture is the Kato's conjecture if $(B,N)=(\Lambda,T)$, and our new conjecture if $(B,N)=(A,M)$ . See Conjecture \ref{3.9} for the precise formulation.

\begin{conjecture}\label{1.1}$(\mathrm{Conjecture}\, \ref{3.9})$ We can uniquely define a $B$-linear isomorphism
$$\varepsilon_{B,\zeta}(N):\bold{1}_B\isom\Delta_B(N)$$
for each pair $(B,N)$ of type $(\Lambda,T)$ or $(A, M)$ and for each $\mathbb{Z}_p$-basis $\zeta$ of $\mathbb{Z}_p(1)$, which is compatible with any base changes $B\rightarrow B'$, exact sequences 
$0\rightarrow N_1\rightarrow N_2\rightarrow N_3\rightarrow 0$, and Tate duality, and satisfies the following:
\begin{itemize}
\item[$(\mathrm{v})$] For any $f:\Lambda\rightarrow A$ as above, we have
$$\varepsilon_{\Lambda,\zeta}(T)\otimes\mathrm{id}_A=\varepsilon_{A,\zeta}(\bold{D}_{\mathrm{rig}}(T\otimes_{\Lambda}A))$$
under the canonical isomorphism $\Delta_{\Lambda}(T)\otimes_{\Lambda}A\isom \Delta_A(\bold{D}_{\mathrm{rig}}(T\otimes_{\Lambda}A))$.
\item[$(\mathrm{vi})$]
 Let $L=A$ be a finite extension of $\mathbb{Q}_p$, and let $N$ be a de Rham represen-
tation of $G_{\mathbb{Q}_p}$ or de Rham $(\varphi,\Gamma)$-module over $\mathcal{R}_L$. Then we have 
 $$\varepsilon_{L,\zeta}(N)=\varepsilon_{L,\zeta}^{\mathrm{dR}}(N),$$
where the isomorphism
$$\varepsilon_{L,\zeta}^{\mathrm{dR}}(N):\bold{1}_L\isom\Delta_L(N)$$
 which is called the de Rham $\varepsilon$-isomorphism is defined using the Bloch-Kato exponential and the dual exponential of $N$ and the local factors ($L$-factor, $\varepsilon$-constant, ``gamma factor") associated to $\bold{D}_{\mathrm{pst}}(N)$ and $\bold{D}_{\mathrm{pst}}(N^*)$.
 \end{itemize}
 \end{conjecture}
\begin{rem}To define the condition (vi) for de Rham  $(\varphi,\Gamma)$-modules, we need to generalize the Bloch-Kato exponential for $(\varphi,\Gamma)$-modules, which was 
 done in \cite{Na14a}.
 \end{rem}
Roughly speaking, this conjecture says that the local factor which appears in
the functional equation of the $L$-functions of motif $p$-adically interpolate to all the
families of $p$-adic Galois representations and also rigid analytically interpolate to
all the families of $(\varphi,\Gamma)$-modules in a compatible way. In fact, in \cite{Ka93a}, Kato
formulated a conjecture called the generalized Iwasawa main conjecture which
asserts the existence of a compatible family of ``zeta"-isomorphisms 
 $$z_{\Lambda}(\mathbb{Z}[1/S],T):\bold{1}_{\Lambda}\isom\Delta_{\Lambda}^{\mathrm{global}}(T)$$ for any $\Lambda$-representation $T$ of $G_{\mathbb{Q},S}$ ($S$ is a finite set of primes) 
which interpolate the special values of $L$-functions of motif, and also in \cite{Ka93b} formulated another conjecture called the global $\varepsilon$-conjecture which asserts the functional equation between $z_{\Lambda}(\mathbb{Z}[1/S],T)$ and $z_{\Lambda}(\mathbb{Z}[1/S],T^*)$ 
 whose local factor at $p$ is $\varepsilon_{\Lambda}(T|_{G_{\mathbb{Q}_p}})$.
In \cite{Ka93b} (see also \cite{Ve13}), Kato proved the local (and even the global) $\varepsilon$- conjecture for the rank one case. As a generalization of his theorem, our main theorem of this article is the following.
 
\begin{thm} $(\mathrm{Theorem} \,\,\ref{3.12})$ Conjecture $\ref{1.1}$ is true for the rank case.
 \end{thm}
 
\begin{rem}From this theorem, we can immediately obtain some results in the trianguline case. 
In particular, since the $(\varphi,\Gamma)$-modules associated to the twists of crystalline representations are trianguline, we can compare our results with the previous results \cite{BB08}, \cite{LVZ14} on the 
local $\varepsilon$-conjecture for the crystalline case. See Corollary \ref{3.13}, \ref{3.14} for more details.
 \end{rem}
\begin{rem}
 The non trianguline case is much more difficult, but is much more interesting since this case corresponds to the supercuspidal representations of $\mathrm{GL}_n(\mathbb{Q}_p)$ via local Langlands correspondence, whose $\varepsilon$-constants are in general  difficult to explicitly describe. In the next article \cite{Na}, we construct $\varepsilon$-isomorphisms for all rank two torsion $p$-adic representations of $\mathrm{Gal}(\overline{\mathbb{Q}}_p/\mathbb{Q}_p)$ using Colmez's theory \cite{Co10} of  $p$-adic local Langlands correspondence for $\mathrm{GL}_2(\mathbb{Q}_p)$. More precisely, we will show that (a modified version of) the pairing defined in Corollaire VI.
 6.2 of \cite{Co10} essentially gives us 
 $\varepsilon$-isomorphisms for the rank two case.
 In trianguline case, 
 we will show that the $\varepsilon$-isomorphisms constructed in \cite{Na} coincide with those constructed in this article. More interestingly, for de Rham and non-trianguline case,
  we will show that the $\varepsilon$-isomorphisms defined in \cite{Na} satisfy the suitable
  interpolation property for the critical range of Hodge-Tate weights using Emerton's theorem 
  on the compatibility of classical and $p$-adic Langlands correspondence \cite{Em}. 
  Moreover, as an application, we will prove a functional equation of Kato's Euler system 
  associated to any Hecke eigen elliptic cusp form without any condition at $p$.
   \end{rem}
 
\subsection{Structure of the paper} 
In $\S2$, we recall the results of \cite{KPX14}, \cite{Po13a}
  and \cite{Na14a}. After recalling the definition of $(\varphi,\Gamma)$-modules over the relative Robba ring, we recall the main results of \cite{KPX14},\cite{Po13a} on the cohomology of $(\varphi,\Gamma)$- modules, i.e. comparison with Galois cohomology, finiteness, base change property, Euler-Poincar\'e formula, Tate duality, and the classification of rank one objects, all of which are essential for the formulation of our conjecture. We next recall the result of \cite{Na14a}  on the theory of the Bloch-Kato exponential map of $(\varphi,\Gamma)$-modules. Since the result of \cite{Na14a} is not sufficient for our purpose, we slightly generalize the result of this paper. In particular, we show the existence of Bloch-Kato's fundamental exact sequences involving $\bold{D}_{\mathrm{cris}}(M)$ (Proposition \ref{2.21}), establishing Bloch-Kato's duality for the finite cohomology of $(\varphi,\Gamma)$-modules (Proposition \ref{2.24}). The explicit formulae of our Bloch-Kato's exponential maps (Proposition \ref{2.23}) are frequently used in later sections.
 
 In $\S$3, using the preliminaries recalled in $\S$2, we formulate our $\varepsilon$-conjecture and state our main theorem of this paper. Since the conjecture is formulated by using the notion of determinant, we first recall this notion in $\S$3.1. In $\S$3.2, using the determinant of cohomology of $(\varphi,\Gamma)$-modules, we define a graded invertible $A$- module $\Delta_A(M)$ called the fundamental line for any $(\varphi,\Gamma)$-module $M$ over $\mathcal{R}_A$. In $\S$3.3, for any de Rham $(\varphi,\Gamma)$-module $M$ , we define a trivialization (called de Rham $\varepsilon$-isomorphism) of the fundamental line using the Bloch-Kato fundamental exact sequence, Deligne-Langlands-Fontaine-Perrin-Riou's $\varepsilon$-constants and the ``gamma -factor" associated to $\bold{D}_{\mathrm{pst}}(M)$. In $\S$3.4, we formulate our conjecture and compare our conjecture with Kato's conjecture, and state our main theorem of this article, which solves the conjecture for all rank one $(\varphi,\Gamma)$-modules.
 
$\S$4 is the main part of this paper, where we prove the conjecture for the rank one case. In $\S$ 4.1, using the theory of analytic Iwasawa cohomology \cite{KPX14}, \cite{Po13b}, and using the standard technique of $p$-adic Fourier transform, we construct our $\varepsilon$-isomorphism for all rank one $(\varphi,\Gamma)$modules. In $\S$ 4.2, we show that our $\varepsilon$-isomorphism defined in $\S$ 4.1 specializes to the de Rham $\varepsilon$-isomorphism defined in $\S$ 3.2 at each de Rham point. 
In $\S$ 4.2.1, we first verify this condition (which we call the de Rham condition) for the ``generic" rank one de Rham $(\varphi,\Gamma)$- modules by establishing a kind of explicit reciprocity law (Proposition \ref{4.14}, \ref{4.15}). In the process of proving this, we prove a proposition (Proposition \ref{4.13.5}) on the compatibility of our $\varepsilon$-isomorphism with a natural differential operator. Using the result in the generic case and the density argument, we prove the compatibility of our $\varepsilon$-isomorphism with Tate duality and compare our $\varepsilon$-isomorphism  with Kato's $\varepsilon$-isomorphism. In $\S$ 4.2.2, we verify the de Rham condition via explicit calculations for the exceptional case which includes the case of $\mathcal{R}, \mathcal{R}(1)$ (the $(\varphi,\Gamma)$-modules corresponding to $\mathbb{Q}_p, \mathbb{Q}_p(1)$ respectively).

In the appendix, we explicitly calculate the cohomologies $\mathrm{H}^i_{\varphi,\gamma}(\mathcal{R}(1))$ and $\mathrm{H}^i_{\varphi,\gamma}(\mathcal{R})$, which will be used in $\S$ 4.2.2.

\subsection{Notation}
Throughout this paper, we fix a prime number $p$. 
The letter $A$ will always denote a $\mathbb{Q}_p$-affinoid algebra; we use 
$\mathrm{Max}(A)$ to denote the associated rigid analytic space. 
Fix an algebraic closure $\overline{\mathbb{Q}}_p$ of $\mathbb{Q}_p$, and we consider any finite extension $K$ 
of $\mathbb{Q}_p$ inside $\overline{\mathbb{Q}}_p$. Let  $|-|:\overline{\mathbb{Q}}_p^{\times}\rightarrow \mathbb{Q}_{>0}$ to be the absolute value such that $|p|=p^{-1}$.
For $n\geqq 0$, let denote $\mu_{p^n}$ for the set of $p^n$-th power roots of unity in $\overline{\mathbb{Q}}_p$, and put
$\mu_{p^{\infty}}:=\cup_{n\geqq 1}\mu_{p^n}$. 
For a finite extension $K$ of $\mathbb{Q}_p$, put
$K_n:=K(\mu_{p^n})$ for $\infty\geqq n\geqq 0$. 
Let denote
$\chi:\Gamma_{\mathbb{Q}_p}:=\mathrm{Gal}(\mathbb{Q}_{p,\infty}/\mathbb{Q}_p)\isom 
\mathbb{Z}_p^{\times}$ for the cyclotomic character given by 
$\gamma(\zeta)=\zeta^{\chi(\gamma)}$ for $\gamma\in \Gamma$ and 
$\zeta\in \mu_{p^{\infty}}$. 
Set $G_K:=\mathrm{Gal}(\overline{\mathbb{Q}}_p/K)$, $H_K:=\mathrm{Gal}(\overline{\mathbb{Q}}_p/
K_{\infty})$, and $\Gamma_K:=\mathrm{Gal}(K_{\infty}/K).$
We let $k$  be the residue field of $K$, with $F:=W(k)[1/p]$.
Put $\mathbb{Z}_p(1):=\varprojlim_{n\geqq 0}\mu_{p^n}$. 
For $k\in \mathbb{Z}$, put $\mathbb{Z}_p(k):=
\mathbb{Z}_p(1)^{\otimes k}$ equipped with a natural action of $\Gamma_K$. 
For a $\mathbb{Z}_p[G_K]$-module 
$N$, let denote $N(k):=N\otimes_{\mathbb{Z}_p}\mathbb{Z}_p(k)$. When we fix a generator 
$\zeta=\{\zeta_{p^n}\}_{n\geqq 0}\in \mathbb{Z}_p(1)$, we put $\bold{e}_1:=\zeta$ and 
$\bold{e}_k:=\bold{e}_1^{\otimes k}\in \mathbb{Z}$. 
For a continuous $G_K$-module $N$, let denote by $C^{\bullet}_{\mathrm{cont}}(G_K,N)$ 
the complex of continuous cochains of $G_K$ with the values in $N$. 
Denote by $\mathrm{H}^i(K, N):=\mathrm{H}^i(C^{\bullet}_{\mathrm{cont}}(G_K,N))$.
For a group $G$, denote $G_{\mathrm{tor}}$ for the subgroup of $G$ consisting of 
all torsion elements in $G$.
If $G$ is finite group, denote $|G|$ for the order of $G$.
For a commutative ring $R$, let denote by $\bold{P}_{\mathrm{fg}}(R)$  the category of 
finitely generated projective $R$-modules. For $N\in \bold{P}_{\mathrm{fg}}(R)$, denote by $\mathrm{rk}_RN$ 
the rank of $N$ and by $N^{\vee}:=\mathrm{Hom}_R(N, R)$. Let 
$[,]:N_1\times N_2\rightarrow R$ be a perfect pairing. Then we always identify $N_2$ with $N_1^{\vee}$ by 
the isomorphism $N_2\isom N_1^{\vee}: x\mapsto (y\mapsto [y,x])$. Let denote by 
$\bold{D}^{-}(R)$ the derived category of bounded below complexes of $R$-modules.
For $a_1\leqq a_2\in \mathbb{Z}$, let denote by $\bold{D}_{\mathrm{perf}}^{[a_1,a_2]}(R)$ (respectively $\bold{D}^b_{\mathrm{perf}}(R)$) 
denote the full subcategory of $\bold{D}^{-}(R)$ consisting of the complexes of $R$-modules which are quasi isomorphic 
to a complex $P^{\bullet}$ of $\bold{P}_{\mathrm{fg}}(R)$ concentrated in 
degrees in $[a_1,a_2]$ (respectively bounded degree). There exists a duality functor 
$\bold{R}\mathrm{Hom}_R(-,R):\bold{D}^{[a_1,a_2]}_{\mathrm{perf}}(R)\rightarrow 
\bold{D}^{[-a_2,-a_1]}_{\mathrm{perf}}(R)$ characterized by 
$\bold{R}\mathrm{Hom}_R(P^{\bullet},R):=\mathrm{Hom}_R(P^{-\bullet}, R)$ for any 
bounded complex $P^{\bullet}$ of $\bold{P}_{\mathrm{fg}}(R)$. 
Define the notion $\chi_R(-)$ of Euler characteristic  for any objects of  $ \bold{D}^{b}_{\mathrm{perf}}(R)$ which is characterized by 
$\chi_R(P^{\bullet}):=\sum_{i\in \mathbb{Z}}(-1)^i\mathrm{rk}_RP^i\in \mathrm{Map}
(\mathrm{Spec}(R), \mathbb{Z})$ for any bounded complex $P^{\bullet}$ of $\bold{P}_{\mathrm{fg}}(R)$.

\section{Cohomology and Bloch-Kato exponential of $(\varphi,\Gamma)$-modules}

\subsection{Cohomology of $(\varphi,\Gamma)$-modules}
In this subsection, we recall the definition of ( families of ) $(\varphi,\Gamma)$-modules 
and the definition of their cohomologies following 
\cite{KPX14}, and then recall the results of their article on the 
finiteness of the cohomology. 


 Put $\omega:=p^{-1/(p-1)}\in \mathbb{R}_{>0}$. 
 For $r\in \mathbb{Q}_{>0}$, define the $r$-Gauss norm $|-|_r$ on $\mathbb{Q}_p[T^{\pm}]$
 by the formula
 $|\sum_i a_iT^i|_r:=\mathrm{max}_{i}\{|a_i|\omega^{ir}\}$. 
 For $0<s\leqq r \in \mathbb{Q}_{>0}$, we write $A^1[s,r]$ for the rigid analytic annulus defined over 
 $\mathbb{Q}_p$ in the variable $T$ with radii 
 $|T|\in [\omega^r,\omega^s]$; its ring of analytic functions, denoted by $\mathcal{R}^{[s,r]}$, is the 
 completion of $\mathbb{Q}_p[T^{\pm}]$ with respect to the norm $|\cdot |_{[s,r]}
 :=\mathrm{max}\{|\cdot|_r, |\cdot|_s\}$. We also allow $r$ (but not $s$) to be $\infty$, in which 
 case $A^1[s,r]$ is interpreted as the rigid analytic disc in the variable $T$ with radii $|T|\leqq \omega^s$; its ring of analytic functions $\mathcal{R}^{[s,r]}=\mathcal{R}^{[s,\infty]}$ is the completion of $\mathbb{Q}_p[T]$ with respect to $|\cdot|_s$. Let $A$ be a $\mathbb{Q}_p$-affinoid algebra.
 Let $\mathcal{R}_A^{[s,r]}$ denote the ring of rigid analytic functions on the relative annulus 
 (or disc if $r=\infty$) $\mathrm{Max}(A)\times A^1[s,r]$; its ring of analytic functions is 
 $\mathcal{R}_A^{[s,r]}:=
 \mathcal{R}^{[s,r]}\hat{\otimes}_{\mathbb{Q}_p}A$. Put 
 $\mathcal{R}_A^r:=\cap_{0<s\leqq r}\mathcal{R}_A^{[s,r]}$ and 
 $\mathcal{R}_A:=\cup_{0<r}\mathcal{R}_A^r$.

Let $k'$ be the residue field of $K_{\infty}$, with $F':=W(k')[1/p]$. Put $\tilde{e}_K:=[K_{\infty}: F'_{\infty}]$. 

For $0<s\leqq r$, we set 
$\mathcal{R}^{[s,r]}(\pi_K)$ to be the formal substitution of $T$ by $\pi_K$ in the ring 
$\mathcal{R}^{[s/\tilde{e}_K,r/\tilde{e}_K]}_{F'}$; we set 
$\mathcal{R}_A^{[s,r]}(\pi_K):=\mathcal{R}^{[s,r]}(\pi_K)\hat{\otimes}_{\mathbb{Q}_p}A$. 
We define $\mathcal{R}^r_A(\pi_K), \mathcal{R}_A(\pi_K)$ similarly; the latter is referred to as 
the relative Robba ring over $A$ for $K$.

By the theory of fields of norms, there exists 
a constant $C(K)>0$ and, for any $0<r\leqq C(K)$, we can 
equip $\mathcal{R}_{A}^{r}(\pi_K)$ with a finite 
\'etale $\mathcal{R}^r_A(\pi_{\mathbb{Q}_p})$ algebra free of rank 
$[K_{\infty}:\mathbb{Q}_{p,\infty}]$ with the Galois group $H_{\mathbb{Q}_p}/H_K$.
 More generally,  for any finite extensions 
$L\supseteq K\supseteq \mathbb{Q}_p$, we can naturally 
equip $\mathcal{R}_{A}^{r}(\pi_L)$ with a structure 
of finite \'etale $\mathcal{R}_A^{r}(\pi_K)$-algebra free of rank $[L_{\infty}:K_{\infty}]$
 with the Galois group $H_K/H_L$ for any $0<r\leqq \mathrm{min}\{C(K), 
 C(L)\}$.

There are commuting $A$-linear actions of $\Gamma_K$ on $\mathcal{R}^{[s,r]}_A(\pi_K)$
 and of an operator 
$\varphi:\mathcal{R}_A^{[s,r]}(\pi_K)\rightarrow \mathcal{R}_A^{[s/p,r/p]}(\pi_K)$ 
for $0<s\leqq r\leqq C(K)$. 
The actions on the coefficients $F'$ are the natural ones, i.e. $\Gamma_K$ through its 
quotient $\mathrm{Gal}(F'/F)$ and $\varphi$ by the canonical lift of $p$-th Frobenius 
on $k'$. 
For $0<s\leqq r\leqq C(K)$, $\varphi$ makes 
$\mathcal{R}_A^{[s/p,r/p]}(\pi_K)$ into a free $\mathcal{R}_A^{[s,r]}(\pi_K)$-module 
of rank $p$, and we obtain a $\Gamma_K$-equivariant left inverse 
$\psi:\mathcal{R}_A^{[s/p,r/p]}(\pi_K)\rightarrow \mathcal{R}_A^{[s,r]}(\pi_K)$ by 
the formula $\frac{1}{p}\varphi^{-1}\circ \mathrm{Tr}_{\mathcal{R}_A^{[s/p,r/p]}(\pi_K)/
\varphi(\mathcal{R}_A^{[s,r]}(\pi_K))}$. The map $\psi$ naturally extends to the maps 
$\mathcal{R}^{r/p}_A(\pi_K)\rightarrow \mathcal{R}_A^{r}(\pi_K)$ for $0< r\leqq C(K)$ and 
$\mathcal{R}_A(\pi_K)\rightarrow \mathcal{R}_A(\pi_K)$. 

\begin{rem}\label{2.1}
In fact, these rings are constructed using Fontaine's rings of $p$-adic periods. 
We don't have any canonical choice of the parameter $\pi_K$ for general $K$, but 
the ring $\mathcal{R}_A(\pi_K)$ and the actions of $\varphi$, $\Gamma_K$ don't depend on the choice of $\pi_K$. More 
precisely, $\mathcal{R}(\pi_K)$ is defined as a subring of the ring $\widetilde{\bold{B}}^{\dagger}_{\mathrm{rig}}$ of $p$-adic periods defined in \cite{Ber02}, and this subring does not depend on the choice 
of $\pi_K$, and the actions of $\varphi$, $\Gamma_K$ is induced by the natural actions of 
$\varphi$, $G_K$ on $\widetilde{\bold{B}}^{\dagger}_{\mathrm{rig}}$.

However, for unramified $K$, once we fix a $\mathbb{Z}_p$-basis $\zeta:=\{\zeta_{p^n}\}_{n\geqq 0}$  
of $\mathbb{Z}_p(1):=\varprojlim_{n\geqq 0}\mu_{p^n}$, 
 we have a natural choice of $\pi_K$ as follows. Let $\overline{\mathbb{Z}}_p$ be the integral closure of $\overline{\mathbb{Q}}_p$, and let $\widetilde{\mathbb{E}}^+:=\varprojlim_{n\geqq 0}\overline{\mathbb{Z}}_p/p\overline{\mathbb{Z}}_p$ be the projective limit with respect to $p$-th power map, and let $[-]:\widetilde{\mathbb{E}}^+\rightarrow W(\widetilde{\mathbb{E}}^+)$ be 
the Teichm\"uller lift to the ring $W(\widetilde{\mathbb{E}}^+)$ of Witt vectors.
Under the fixed $\zeta$, 
we can choose $\pi_K=\pi_{\mathbb{Q}_p}=\pi_{\zeta}:=[(\overline{\zeta}_{p^n})_{n\geqq 0}]-1\in W(\widetilde{\mathbb{E}}^+)\subseteq \widetilde{\bold{B}}^{\dagger}_{\mathrm{rig}}$, and then $\varphi$ and $\Gamma_{\mathbb{Q}_p}$ act by 
$\varphi(\pi_{\zeta})=(1+\pi_{\zeta})^p-1$ and $\gamma(\pi_{\zeta})=(1+\pi_{\zeta})^{\chi(\gamma)}-1$ for $\gamma\in \Gamma_{\mathbb{Q}_p}$.

\end{rem}
\begin{notation}\label{2.2}
From \S 3, we will concentrate on the case $K=\mathbb{Q}_p$ and fix  $\zeta:=\{\zeta_{p^n}\}_{n\geqq 0}$ as above. 
Then, we use the notation $\Gamma:=\Gamma_{\mathbb{Q}_p}$, $\pi:=\pi_{\zeta}$ and omit $(\pi_{\mathbb{Q}_p})$ from the 
notation of Robba rings by writing, for example, $\mathcal{R}_A^{[s,r]}$ instead of $\mathcal{R}_A^{[s,r]}(\pi_{\mathbb{Q}_p})$. 
In this case, $\mathcal{R}_A^{[s/p,r/p]}=\oplus_{0\leqq i\leqq p-1}
(1+\pi)^i\varphi(\mathcal{R}_A^{[s,r]})$, so if $f=\sum_{i=0}^{p-1}(1+\pi)^i\varphi(f_i)$ then 
$\psi(f)=f_0$. 
We define the special element $t=\mathrm{log}(1+\pi)\in \mathcal{R}_A^{\infty}$. We have 
$\varphi(t)=pt$ and $\gamma(t)=\chi(\gamma)t$ for $\gamma\in \Gamma$.
\end{notation}

We first recall the definitions of $\varphi$-modules over $\mathcal{R}_A(\pi_K)$  following Definition 2.2.5 of \cite{KPX14}. 

\begin{defn}\label{2.3}
Choose $0<r_0\leqq C(K)$. A $\varphi$-module over $\mathcal{R}^{r_0}_A(\pi_K)$ is a finite 
projective $\mathcal{R}^{r_0}_A(\pi_K)$-module $M^{r_0}$ equipped with a $\mathcal{R}_A^{r_0/p}(\pi_K)$-linear isomorphism 
$\varphi^*M^{r_0}\isom M^{r_0}\otimes_{\mathcal{R}_A^{r_0}(\pi_K)}\mathcal{R}_A^{r_0/p}(\pi_K)$.
A $\varphi$-module $M$ over $\mathcal{R}_A(\pi_K)$ is a base change to $\mathcal{R}_A(\pi_K)$ of a $\varphi$-module over some $\mathcal{R}^{r_0}_A(\pi_K)$.

\end{defn}

For a $\varphi$-module $M^{r_0}$ over $\mathcal{R}^{r_0}_A(\pi_K)$ and for $0<s\leqq r\leqq r_0$, 
we set $M^{[s,r]}=M^{r_0}\otimes_{\mathcal{R}_A^{r_0}(\pi_K)}\mathcal{R}^{[s,r]}_A(\pi_K)$ 
and $M^s=M^{r_0}\otimes_{\mathcal{R}_A^{r_0}(\pi_K)}\mathcal{R}^s_A(\pi_K)$.
For $0<s\leqq r_0$, the given isomorphism $\varphi^*(M^{r_0})\isom M^{r_0/p}$ induces a
 $\varphi$-semilinear map 
$\varphi:M^s\rightarrow \varphi^*M^s\isom \varphi^*M^{r_0}\otimes_{\mathcal{R}^{r_0/p}_A(\pi_K)}
\mathcal{R}_A^{s/p}(\pi_K)\isom M^{r_0/p}\otimes_{\mathcal{R}^{r_0/p}_A(\pi_K)}\mathcal{R}^{s/p}_A(\pi_K)=M^{s/p}$, where the first map $M^s\hookrightarrow \varphi^*M^s$ is given by $x\mapsto x\otimes 1 
\in M^s\otimes_{\mathcal{R}^{s}_A(\pi_K),\varphi}\mathcal{R}^{s/p}_A(\pi_K)=:\varphi^{*}M^s$ and the second isomorphism is just the associativity of tensor products and the third isomorphism 
is the base change of the given isomorphism $\varphi^*M^{r_0}\isom M^{r_0/p}$. 
This map $\varphi$ also induces an $A$-linear homomorphism 
$\psi:M^{s/p}=\varphi(M^s)\otimes_{\varphi(\mathcal{R}_A^s(\pi_K))}\mathcal{R}_A^{s/p}(\pi_K)
\rightarrow M^s$ given by $\psi(\varphi(m)\otimes f)=m\otimes \psi(f)$ for $m\in M^s$ and $f\in 
\mathcal{R}^{s/p}_A(\pi_K)$. For a $\varphi$-module $M$ over 
$\mathcal{R}_A(\pi_K)$, the maps $\varphi:M^s\rightarrow M^{s/p}$ and $\psi:M^{s/p}\rightarrow M^{s}$ naturally extend to $\varphi:M\rightarrow M$ and $\psi:M\rightarrow M$.

We recall the definition of $(\varphi,\Gamma)$-modules over $\mathcal{R}_A(\pi_K)$ 
following Definition 2.2.12 of \cite{KPX14}. 

\begin{defn}
Choose $0<r_0\leqq C(K) $. A $(\varphi,\Gamma)$-module over 
$\mathcal{R}_A^{r_0}(\pi_K)$ is a $\varphi$-module over 
$\mathcal{R}^{r_0}_A(\pi_K)$ equipped with a commuting semilinear continuous 
action of $\Gamma_K$. A $(\varphi,\Gamma)$-modules 
over $\mathcal{R}_A(\pi_K)$ is a base change of a $(\varphi,\Gamma)$-module 
over $\mathcal{R}_A^{r_0}(\pi_K)$ for some $0<r_0\leqq C(K)$.

\end{defn}

We can generalize these notions for general rigid analytic space as in 
Definition 6.1.1 of \cite{KPX14}
\begin{defn}\label{2.5}
Let  $X$ be a rigid analytic space over $\mathbb{Q}_p$.
A $(\varphi,\Gamma)$-module over $\mathcal{R}_X(\pi_K)$ is a compatible family of 
$(\varphi,\Gamma)$-modules over $\mathcal{R}_A(\pi_K)$ for each 
affinoid $\mathrm{Max}(A)$ of $X$. \end{defn}

For $(\varphi,\Gamma)$-modules $M$, $N$ over $\mathcal{R}_X(\pi_K)$. 
We define $M\otimes N:=M\otimes_{\mathcal{R}_X(\pi_K)}N$ for the tensor product 
equipped with the diagonal action of ($\varphi,\Gamma_K$). We also define 
$M^{\vee}:=\mathrm{Hom}_{\mathcal{R}_X(\pi_K)}(M, \mathcal{R}_X(\pi_K))$ for 
the dual $(\varphi,\Gamma)$-module.

For a $(\varphi,\Gamma)$-module $M$ over $\mathcal{R}_A(\pi_K)$. We denote 
$r_M:=\mathrm{rk}_{\mathcal{R}_A(\pi_K)}M\in \mathrm{Map}(\allowbreak\mathrm{Spec}(\mathcal{R}_A(\pi_K)), 
\mathbb{Z}_{\geqq 0})$ for the rank of $M$, where $\mathrm{Map}(-,-)$ is 
the set of continuous maps and $\mathbb{Z}_{\geqq 0}$ is equipped with the discrete 
topology. We will see later (in Remark \ref{rank}) that $r_M$ is 
in fact in $\mathrm{Map}(\mathrm{Spec}(A), \mathbb{Z}_{\geqq 0})$, i.e 
we have $r_M=\mathrm{pr}\circ f_M$ for unique $f_M\in \mathrm{Map}(\mathrm{Spec}(A), 
\mathbb{Z}_{\geqq 0})$, where $\mathrm{pr}:\mathrm{Spec}(\mathcal{R}_A(\pi_K))\rightarrow 
\mathrm{Spec}(A)$ is the natural projection. We also denote by $r_M:=f_M$.

The importance of $(\varphi,\Gamma)$-module follows from the following theorem.

\begin{thm}\label{2.6}$(\mathrm{Theorem}\,3.11$ $\mathrm{of}$ \cite{KL10}$)$
Let $V$ be a vector bundle over $X$ equipped with a continuous 
$\mathcal{O}_X$-linear action of $G_K$. Then there is functorially associated to $V$ a 
$(\varphi,\Gamma)$-module $\bold{D}_{\mathrm{rig}}(V)$ over 
$\mathcal{R}_X(\pi_K)$. The rule $V\mapsto \bold{D}_{\mathrm{rig}}(V)$ is fully faithful 
and exact, and it commutes with base change in $X$. 
\end{thm}

For example, we have a canonical isomorphism $\bold{D}_{\mathrm{rig}}(A(k))=\mathcal{R}_A(\pi_K)(k)$ 
for $k\in \mathbb{Z}$.


From \S3, we will concentrate on the case where $K=\mathbb{Q}_p$ and $M$ is rank one $(\varphi,\Gamma)$-modules
over $\mathcal{R}_X$. Here, we recall the result of 
\cite{KPX14} concerning the classification of rank one $(\varphi,\Gamma)$-modules. 
Actually, they obtained a similar result for general $K$, but we don't recall it 
since we don't use it.

\begin{defn}\label{2.8}
For 
a 
continuous homomorphism 
$\delta:\mathbb{Q}_p^{\times}\rightarrow \Gamma(X,\mathcal{O}_X)^{\times}$, we define 
$\mathcal{R}_X(\delta)$ to be the rank one $(\varphi,\Gamma)$-module 
$\mathcal{R}_X\cdot \bold{e}_{\delta}$ over $\mathcal{R}_X$ 
with $\varphi(\bold{e}_{\delta})=\delta(p)\bold{e}_{\delta}$ 
and $\gamma(\bold{e}_{\delta})=\delta(\chi(\gamma))\bold{e}_{\delta}$ for $\gamma\in \Gamma$.
\end{defn}
\begin{thm}\label{2.9}$(\mathrm{Theorem }\,\, 6.1.10\,\, \mathrm{ of }$ \cite{KPX14}$)$
Let 
$M$ be a rank 1 $(\varphi,\Gamma)$-module 
over $\mathcal{R}_X$. Then, there exist a continuous homomorphism
$\delta:\mathbb{Q}_p^{\times}\rightarrow \Gamma(X,\mathcal{O}_X)^{\times}$ and an invertible sheaf 
$\mathcal{L}$ on $X$, the pair of which is unique up to isomorphism, such that 
$M\isom \mathcal{R}_X(\delta)\otimes_{\mathcal{O}_X}\mathcal{L}$.
\end{thm}

\begin{notation}\label{2.10}
i) For $\delta,\delta':\mathbb{Q}_p^{\times}\rightarrow \Gamma(X,\mathcal{O}_X)^{\times}$, we fix an isomorphism 
$\mathcal{R}_X(\delta)\otimes\mathcal{R}_X(\delta')\isom \mathcal{R}_X(\delta\delta')$ by $\bold{e}_{\delta}\otimes
\bold{e}_{\delta'}\mapsto \bold{e}_{\delta\delta'}$, and fix an isomorphism $\mathcal{R}_X(\delta)^{\vee}\isom \mathcal{R}_X(\delta^{-1})$ by $\bold{e}_{\delta}^{\vee}\mapsto \bold{e}_{\delta^{-1}}$.

ii) For $k\in \mathbb{Z}$, we define a continuous homomorphism
$x^k:\mathbb{Q}_p^{\times}\rightarrow \Gamma(X,\mathcal{O}_X)^{\times}:y\mapsto 
y^k$.  Define $|x|:\mathbb{Q}_p^{\times}\rightarrow \Gamma(X,\mathcal{O}_X)^{\times}:p\mapsto p^{-1}, a\mapsto 1$ for $a\in \mathbb{Z}_p^{\times}$. 
Then, the homomorphism $x|x|$ corresponds to Tate twist, i.e. we have an isomorphism 
$\mathcal{R}_X(1)\isom \mathcal{R}_X(x|x|)$. When we fix a generator $\zeta\in \mathbb{Z}_p(1)$, 
we identify $\mathcal{R}_X(1)=\mathcal{R}_X(x|x|)$ by $\bold{e}_1\mapsto \bold{e}_{x|x|}$.
\end{notation}

We next recall some cohomology theories concerning $(\varphi,\Gamma)$-modules. 
Denote by $\Delta$ for the largest $p$-power torsion subgroup of 
$\Gamma_{K}$. Fix  $\gamma\in \Gamma_K$ whose image in 
$\Gamma_K/\Delta$ is a topological generator.
For a $\Delta$-module $M$, put $M^{\Delta}=\{m\in M|\sigma(m)=m \,\text{ for all } \,\sigma\in \Delta\}$.

\begin{defn}\label{2.11}
For  a $(\varphi,\Gamma)$-module $M$ over $\mathcal{R}_A(\pi_K)$, we define the complexes 
$C^{\bullet}_{\varphi,\gamma}(M)$ and $C^{\bullet}_{\psi,\gamma}(M)$ of $A$-modules
concentrated in degree $[0,2]$, and define a morphism $\Psi_M$ between them as follows:
\begin{equation}\label{a}
\begin{CD}
C^{\bullet}_{\varphi,\gamma}(M)@.= [M^{\Delta} @> (\gamma-1,\varphi-1)>> 
M^{\Delta}\oplus M^{\Delta} @> (\varphi-1)\oplus(1-\gamma) >> M^{\Delta}] \\
@V \Psi_M VV  @ VV\mathrm{id} V @ VV \mathrm{id}\oplus -\psi V @ VV -\psi V \\
C^{\bullet}_{\psi,\gamma}(M)@.= [M^{\Delta} @> (\gamma-1,\psi-1)>> 
M^{\Delta}\oplus M^{\Delta} @> (\psi-1)\oplus(1-\gamma) >> M^{\Delta}] .
\end{CD}
\end{equation}
The map $\Psi_M$ is quasi-isomorphism by Proposition 2.3.4 of \cite{KPX14}.
\end{defn}
For $i\in \mathbb{Z}_{\geqq 0}$, define $\mathrm{H}^i_{\varphi,\gamma}(M)$ for 
the $i$-th cohomology of 
$C^{\bullet}_{\varphi,\gamma}(M)$, called the $(\varphi,\Gamma)$-cohomology of $M$. We similarly define $\mathrm{H}^i_{\psi,\gamma}(M)$ to be the $i$-th cohomology of $C^{\bullet}_{\psi,\gamma}(M)$, called the 
$(\psi,\Gamma)$-cohomology of $M$. In this article, we freely identify 
$C^{\bullet}_{\varphi,\gamma}(M)$  (respectively $\mathrm{H}^i_{\varphi,\gamma}(M)$) 
with $C^{\bullet}_{\psi,\gamma}(M)$ (respectively $\mathrm{H}^i_{\psi,\gamma}(M)$) via the 
quasi-isomorphism $\Psi_M$.

More generally, for $h=\varphi, \psi$ and any module $N$ with a commuting actions of $h$ 
and $\Gamma$, 
we similarly define the complexes $C^{\bullet}_{h,\gamma}(N)$and denote the resulting cohomology by 
$\mathrm{H}^i_{h,\gamma}(N)$. 
We denote $[x,y]\in \mathrm{H}^1_{h,\gamma}(N)$ 
(respectively $[z]\in \mathrm{H}^2_{h,\gamma}(N)$) for the 
element represented by a one cocycle $(x,y)\in N^{\Delta}\oplus N^{\Delta}$ (respectively 
by $z\in N^{\Delta}$). The functor 
$N\mapsto C^{\bullet}_{h,\gamma}(N)$ from the category of 
topological $A$-modules which are Hausdorff with commuting continuous actions 
of $h,\Gamma_K$ to the category of complexes of $A$-modules is independent of the choice of $\gamma$ up to 
canonical isomorphism, i.e. for another choice $\gamma'\in \Gamma_K$, we have a canonical isomorphism 
\begin{equation}\label{b}
\begin{CD}
C^{\bullet}_{h,\gamma}(N)@.= [N^{\Delta} @> (\gamma-1,h-1)>> 
N^{\Delta}\oplus N^{\Delta} @> (h-1)\oplus(1-\gamma) >> N^{\Delta}] \\
@V \iota_{\gamma,\gamma'} VV  @ VV\mathrm{id} V @ VV \frac{\gamma'-1}{\gamma-1}\oplus \mathrm{id} V @ VV 
\frac{\gamma'-1}{\gamma-1} V \\
C^{\bullet}_{h,\gamma'}(N)@.= [N^{\Delta} @> (\gamma'-1,h-1)>> 
N^{\Delta}\oplus N^{\Delta} @> (h-1)\oplus(1-\gamma') >> N^{\Delta}] .
\end{CD}
\end{equation}

For a commutative ring $R$, let denote $\bold{D}^-(R)$ for the derived category of 
bounded below complexes of $R$-modules. We use the  same notation $C^{\bullet}_{h,\gamma}(N)\in \bold{D}^-(A)$ for the object represented by this  complex.

 Let $V$ be a finite projective $A$-module with a continuous $A$-linear action of $G_K$. 
Let denote $C^{\bullet}_{\mathrm{cont}}(G_K, V)$ for the complex of continuous $G_K$-cochains 
with values in $V$, and denote $\mathrm{H}^i(K, V)$ for the cohomology. 
By Theorem 2.8 of \cite{Po13a}, we have a functorial isomorphism 
$$C^{\bullet}_{\mathrm{cont}}(G_K, V)\isom C^{\bullet}_{\varphi,\gamma}(\bold{D}_{\mathrm{rig}}(V))$$ 
in $\bold{D}^-(A)$ and a functorial $A$-linear isomorphism 
$$\mathrm{H}^i(K, V)\isom \mathrm{H}^i_{\varphi,\gamma}(\bold{D}_{\mathrm{rig}}(V)).$$

\begin{defn}\label{2.12}
For $(\varphi,\Gamma)$-modules $M, N$ over $\mathcal{R}_A(\pi_K)$, we have a 
natural $A$-bilinear cup product morphism 
$$C^{\bullet}_{\varphi,\gamma}(M)\times C^{\bullet}_{\varphi,\gamma}(N)
\rightarrow C^{\bullet}_{\varphi,\gamma}(M\otimes N),$$
see Definition 2.3.11 of \cite{KPX14} for the definition. This induces an $A$-bilinear graded commutative 
cup product pairing 
$$\cup:\mathrm{H}^i_{\varphi,\gamma}(M)
\times \mathrm{H}^{j}_{\varphi,\gamma}(N)\rightarrow \mathrm{H}^{i+j}_{\varphi,\gamma}(M\otimes N).$$ 
For example, this is defined by the formulae
$$x\cup [y]:=[x\otimes y] \text{ for } i=0, j=2,$$
$$[x_1,y_1]\cup [x_2,y_2]:=[x_1\otimes \gamma(y_2)-y_1\otimes \varphi(x_2)] \text{ for } i=j=1.$$
\end{defn}
\begin{rem}\label{2.12.111}
We remark that the definition of the cup product for $\mathrm{H}_{\varphi,\gamma}^1(-)\times
\mathrm{H}^1_{\varphi,\gamma}(-)\rightarrow \mathrm{H}^2_{\varphi,\gamma}(-)$ given in my previous paper \cite{Na14a} is $(-1)$-times of the above definition. The above one seems to be the standard one in the literatures. All the results of \cite{Na14a} holds without any changes 
when we use the above definition except Lemma 2.13 and Lemma 2.14 of \cite{Na14a}, where 
we need to multiply $(-1)$ for the commutative diagrams there to be commutative.
\end{rem}

\begin{defn}\label{dual}
Let denote by $M^*:=M^{\vee}(1)$ for  the Tate dual of $M$. Using the cup product,
the evaluation map $\mathrm{ev}:M^*\otimes M\rightarrow \mathcal{R}_A(\pi_K)(1):f\otimes x\mapsto f(x)$, the comparison isomorphism $\mathrm{H}^{2}(K, A(1))
\isom \mathrm{H}^2_{\varphi,\gamma}(\mathcal{R}_A(\pi_K)(1))$ and 
the Tate's trace map $\mathrm{H}^2(K, A(1))\isom A$, one gets the Tate duality pairings 

\[
\begin{array}{ll}
C^{\bullet}_{\varphi,\gamma}(M^*)\times C^{\bullet}_{\varphi,\gamma}(M)
&\rightarrow C^{\bullet}_{\varphi,\gamma}(M^*\otimes M) \\
&\rightarrow C^{\bullet}_{\varphi,\gamma}(\mathcal{R}_A(\pi_K)(1)))
\rightarrow \mathrm{H}^2_{\varphi,\gamma}(\mathcal{R}_A(\pi_K)(1))[-2] \\
&\isom \mathrm{H}^2(K, A(1))[-2]\isom A[-2]
\end{array}
\]
and 
$$\langle-,-\rangle:\mathrm{H}^i_{\varphi,\gamma}(M^*)\times \mathrm{H}^{2-i}_{\varphi,\gamma}(M)\rightarrow A.$$
\end{defn}

\begin{rem}\label{2.13}
In the appendix,
we explicitly describe the isomorphism 
$\mathrm{H}^2_{\varphi,\gamma}(\mathcal{R}_A(1))\allowbreak\isom \mathrm{H}^2(G_{\mathbb{Q}_p}, A(1))\isom A$ using the residue map; see Proposition \ref{explicitdual}.

\end{rem}


One of the main results of \cite{KPX14} which is crucial to formulate our conjecture is the following.
\begin{thm}\label{2.16}$(\mathrm{Theorem}\,\, 4.4.3, \mathrm{Theorem}\,\,4.4.4 \,\,\mathrm{of}$ \cite{KPX14}$)$
Let $M$ be a $(\varphi,\Gamma)$-module over $\mathcal{R}_A(\pi_K)$.
\begin{itemize}
\item[(1)]$C^{\bullet}_{\varphi,\gamma}(M)\in \bold{D}^{[0,2]}_{\mathrm{perf}}(A)$. In particular, the cohomology groups $\mathrm{H}^i_{\varphi,\gamma}(M)$ are finite $A$-modules.
\item[(2)]Let $A\rightarrow A'$ be a continuous morphism of $\mathbb{Q}_p$-affinoid algebras. Then, the canonical morphism $C^{\bullet}_{\varphi,\gamma}(M)\otimes^{\bold{L}}_{A}A'\rightarrow 
C^{\bullet}_{\varphi,\gamma}(M\hat{\otimes}_AA')$ is a quasi-isomorphism. In particular, if $A'$ is flat over $A$, 
we have $\mathrm{H}^i_{\varphi,\gamma}(M)\otimes_A A'\isom \mathrm{H}^i_{\varphi,\gamma}(M\hat{\otimes}_A A')$.
\item[(3)]$(\mathrm{Euler}$-$\mathrm{Poincar}$\'e$\,\,\mathrm{characteristic}\,\,\mathrm{ formula})$We have $\chi_A(C^{\bullet}_{\varphi,\gamma}(M))=
-[K:\mathbb{Q}_p]\cdot r_M$.
\item[(4)]$(\mathrm{Tate}\,\,\mathrm{ duality})$ The Tate duality pairing defined in Definition \ref{dual} induces a quasi-isomorphism 
$$C^{\bullet}_{\varphi,\gamma}(M)\isom \bold{R}\mathrm{Hom}_A(C^{\bullet}_{\varphi,\gamma}(M^*), A)[-2].$$

\end{itemize}
\end{thm}

\begin{rem}\label{rank}
By the equality of (3) of the above theorem, the rank $r_M\in \mathrm{Map}(\mathrm{Spec}(\mathcal{R}_A(\pi_K)),\mathbb{Z}_{\geqq 0})$ is contained in $\mathrm{Map}(\mathrm{Spec}(A), 
\mathbb{Z}_{\geqq 0})$.

\end{rem}

Let $X$ be a rigid analytic space over $\mathbb{Q}_p$ and 
let $M$ be a $(\varphi,\Gamma)$-module over 
$\mathcal{R}_X(\pi_K)$. By (1) and (2) of the above theorem, the correspondence $U\mapsto \mathrm{H}^i_{\varphi,\gamma}(M|_U)$ for each affinoid open $U$ in $X$ 
defines a coherent $\mathcal{O}_X$-module for each $i\in[0,2]$, for which we also 
denote by $\mathrm{H}^i_{\varphi,\gamma}(M)$.

\subsection{Bloch-Kato exponential for $(\varphi,\Gamma)$-modules}

In \S2 of \cite{Na14a}, we developed the theory of 
the Bloch-Kato exponential purely in terms of 
$(\varphi,\Gamma)$-modules over Robba ring. This subsection is a complement of \S2
of \cite{Na14a}; we recall the definitions of the
Bloch-Kato exponential and the dual exponential for $(\varphi,\Gamma)$-modules 
over the ( relative ) Robba ring, and we slightly generalize the results of \cite{Na14a}, which are needed to formulate our conjecture in the next section. 

Define $n(K)\geqq 1$ to be the minimal integer $n$ such that 
$1/p^{n-1}\leqq \tilde{e}_KC(K)$, and put
$\mathcal{R}_A^{(n)}(\pi_K)=\mathcal{R}_A^{1/(p^{n-1}\tilde{e}_K)}(\pi_K)$ 
for $n\geqq n(K)$. 
For $n\geqq n(K)$, one has a
 $\Gamma_K$-equivariant $A$-algebra homomorphism 
 $$\iota_n:\mathcal{R}_A^{(n)}(\pi_K)\rightarrow (K_n\otimes_{\mathbb{Q}_p}A)[[t]]$$ such that 
 $$\iota_n(\pi)=\zeta_{p^n}\cdot\mathrm{exp}\left(\frac{t}{p^n}\right)-1\,\,\text{ and } \,\,\,
 \iota_n(a)=\varphi^{-n}(a)\,\,\,(a\in F').$$ 
 For $n\geqq n(K)$, we have the following  commutative diagrams:
 \begin{equation*}
 \begin{CD}
 \mathcal{R}^{(n)}_A(\pi_K) @> \iota_n >> (K_n\otimes_{\mathbb{Q}_p}A)[[t]]@.\,\,\,\,\, \mathcal{R}^{(n+1)}_A(\pi_K)@> \iota_{n+1} >> (K_{n+1}\otimes_{\mathbb{Q}_p}A)[[t]] \\
 @ VV \varphi V @VV \mathrm{can} V @VV \psi V @ VV \frac{1}{p}\cdot\mathrm{Tr}_{K_{n+1}/K_n} V\\
  \mathcal{R}^{(n+1)}_A(\pi_K) @> \iota_{n+1} >> (K_{n+1}\otimes_{\mathbb{Q}_p}A)[[t]],@. \,\,\,\,\,\mathcal{R}^{(n)}_A(\pi_K)@> \iota_{n} >> (K_n\otimes_{\mathbb{Q}_p}A)[[t]] ,
 \end{CD}
 \end{equation*}
 where $\mathrm{can}$ is the canonical injection and $\frac{1}{p}\cdot\mathrm{Tr}_{K_{n+1}/K_n}$ is 
 defined by $$\sum_{k\geqq 0}a_kt^k\mapsto 
 \sum_{k\geqq 0}\frac{1}{p}\cdot\mathrm{Tr}_{K_{n+1}/K_n}(a_k)t^k.$$

 Let $M$ be a $(\varphi,\Gamma)$-module over $\mathcal{R}_A(\pi_K)$ obtained as a base change of a $(\varphi,\Gamma)$-module $M^{r_0}$ over $\mathcal{R}_A^{r_0}(\pi_K)$ for some $0<r_0\leqq c(K)$. 
 Define $n(M)\in \mathbb{Z}_{\geqq  n(K)}$ to be  the minimal integer such that $1/p^{n-1}\leqq \tilde{e}_Kr_0$. 
 Put $M^{(n)}=M^{1/(p^{n-1}\tilde{e}_K)}$ for $n\geqq n(M)$, then 
 $\varphi$ and $\psi$ induce $\varphi:M^{(n)}\rightarrow M^{(n+1)}$ and 
 $\psi:M^{(n+1)}\rightarrow M^{(n)}$ respectively. Define 
 $$ \bold{D}^{+ }_{\mathrm{dif},n}(M)=M^{(n)}\otimes_{\mathcal{R}_A^{(n)}(\pi_K), \iota_n}(K_n\otimes_{\mathbb{Q}_p}A)[[t]] \,\,( \text{respectively}\,\,
 \bold{D}_{\mathrm{dif},n}(M)=\bold{D}^+_{\mathrm{dif},n}(M)[1/t]),$$
 which is a finite 
 projective $(K_n\otimes_{\mathbb{Q}_p}A)[[t]]$-module (respectively
 $(K_n\otimes_{\mathbb{Q}_p}A)((t))$-module )
 with a semilinear action of $\Gamma_K$. We also denote $\iota_n:M^{(n)}\rightarrow \bold{D}^{+ }_{\mathrm{dif},n}(M)$ for the map defined by $x\mapsto x\otimes 1$.
 
 Using the base change of the Frobenius structure $\varphi^*M^{(n)}\isom M^{(n+1)}$ by the map $\iota_{n+1}$,
 we obtain a $\Gamma_K$-equivariant $(K_{n+1}\otimes_{\mathbb{Q}_p}A)[[t]]$-linear isomorphism 
 \begin{multline*}
 \bold{D}^{+}_{\mathrm{dif},n}(M)\otimes_{(K_n\otimes_{\mathbb{Q}_p}A)[[t]]}
 (K_{n+1}\otimes_{\mathbb{Q}_p}A)[[t]]\isom \varphi^*(M^{(n)})\otimes_{\mathcal{R}_A^{(n+1)}(\pi_K),\iota_{n+1}}(K_{n+1}\otimes_{\mathbb{Q}_p}A)[[t]]\\
 \isom M^{(n+1)}\otimes_{\mathcal{R}_A^{(n+1)}(\pi_K),\iota_{n+1}}(K_{n+1}\otimes_{\mathbb{Q}_p}A)[[t]]
 =\bold{D}^+_{\mathrm{dif},n+1}(M).
 \end{multline*}
 Using this isomorphism, we obtain $\Gamma_K$-equivariant $(K_n\otimes_{\mathbb{Q}_p}A)[[t]]$-linear morphisms 
 $$\mathrm{can}:\bold{D}^+_{\mathrm{dif},n}(M)\xrightarrow{x\mapsto x\otimes 1}\bold{D}^+_{\mathrm{dif},n}(M)
 \otimes_{(K_n\otimes_{\mathbb{Q}_p}A)[[t]]}(K_{n+1}\otimes_{\mathbb{Q}_p}A)[[t]]
 \isom \bold{D}^+_{\mathrm{dif},n+1}(M)$$ and 

\begin{multline*}
\frac{1}{p}\cdot\mathrm{Tr}_{K_{n+1}/K_n}:\bold{D}^+_{\mathrm{dif},n+1}(M)\isom 
 \bold{D}^+_{\mathrm{dif},n}(M)\otimes_{(K_n\otimes_{\mathbb{Q}_p}A)[[t]]}(K_{n+1}\otimes_{\mathbb{Q}_p}A)[[t]]\\
 \xrightarrow{x\otimes f \mapsto 
 \frac{1}{p}\cdot\mathrm{Tr}_{K_{n+1}/K_n}(f)x}\bold{D}^+_{\mathrm{dif},n}(M).
 \end{multline*}
 
 These naturally induce $\mathrm{can}:\bold{D}_{\mathrm{dif},n}(M)\rightarrow 
 \bold{D}_{\mathrm{dif},n+1}(M)$ and $\frac{1}{p}\cdot\mathrm{Tr}_{K_{n+1}/K_n}:\bold{D}_{\mathrm{dif},n+1}(M)
 \rightarrow \bold{D}_{\mathrm{dif},n}(M)$, and we have the following commutative diagrams:
 \begin{equation*}
 \begin{CD}
 M^{(n)}@> \iota_n >> \bold{D}^+_{\mathrm{dif},n}(M)@.\,\,\,\,\, M^{(n+1)}@> \iota_{n+1} >> \bold{D}^+_{\mathrm{dif},n+1}(M) \\
 @ VV \varphi V @VV \mathrm{can} V @VV \psi V @ VV \frac{1}{p}\cdot\mathrm{Tr}_{K_{n+1}/K_n} V\\
 M^{(n+1)} @> \iota_{n+1} >> \bold{D}^+_{\mathrm{dif},n+1}(M),@.\,\,\,\,\, M^{(n)}@> \iota_{n} >> \bold{D}^+_{\mathrm{dif},n}(M).
 \end{CD}
 \end{equation*}
 Put $\bold{D}^{(+)}_{\mathrm{dif}}(M):=\varinjlim_{n\geqq n(M)}\bold{D}^{(+)}_{\mathrm{dif}, n}(M)$, where 
 the transition map is $\mathrm{can}:\bold{D}^{(+)}_{\mathrm{dif},n}(M)\rightarrow\bold{D}^{(+)}_{\mathrm{dif},n+1}(M)$. Then, we have $\bold{D}^{(+)}_{\mathrm{dif}}(M)=\bold{D}^{(+)}_{\mathrm{dif},n}(M)\otimes_{(K_n\otimes_{\mathbb{Q}_p}A)[[t]]}
 (K_{\infty}\otimes_{\mathbb{Q}_p}A)[[t]]$ for any $n\geqq n(M)$,  where we define 
 $(K_{\infty}\otimes_{\mathbb{Q}_p}A)[[t]]=\cup_{m\geqq 1}(K_m\otimes_{\mathbb{Q}_p}A)[[t]]$.

For an $A[\Gamma_K]$-module $N$,  we define a complex of $A$-modules 
concentrated in degree 
$[0,1]$
$$C^{\bullet}_{\gamma}(N)=[N^{\Delta}\xrightarrow{\gamma-1}N^{\Delta}]$$
and denote $\mathrm{H}^i_{\gamma}(N)$ for the cohomology of $C^{\bullet}_{\gamma}(N)$. If 
$N$ is a topological Hausdorff $A$-module with a continuous action of $\Gamma_K$, the complex $C^{\bullet}_{\gamma}(N)$ is also independent of the choice of $\gamma$ up to 
canonical isomorphism.  

Let $M$ be a $(\varphi,\Gamma)$-modules over $\mathcal{R}_A(\pi_K)$. For $n\geqq n(M)$ and $M_0=M, M[1/t]$, we define a complex $\widetilde{C}^{\bullet}_{\varphi,\gamma}(M^{(n)}_0)$ concentrated in degree $[0.2]$ by
$$\widetilde{C}^{\bullet}_{\varphi,\gamma}(M_0^{(n)}):=[M_0^{(n),\Delta}\xrightarrow{(\gamma-1)\oplus
(\varphi-1)}M_0^{(n),\Delta}\oplus M_0^{(n+1),\Delta}\xrightarrow{(\varphi-1)\oplus(1-\gamma)}
M_0^{(n+1),\Delta}].$$
Of course, we have 
$\varinjlim_{n}\widetilde{C}^{\bullet}_{\varphi,\gamma}(M_0^{(n)})= C^{\bullet}_{\varphi,\gamma}(M_0)$, where the transition map is the natural one induced by 
the canonical inclusion $M_0^{(n)}\hookrightarrow M_0^{(n+1)}$. We define 
another complex 
$$C^{(\varphi), \bullet}_{\varphi,\gamma}(M_0):=\varinjlim_{n,\varphi}\widetilde{C}^{\bullet}_{\varphi,\gamma}(M_0^{(n)}),$$ where the transition map is the natural one
induced by 
$\varphi:M_0^{(n)}\rightarrow M_0^{(n+1)}$. 
We similarly define 
$C^{(\varphi), \bullet}_{\gamma}(M_0):=\varinjlim_{n,\varphi}C^{\bullet}_{\gamma}
(M_0^{(n)})$ and 
denote $\mathrm{H}^{ (\varphi), i }_{\varphi,\gamma}(M_0) $ (respectively
$\mathrm{H}^{(\varphi), i}_{\gamma}(M_0)$) for the cohomology of $C^{(\varphi),\bullet}_{\varphi,\gamma}(M_0)$ (respectively $C^{(\varphi),\bullet}_{\gamma}(M_0)$).  For $n\geqq n(M)$, we equip $C^{\bullet}_{\gamma}(M_0^{(n)})$ with a structure of a complex of $F$-vector spaces by 
$ax:=\varphi^n(a)x$ for $a\in F, x\in C^{\bullet}_{\gamma}(M_0^{(n)})$. Then, $C^{(\varphi),\bullet}_{\gamma}(M_0)$
(respectively $\mathrm{H}^{(\varphi),i}_{\gamma}(M_0)$) is
 also naturally equipped with a structure of a complex of $F$-vector spaces (respectively a $F$-vector space).

By the compatibility of $\varphi:M^{(n)}\hookrightarrow M^{(n+1)}$ and
$\mathrm{can}:\bold{D}^+_{\mathrm{dif},n}(M)\hookrightarrow \bold{D}^+_{\mathrm{dif},n+1}(M)$ 
with respect to the map $\iota_n:M^{(n)}\rightarrow \bold{D}^+_{\mathrm{dif},n}(M)$, the map 
$\iota_n$ induces canonical maps
$$\iota:C_{\gamma}^{(\varphi),\bullet}(M)\rightarrow C^{\bullet}_{\gamma}
(\bold{D}^+_{\mathrm{dif}}(M))\text{  and }
\iota:C^{(\varphi),\bullet}_{\gamma}(M[1/t])\rightarrow C^{\bullet}_{\gamma}(\bold{D}_{\mathrm{dif}}(M))$$ which are $F\otimes_{\mathbb{Q}_p}A$-linear.

\begin{lemma}\label{2.18}
For $n\geqq n(M)$, the natural maps
$$C^{\bullet}_{\gamma}(\bold{D}^{(+)}_{\mathrm{dif},n}(M))\rightarrow 
C^{\bullet}_{\gamma}(\bold{D}^{(+)}_{\mathrm{dif},n+1}(M)),\,\, 
C^{\bullet}_{\gamma}(M^{(n)}_0)
\rightarrow C^{\bullet}_{\gamma}(M_0^{(n+1)})$$ and 
$$\widetilde{C}^{\bullet}_{\varphi,\gamma}(M_0^{(n)})\rightarrow 
\widetilde{C}^{\bullet}_{\varphi,\gamma}(M_0^{(n+1)})$$ for 
$M_0=M, M[1/t]$ which are induced by $\varphi$ are  quasi-isomorphism. 
Similarly, the maps $$C^{\bullet}_{\gamma}(\bold{D}^{(+)}_{\mathrm{dif},n}(M))\rightarrow 
C^{\bullet}_{\gamma}(\bold{D}^{(+)}_{\mathrm{dif}}(M)),\,\,
C^{\bullet}_{\gamma}(M^{(n)}_0)
\rightarrow C^{(\varphi),\bullet}_{\gamma}(M_0)$$ and 
$$\widetilde{C}^{\bullet}_{\varphi,\gamma}(M_0^{(n)})\rightarrow 
C^{(\varphi),\bullet}_{\varphi,\gamma}(M_0)$$ for 
$M_0=M, M[1/t]$ are quasi-isomorphism.
\end{lemma}
\begin{proof}
The latter statement is trivial if we can prove the first statement.
Let's prove the first statement.
 We first note that 
 $\gamma-1:(M^{(n)}_0)^{\psi=0}\rightarrow 
(M^{(n)}_0)^{\psi=0}$ is isomorphism for $n\geqq n(M)+1$ 
by Theorem  3.1.1 of \cite{KPX14}  ( precisely, this fact for $M_0=M[1/t]$ follows from the proof of this theorem).
Taking the base change of this isomorphism 
by the map $\iota_n:\mathcal{R}_A^{(n)}(\pi_K)\rightarrow (K_n\otimes_{\mathbb{Q}_p}A)[[t]]$, we also have 
that $\gamma-1:(\bold{D}^{(+)}_{\mathrm{dif},n}(M))
^{\frac{1}{p}\cdot\mathrm{Tr}_{K_n/K_{n-1}}=0}\rightarrow (\bold{D}^{(+)}_{\mathrm{dif},n}(M))^{\frac{1}{p}\cdot
\mathrm{Tr}_{K_n/K_{n-1}}=0}$ is 
 isomorphism for  $n\geqq n(M)+1$. 
Using these facts, we prove the lemma as follows. 
Here, we only prove that the map $C^{\bullet}_{\gamma}(M^{(n)}_0)
\rightarrow C^{\bullet}_{\gamma}(M_0^{(n+1)})$ induced by $\varphi:M_0^{(n)}\rightarrow M_0^{(n+1)}$  is 
quasi-isomorphism for $n\geqq n(M)$ since the other cases can be proved in the same way.  Since 
we have a $\Gamma_K$-equivariant decomposition 
$M_0^{(n+1)}=\varphi(M_0^{(n)})\oplus (M_0^{(n+1)})^{\psi=0}$, 
 we obtain a decomposition
$C^{\bullet}_{\gamma}(M_0^{(n+1)})=\varphi(C^{\bullet}_{\gamma}(M_0^{(n)}))
\oplus C^{\bullet}_{\gamma}((M_0^{(n+1)})^{\psi=0})$. Since the complex $C^{\bullet}_{\gamma}((M_0^{(n+1)})^{\psi=0})$ is acyclic by the above remark and 
$\varphi:M_0^{(n)}\rightarrow M_0^{(n+1)}$ is injection, the map 
$\varphi:C^{\bullet}_{\gamma}(M_0^{(n)})\rightarrow C^{\bullet}_{\gamma}(M_0
^{(n+1)})$ is quasi-isomorphism.

\end{proof}

For another canonical map $C^{\bullet}_{\gamma}(M^{(n)}_0)\rightarrow 
C^{\bullet}_{\gamma}(M_0)$ which is induced by the canonical inclusion 
$M^{(n)}\hookrightarrow M$, we can show the following lemma.

\begin{lemma}\label{2.19}
For $n\geqq n(M)$ and $M_0=M, M[1/t]$, the inclusion 
$$\mathrm{H}^0_{\gamma}(M^{(n)}_0)\hookrightarrow \mathrm{H}^0_{\gamma}(M_0)$$  induced by 
the canonical inclusion $M^{(n)}_0\hookrightarrow M_0$ is 
isomorphism.
\end{lemma}
\begin{proof}
It suffices to show that 
$\mathrm{H}^0_{\gamma}(M^{(n)}_0)\hookrightarrow \mathrm{H}^0_{\gamma}(M^{(n+1)}_0)$ is isomorphism for each $n\geqq n(M)$. We first prove this claim when 
$A$ is a finite $\mathbb{Q}_p$-algebra. In this case, we may assume that $A=\mathbb{Q}_p$. 
Since we have an inclusion 
$\iota_n:\mathrm{H}^0_{\gamma}(M_0^{(n)})
\hookrightarrow \mathrm{H}^0_{\gamma}(\bold{D}_{\mathrm{dif}}(M))$ and 
the latter is a finite dimensional $\mathbb{Q}_p$-vector space, 
$\mathrm{H}^0_{\gamma}(M_0^{(n)})$ is also finite dimensional. 
Since $\varphi:C^{\bullet}_{\gamma}(M^{(n)}_0)\rightarrow C^{\bullet}_{\gamma}(M_0^{(n+1)})$ is quasi-isomorphism for $n\geqq n(M)$ by the 
above lemma, we obtain an isomorphism $\varphi:\mathrm{H}^0_{\gamma}(M_0^{(n)})\isom 
\mathrm{H}^0_{\gamma}(M_0^{(n+1)})$. In particular, the dimension of $\mathrm{H}^0_{\gamma}(M_0^{(n)})$ is independent of $n\geqq n(M)$. Hence, the canonical inclusion 
$\mathrm{H}^0_{\gamma}(M_0^{(n)})\hookrightarrow \mathrm{H}^0_{\gamma}(M_0^{(n+1)})$ is isomorphism. 

We next prove the claim for general $A$. By Lemma 6.4 of \cite{KL10}, there exists a strict inclusion 
$A\hookrightarrow \prod_{i=1}^kA_i$ of topological rings, in which 
each $A_i$ is a finite algebra over a completely discretely valued field. 
If we similarly define the rings $\mathcal{R}^{(n)}_{A_i}(\pi_K)$, $\mathcal{R}_{A_i}(\pi_K)$, we can generalize the notions concerning $(\varphi,\Gamma)$-modules for $\mathcal{R}_{A_i}(\pi_K)$. 
In particular, the above claim holds for $M_{0,i}:=M_0\hat{\otimes}_A A_i$ for each $i$. 
Consider the following canonical diagram of exact sequences,
\begin{equation*}
\begin{CD}
0@>>> M_0^{(n)}@ >>> M_0^{(n+1)}@ >>>  M_0^{(n+1)}/M_0^{(n)}\rightarrow0\\
@.@VVV @ VVV @ VVV\\
0@>>> \prod_{i=1}^k M_{0,i}^{(n)}@>>> 
\prod_{i=1}^kM_{0,i}^{(n+1)} @>>> 
\prod_{i=1}^kM_{0,i}^{(n+1)}/M_{0,i}^{(n)}\rightarrow0.
\end{CD}
\end{equation*}
If we can show that the right vertical arrow is injection, then the claim for $A$ follows from 
the claim for each  $A_i$ by a simple diagram chase. To show that the right vertical arrow is 
injection, we may assume that $M=\mathcal{R}_A(\pi_K)$ since $M^{(n)}$ is finite 
projective over $\mathcal{R}_A^{(n)}(\pi_K)$ for each $n$. Then, the natural map 
$\mathcal{R}^{(n+1)}_A(\pi_K)[1/t]/\mathcal{R}_A^{(n)}(\pi_K)[1/t]\rightarrow 
\prod_{i=1}^k\mathcal{R}^{(n+1)}_{A_i}(\pi_K)[1/t]/\mathcal{R}_{A_i}^{(n)}(\pi_K)[1/t]$ is 
injection since the inclusion $A\hookrightarrow \prod_{i=1}^kA_i$ is strict, which 
proves the claim for general $A$, hence proves the lemma.

\end{proof}

\begin{rem}\label{2.20}
The author doesn't know whether the natural map 
$\mathrm{H}^1_{\gamma}(M_0^{(n)})\rightarrow \mathrm{H}^1_{\gamma}
(M_0)$ induced by the canonical inclusion $M_0^{(n)}\hookrightarrow M_0$ is isomorphism 
or not. 

\end{rem}

For the $(\varphi,\Gamma)$-cohomology, we can prove the following lemma.

\begin{lemma}\label{2.21}
\begin{itemize}
\item[(1)]
For  $n\geqq n(M)$ and for $M_0=M, M[1/t]$, the map 
$$\widetilde{C}^{\bullet}_{\varphi,\gamma}(M_0^{(n)})\allowbreak\rightarrow 
C^{\bullet}_{\varphi,\gamma}(M_0)$$ induced by the canonical inclusion 
$M_0^{(n)}\hookrightarrow M_0$ is quasi-isomorphism. 
\item[(2)]In $\bold{D}^{-}(A)$, the isomorphism 
$$C^{\bullet}_{\varphi,\gamma}(M_0)\isom C^{(\varphi),\bullet}_{\varphi,\gamma}(M_0)$$ which is obtained as the composition of the inverse of the isomorphism $\widetilde{C}^{\bullet}_{\varphi,\gamma}
(M_0^{(n)})\isom C^{\bullet}_{\varphi,\gamma}(M_0)$ in (1) with the isomorphism 
$\widetilde{C}^{\bullet}_{\varphi,\gamma}(M_0^{(n)})\isom C^{(\varphi),\bullet}_{\varphi,\gamma}(M_0)$ in Lemma \ref{2.18} is independent of the choice of $n\geqq n(M)$.
\end{itemize}

\end{lemma}

\begin{proof}For $n\geqq n(M)$, we define a map $f_{\bullet}:\widetilde{C}^{\bullet}_{\varphi,\gamma}(M_0^{(n)})\rightarrow 
\widetilde{C}^{\bullet}_{\varphi,\gamma}(M_0^{(n+1)})[+1]$ by 
$f_1:M_0^{(n),\Delta}\oplus M_0^{(n+1),\Delta}\rightarrow M_0^{(n+1),\Delta}:(x,y)\mapsto 
y,$
and $f_2:M_0^{(n+1),\Delta}\rightarrow M_0^{(n+1),\Delta}\oplus M_0^{(n+2),\Delta}:
x\mapsto (x,0).$ This gives a homotopy between $\varphi:\widetilde{C}^{\bullet}_{\varphi,\gamma}(M_0^{(n)})\rightarrow \widetilde{C}^{\bullet}_{\varphi,\gamma}(M_0^{(n+1)})$ and 
$\mathrm{can}:\widetilde{C}^{\bullet}_{\varphi,\gamma}(M_0^{(n)})\allowbreak\rightarrow \widetilde{C}^{\bullet}_{\varphi,\gamma}(M_0^{(n+1)})$ induced by the canonical inclusion $M_0^{(n)}\hookrightarrow M_0^{(n+1)}$. Hence, $\mathrm{can}:\widetilde{C}^{\bullet}_{\varphi,\gamma}(M_0^{(n)})\rightarrow \widetilde{C}^{\bullet}_{\varphi,\gamma}(M_0^{(n+1)})$ is also isomorphism by Lemma \ref{2.18}, and 
the map $\widetilde{C}^{\bullet}_{\varphi,\gamma}(M_0^{(n)})\rightarrow 
C^{\bullet}_{\varphi,\gamma}(M_0)$ is also isomorphism by taking the limit, which proves (1). 

In a similar way, we can 
show that the map 
$\mathrm{can}:C^{(\varphi),\bullet}_{\varphi,\gamma}(M_0)\rightarrow C^{(\varphi),\bullet}_{\varphi,\gamma}(M_0)$ induced by the canonical inclusions $\mathrm{can}:M_0^{(n)}\hookrightarrow M_0^{(n+1)}$ for any $n\geqq n(M)$ is homotopic to 
the identity map. Hence, we obtain the following commutative diagram in $\bold{D}^-(A)$ for 
any $n\geqq n(M)$,
\begin{equation*}
\begin{CD}
\widetilde{C}^{\bullet}_{\varphi,\gamma}(M_0^{(n)})@>>> C^{(\varphi),\bullet}_{\varphi,\gamma}(M_0) \\
@VV\mathrm{can} V @ VV\mathrm{id} V \\
\widetilde{C}^{\bullet}_{\varphi,\gamma}(M_0^{(n+1)})@>>> C^{(\varphi),\bullet}_{\varphi,\gamma}(M_0),
\end{CD}
\end{equation*}
from which we obtain the second statement in the lemma.

\end{proof}

We define a morphism 
$$f:C^{\bullet}_{\varphi,\gamma}(M_0)\rightarrow 
C^{(\varphi),\bullet}_{\gamma}(M_0)$$ in $\bold{D}^-(A)$ as the composition of 
the isomorphism $C^{\bullet}_{\varphi,\gamma}(M_0)\isom 
C^{(\varphi),\bullet}_{\varphi,\gamma}(M_0)$ in the above lemma (2)
with the
map $C^{(\varphi),\bullet}_{\varphi,\gamma}(M_0)
\rightarrow C^{(\varphi),\bullet}_{\gamma}(M_0)$ which 
is induced by 
\begin{equation*}
\begin{CD}
\widetilde{C}^{\bullet}_{\varphi,\gamma}(M_0^{(n)})@.=[M_0^{(n),\Delta}@> (\gamma-1)\oplus(\varphi-1) >> M_0^{(n),\Delta}\oplus M_0^{(n+1),\Delta}
@> (\varphi-1)\oplus(1-\gamma) >> M_0^{(n+1),\Delta}]\\
@VVV@VV\mathrm{id} V @ VV (x,y)\mapsto x V @. \\
C^{\bullet}_{\gamma}(M_0^{(n)})@.=[M_0^{(n),\Delta}@> \gamma-1
 >> M_0^{(n),\Delta}].@.
\end{CD}
\end{equation*}
We define 
$$g:C^{\bullet}_{\varphi,\gamma}(M)
\xrightarrow{f}
 C^{(\varphi),\bullet}_{\gamma}(M)\xrightarrow{\iota} 
C^{\bullet}_{\gamma}(\bold{D}^+_{\mathrm{dif}}(M)).$$ 
We denote 
$$\mathrm{can}:C^{(\varphi),\bullet}_{\gamma}(M_0)
\rightarrow C^{(\varphi),\bullet}_{\gamma}(M_0)$$ for the map induced by the 
canonical inclusion
$\mathrm{can}:M_0^{(n)}\rightarrow M_0^{(n+1)}$ for each $n\geqq n(M)$. Under these notations, we prove the following 
proposition, which is a modified version of Theorem 2.8 of \cite{Na14a}.

\begin{prop}\label{2.22}
We have functorial two distinguished triangles (horizontal ones) 
 and the map from the above to the below:
\begin{equation}\label{3.5}
\begin{CD}
C^{\bullet}_{\varphi,\gamma}(M)
@> d_1>> 
C^{\bullet}_{\varphi,\gamma}(M[1/t])\oplus 
C^{\bullet}_{\gamma}(\bold{D}^+_{\mathrm{dif}}(M))
@> d_2 >> 
C^{\bullet}_{\gamma}(\bold{D}_{\mathrm{dif}}(M))\xrightarrow{[+1]} \\
@VV\mathrm{id} V @ VV f\oplus \mathrm{id} V @ VV x\mapsto (0,x)V\\
C^{\bullet}_{\varphi,\gamma}(M)@> d_3>> 
C^{(\varphi),\bullet}_{\gamma}(M[1/t])\oplus 
C^{\bullet}_{\gamma}(\bold{D}^+_{\mathrm{dif}}(M))
@> d_4 >> 
C^{(\varphi),\bullet}_{\gamma}(M[1/t])\oplus
C^{\bullet}_{\gamma}(\bold{D}_{\mathrm{dif}}(M))\xrightarrow{[+1]}
\end{CD}
\end{equation}
with $$d_1(x)=(x,g(x)),\,\,\, d_2(x,y)=g(x)-y,$$
and 
 $$d_3(x)=(f(x), g(x)),\,\,\,\,
d_4(x,y)=((\mathrm{can}-1)x, g(x)-y).$$ 

\end{prop}
\begin{rem}
In \S2 of \cite{Na14a}, we essentially proved that the above horizontal line in this proposition 
is distinguished triangle. For the application to the local $\varepsilon$-conjecture, we also 
need the below triangle, which involves $\bold{D}^K_{\mathrm{cris}}(M):=\mathrm{H}^0_{\gamma}
(M[1/t])$.

\end{rem}

\begin{proof}( of Proposition \ref{2.22})
We first show that the above horizontal line in the proposition is a distinguished 
triangle. Actually, this is the content of Theorem 2.8 of \cite{Na14a}, but 
we briefly recall the proof since we also use it to prove that the below line is 
a distinguished triangle. In this proof, we assume $\Delta=\{1\}$ for simplicity; the general case just follows by
taking the $\Delta$-fixed parts.

For $n\geqq n(M)$, we have the following exact sequence of $A$-modules:
\begin{equation}\label{4e}
0\rightarrow M^{(n)}\xrightarrow{c_1}M^{(n)}[1/t]\oplus \prod_{m\geqq n}
\bold{D}^+_{\mathrm{dif},m}(M)\xrightarrow{c_2}\cup_{k\geqq 0}\prod_{m\geqq n}
\frac{1}{t^k}\bold{D}^+_{\mathrm{dif},m}(M)\rightarrow 0
\end{equation}
with 
$$c_1(x)=(x, (\iota_m(x))_{m\geqq n})\,\,\text{  and } \,\,c_2(x, (y_m)_{m\geqq n})
=(\iota_m(x)-y_m)_{m\geqq n}$$ by Lemma 2.9 of \cite{Na14a} (
precisely, we proved it when $A$ is a finite $\mathbb{Q}_p$-algebra, but we can prove 
it for general $A$ in the same way). 
For $n\geqq n(M)$ and $k\geqq 0$, we define a complex 
$\widetilde{C}^{\bullet}_{\varphi,\gamma}(\frac{1}{t^k}\bold{D}^{+}_{\mathrm{dif}, n}(M))$ concentrated in degree in $[0,2]$ by
\begin{equation}\label{5e}
[\prod_{m\geqq n}\frac{1}{t^k}\bold{D}^+_{\mathrm{dif},m}(M)\xrightarrow{b_0}
\prod_{m\geqq n}\frac{1}{t^k}\bold{D}^+_{\mathrm{dif},m}(M)
\oplus \prod_{m\geqq n+1}\frac{1}{t^k}\bold{D}^+_{\mathrm{dif},m}(M)
\xrightarrow{b_1}\prod_{m\geqq n+1}\frac{1}{t^k}\bold{D}^+_{\mathrm{dif},m}(M)]
\end{equation}
with 
$$b_0((x_m)_{m\geqq n})=(((\gamma-1)x_m)_{m\geqq n}, (x_{m-1}-x_m)_{m\geqq 
n+1})$$
and $$b_1((x_m)_{m\geqq n}, (y_m)_{m\geqq n+1})
=((x_{m-1}-x_m)-(\gamma-1)y_m)_{m\geqq n+1}.$$
Put $\widetilde{C}^{\bullet}_{\varphi,\gamma}(\bold{D}_{\mathrm{dif},n}(M))
=\cup_{k\geqq 0}\widetilde{C}^{\bullet}_{\varphi,\gamma}(\frac{1}{t^k}\bold{D}^+_{\mathrm{dif},n}
(M))$. 
By the above exact sequence (\ref{4e}), we obtain the following exact sequence 
of complexes of $A$-modules:
\begin{equation}\label{6e}
0\rightarrow \widetilde{C}^{\bullet}_{\varphi,\gamma}(M^{(n)})
\rightarrow \widetilde{C}^{\bullet}_{\varphi,\gamma}(M^{(n)}[1/t])
\oplus \widetilde{C}^{\bullet}_{\varphi,\gamma}(\bold{D}^+_{\mathrm{dif},n}(M))
\rightarrow \widetilde{C}^{\bullet}_{\varphi,\gamma}(\bold{D}_{\mathrm{dif},n}(M))
\rightarrow 0.
\end{equation}

Moreover, the map $C^{\bullet}_{\gamma}(\bold{D}^+_{\mathrm{dif},n}(M))
\rightarrow \widetilde{C}^{\bullet}_{\varphi,\gamma}(\bold{D}^{+}_{\mathrm{dif},n}(M))$ which is defined by 
\begin{equation}\label{7e}
\begin{CD}
\bold{D}^+_{\mathrm{dif},n}(M)@> \gamma-1 >> \bold{D}^+_{\mathrm{dif},n}(M) @. \\
@VV x\mapsto (x)_{m\geqq n} V @ VV x\mapsto ((x)_{m\geqq n},0) V @. \\
\prod_{m\geqq n}\bold{D}^+_{\mathrm{dif},m}(M)@>>>
\prod_{m\geqq n}\bold{D}^+_{\mathrm{dif},m}(M)
\oplus \prod_{m\geqq n+1}\bold{D}^+_{\mathrm{dif},m}(M)
@>>>\prod_{m\geqq n+1}\bold{D}^+_{\mathrm{dif},m}(M)
\end{CD}
\end{equation} 
and the similar map $C^{\bullet}_{\gamma}(\bold{D}_{\mathrm{dif},n}(M))
\rightarrow \widetilde{C}^{\bullet}_{\varphi,\gamma}(\bold{D}_{\mathrm{dif},n}(M))$
are easily seen to be quasi-isomorphism since we have the following exact sequence 
\begin{equation}\label{8e}
0\rightarrow \bold{D}^{(+)}_{\mathrm{dif},n}(M)\xrightarrow{x\mapsto (x)_{m\geqq n}}
\prod_{m\geqq n}\bold{D}^{(+)}_{\mathrm{dif},m}(M)\xrightarrow{(x_m)_{m\geqq n}\mapsto (x_{m-1}-x_m)_{m\geqq n+1}} \prod_{m\geqq n+1}\bold{D}^{(+)}_{\mathrm{dif},m}(M)\rightarrow 0.
\end{equation}
Put $\widetilde{C}^{\bullet}_{\varphi,\gamma}(\bold{D}^{(+)}_{\mathrm{dif}}(M))
:=\varinjlim_{n,a^{\bullet}}\widetilde{C}^{\bullet}_{\varphi,\gamma}(\bold{D}^{(+)}_{\mathrm{dif},n}(M))$
where the transition map $a^{\bullet}:\widetilde{C}^{\bullet}_{\varphi,\gamma}(\bold{D}^{(+)}_{\mathrm{dif},n}\allowbreak(M))\rightarrow \widetilde{C}^{\bullet}_{\varphi,\gamma}(\bold{D}^{(+)}_{\mathrm{dif},n+1}(M))$ is defined by 
$$a^0((x_m)_{m\geqq n})=(x_m)_{m\geqq n+1}, \,\,\,\,a^1((x_m)_{m\geqq n}, (y_m)_{m\geqq n+1})
=((x_m)_{m\geqq n+1}, (y_m)_{m\geqq n+2},$$
 and $$a^2((x_m)_{m\geqq n+1})=(x_m)_{m\geqq n+2}.$$ 
 We also define $\widetilde{C}^{(\varphi),\bullet}_{\varphi,\gamma}(\bold{D}^{(+)}_{\mathrm{dif}}(M)):=\varinjlim_{n,a^{' \bullet}}\widetilde{C}^{\bullet}_{\varphi,\gamma}(\bold{D}^{(+)}_{\mathrm{dif},n}(M))$ where the transition map $a^{' \bullet}$ is defined by 
$$a^{' 0}((x_m)_{m\geqq n})=(x_{m-1})_{m\geqq n+1},\,\,\,\, a^{'1}((x_m)_{m\geqq n}, (y_m)_{m\geqq n+1})
=((x_{m-1})_{m\geqq n+1}, (y_{m-1})_{m\geqq n+2}),$$ and $$a^{' 2}((x_m)_{m\geqq n+1})
=(x_{m-1})_{m\geqq n+2}.$$ Then, it is easy to see that quasi-isomorphism $C^{\bullet}_{\gamma}(\bold{D}^{(+)}_{\mathrm{dif},n}(M))\isom \widetilde{C}^{\bullet}_{\varphi,\gamma}(\bold{D}^{(+)}_{\mathrm{dif},n}(M))$ defined 
in (\ref{7e}) is compatible with the transition maps $C^{\bullet}_{\gamma}(\bold{D}^{(+)}_{\mathrm{dif},n}(M))\hookrightarrow C^{\bullet}_{\gamma}(\bold{D}^{(+)}_{\mathrm{dif},n+1}(M))$, 
$a^{\bullet}$ and $a^{'\bullet}$, hence induces quasi-isomorphisms
\begin{equation}\label{9e}
C^{\bullet}_{\gamma}(\bold{D}^{(+)}_{\mathrm{dif}}(M))\isom \widetilde{C}^{\bullet}_{\varphi,\gamma}(\bold{D}^{(+)}_{\mathrm{dif}}(M))\,\,\text{ and }\,\, C^{\bullet}_{\gamma}(\bold{D}^{(+)}_{\mathrm{dif}}(M))\isom \widetilde{C}^{(\varphi),\bullet}_{\varphi,\gamma}(\bold{D}^{(+)}_{\mathrm{dif}}(M)).
\end{equation}
For $\widetilde{C}^{(\varphi),\bullet}_{\varphi,\gamma}(\bold{D}^{(+)}_{\mathrm{dif}}(M))$, 
we also have a left inverse 
\begin{equation}\label{10e}
\widetilde{C}^{(\varphi),\bullet}_{\varphi,\gamma}(\bold{D}^{(+)}_{\mathrm{dif}}(M))\rightarrow C^{\bullet}_{\gamma}(\bold{D}^{(+)}_{\mathrm{dif}}(M))
\end{equation}
 of the above quasi-isomorphism  $C^{\bullet}_{\gamma}(\bold{D}^{(+)}_{\mathrm{dif}}(M))\rightarrow \widetilde{C}^{(\varphi),\bullet}_{\varphi,\gamma}(
\bold{D}^{(+)}_{\mathrm{dif}}(M))$
which is obtained as the limit of the map 
\begin{equation*}
\begin{CD}
\prod_{m\geqq n}\bold{D}^+_{\mathrm{dif},m}(M)@>>>
\prod_{m\geqq n}\bold{D}^+_{\mathrm{dif},m}(M)
\oplus \prod_{m\geqq n+1}\bold{D}^+_{\mathrm{dif},m}(M)
@>>>\prod_{m\geqq n+1}\bold{D}^+_{\mathrm{dif},m}(M)\\
@VV (x_m)_{m\geqq n}\mapsto x_nV @ VV ((x_m)_{m\geqq n}, (y_m)_{m\geqq n+1})
\mapsto x_n V @.\\
\bold{D}^+_{\mathrm{dif},n}(M)@> \gamma-1 >> \bold{D}^+_{\mathrm{dif},n}(M).@. \end{CD}
\end{equation*}

Taking the limits of the map $\widetilde{C}^{\bullet}_{\varphi,\gamma}(M^{(n)})
\rightarrow \widetilde{C}^{\bullet}_{\varphi,\gamma}(\bold{D}^+_{\mathrm{dif},n}(M))
: x\mapsto (\iota_m(x))_{m\geqq n_0}$ ($n_0=n, n+1$), we obtain the following maps
\begin{equation}\label{11e}
C^{\bullet}_{\varphi,\gamma}(M)\rightarrow \widetilde{C}^{\bullet}_{\varphi,\gamma}(
\bold{D}^+_{\mathrm{dif}}(M))\,\, \text{ and } \,\, C^{(\varphi), \bullet}_{\varphi,\gamma}(M)\rightarrow \widetilde{C}^{(\varphi),\bullet}_{\varphi,\gamma}(
\bold{D}^+_{\mathrm{dif}}(M)).
\end{equation}

Taking the limit of the exact sequence (\ref{6e}) with respect to the transition map induced by 
the canonical inclusion $M_0^{(n)}\hookrightarrow M_0^{(n+1)}$ and $a_{\bullet}$, and 
taking the quasi isomorphism $C^{\bullet}_{\gamma}(\bold{D}^{(+)}_{\mathrm{dif}}(M))\isom \widetilde{C}^{\bullet}_{\varphi,\gamma}(\bold{D}^{(+)}_{\mathrm{dif}}(M))$ in (\ref{9e}), we obtain the 
following exact triangle which is the top horizontal line in the proposition
\begin{equation*}
C^{\bullet}_{\varphi,\gamma}(M)
\xrightarrow{d_1} C^{\bullet}_{\varphi,\gamma}(M[1/t])\oplus 
C^{\bullet}_{\gamma}(\bold{D}^+_{\mathrm{dif}}(M))
\xrightarrow{d_2} C^{\bullet}_{\gamma}(\bold{D}_{\mathrm{dif}}(M))
\xrightarrow{[+1]}.
\end{equation*}

On the other hands, since we have 
 $$C^{\bullet}_{\varphi,\gamma}(M^{(n)}[1/t])
=\mathrm{Cone}(1-\varphi: C^{\bullet}_{\gamma}(M^{(n)}[1/t])\rightarrow 
C^{\bullet}_{\gamma}(M^{(n+1)}[1/t]))[-1]$$ for 
$n\geqq n(M)$ 
 (where 
we define $\mathrm{Cone}(f:M^{\bullet}\rightarrow N^{\bullet})[-1]^n=
M^{n}\oplus N^{n-1}$ and $d:M^{n}\oplus N^{n-1}\rightarrow M^{n+1}\oplus N^{n}:
(x,y)\mapsto (d_M(x), -f(x)-d_N(y))$),
taking the limit of the exact sequence (\ref{6e}) with respect to the transition map induced by 
$\varphi:M_0^{(n)}\hookrightarrow M_0^{(n+1)}$ and $a'_{\bullet}$, 
and taking the left inverse $\widetilde{C}^{(\varphi), \bullet}_{\varphi,\gamma}(\bold{D}^{(+)}_{\mathrm{dif}}(M))\rightarrow C^{\bullet}_{\gamma}(\bold{D}^{(+)}_{\mathrm{dif}}(M))$ in (\ref{10e}), 
and identifying $C^{\bullet}_{\varphi,\gamma}(M)\isom C^{(\varphi),\bullet}_{\varphi,\gamma}(M)$ by Lemma 
\ref{2.21} (2), 
we obtain the following exact triangle which is the bottom horizontal arrow in the proposition
\begin{equation*}
C^{\bullet}_{\varphi,\gamma}(M)
\xrightarrow{d_3}C^{(\varphi),\bullet}_{\gamma}(M[1/t])
\oplus C^{\bullet}_{\gamma}(\bold{D}^+_{\mathrm{dif}}(M))
\xrightarrow{d_4}C^{(\varphi),\bullet}_{\gamma}(M[1/t])
\oplus C^{\bullet}_{\gamma}(\bold{D}_{\mathrm{dif}}(M))\xrightarrow{[+1]}
\end{equation*}
with $d_3(x)=(f(x), g(x))$ and $d_4(x,y)=((\mathrm{can}-1)(x), g(x)-y)$, 
which proves the proposition.

\end{proof}

We next recall some notions concerning the $p$-adic Hodge theory for   $(\varphi,\Gamma)$-modules over
 the Robba ring. 
For a $(\varphi,\Gamma)$-module $M$ over $\mathcal{R}_A(\pi_K)$, let denote 
$$\bold{D}_{\mathrm{dR}}^K(M):=\mathrm{H}^0_{\gamma}(\bold{D}_{\mathrm{dif}}(M))\,\,\text{ and }
\,\,\bold{D}_{\mathrm{dR}}^K(M)^{i}:=\bold{D}^K_{\mathrm{dR}}(M)
\cap t^i\bold{D}^+_{\mathrm{dif}}(M)$$ for  $i\in \mathbb{Z}$. 
Let denote 
$$\bold{D}^K_{\mathrm{cris}}(M):=\mathrm{H}^0_{\gamma}(M[1/t]).$$ By Lemma \ref{2.18}, $\varphi:C^{\bullet}_{\gamma}(M[1/t])
\rightarrow C^{\bullet}_{\gamma}(M[1/t])$ induces a $\varphi$-semilinear automorphism

$$\varphi:\bold{D}^K_{\mathrm{cris}}(M)\isom \bold{D}^K_{\mathrm{cris}}(M).$$ More precisely, 
by Lemma \ref{2.19}, 
we have $\bold{D}_{\mathrm{cris}}(M)=\mathrm{H}^0_{\gamma}(M^{(n)}[1/t])$ and $\varphi$ induces an automorphism 
$\varphi:\mathrm{H}^0_{\gamma}(M^{(n)}[1/t])\xrightarrow{\varphi}
\mathrm{H}^0_{\gamma}(M^{(n+1)}[1/t])=\mathrm{H}^0_{\gamma}(M^{(n)}[1/t])$ for 
$n\geqq n(M)$. Using these facts, we define  an isomorphism
$$j_1:\bold{D}^K_{\mathrm{cris}}(M)=\mathrm{H}^0_{\gamma}(M^{(n)}[1/t])
\xrightarrow{\varphi^n}\mathrm{H}^0_{\gamma}(M^{(n)}[1/t])\isom 
\mathrm{H}^{(\varphi), 0}_{\gamma}(M[1/t]),$$ which does not depend on the choice of $n$. Then, the map $\iota:C^{(\varphi),\bullet}_{\gamma}(M[1/t])
\rightarrow C^{\bullet}_{\gamma}(\bold{D}_{\mathrm{dif}}(M))$ induces a $F\otimes_{\mathbb{Q}_p}A$-linear injection 
$$\iota:\bold{D}^K_{\mathrm{cris}}(M)
\xrightarrow{j_1}\mathrm{H}^{(\varphi),0}_{\gamma}(M[1/t])\xrightarrow{\iota}
\bold{D}^K_{\mathrm{dR}}(M).$$  We define another isomorphism 
$$j_2:\bold{D}^K_{\mathrm{cris}}(M)\xrightarrow{j_1}\mathrm{H}^{(\varphi),0}_{\gamma}(M[1/t])
\xrightarrow{\mathrm{can}}\mathrm{H}^{(\varphi),0}_{\gamma}(M[1/t]),$$ where 
$\mathrm{H}^{(\varphi),0}_{\gamma}(M[1/t])
\xrightarrow{\mathrm{can}}\mathrm{H}^{(\varphi),0}_{\gamma}(M[1/t])$ is the  map 
 induced  by $\mathrm{can}:C^{(\varphi),\bullet}_{\gamma}(M[1/t])\allowbreak\rightarrow 
 C^{(\varphi),\bullet}_{\gamma}(M[1/t])$, which is isomorphism by Lemma \ref{2.21}. 
 Then, we obtain the following commutative diagram
 \begin{equation*}
 \begin{CD}
 \bold{D}_{\mathrm{cris}}^K(M)@> 1-\varphi >>\bold{D}^K_{\mathrm{cris}}(M) \\
 @VV j_1 V @ VV j_2 V\\
 \mathrm{H}^{(\varphi),0}_{\gamma}(M[1/t])@> \mathrm{can}-\mathrm{id} >>
 \mathrm{H}^{(\varphi),0}_{\gamma}(M[1/t]).
 \end{CD}
 \end{equation*}

 Let denote 
$$\mathrm{exp}_M:\bold{D}^K_{\mathrm{dR}}(M)\rightarrow \mathrm{H}^1_{\varphi,\gamma}(M), \,\,\,\,\mathrm{exp}_{f,M}:\bold{D}^K_{\mathrm{cris}}(M)\xrightarrow{j_2}\mathrm{H}
^{(\varphi),0}_{\gamma}(M[1/t])\rightarrow 
\mathrm{H}^1_{\varphi,\gamma}(M)$$ for the boundary maps obtained by taking the cohomology of 
 the exact triangles in Proposition \ref{2.22}. We define 
  $$\mathrm{H}^1_{\varphi,\gamma}(M)_{e}=\mathrm{Im}
 (\bold{D}^K_{\mathrm{dR}}(M)\xrightarrow{\mathrm{exp}_M} \mathrm{H}^1_{\varphi,\gamma}(M))$$and 
 
 $$\mathrm{H}^1_{\varphi,\gamma}(M)_{f}=\mathrm{Im}
 (\bold{D}^K_{\mathrm{cris}}(M)\oplus \bold{D}^K_{\mathrm{dR}}(M)\xrightarrow{\mathrm{exp}_{f,M}\oplus \mathrm{exp}_M} \mathrm{H}^1_{\varphi,\gamma}(M)).$$
 We call the latter group the finite cohomology.
 Put $t_M(K):=\bold{D}^K_{\mathrm{dR}}(M)/\bold{D}^K_{\mathrm{dR}}(M)^0$. By Proposition \ref{2.22}, we obtain the following diagram of exact sequences
 \begin{equation}\label{12e}
 \begin{CD}
 0\rightarrow\mathrm{H}^0_{\varphi,\gamma}(M)@>{x\mapsto x}>> \bold{D}^K_{\mathrm{cris}}(M)^{\varphi=1}
 @> x\mapsto \overline{\iota(x)} >> t_M(K)@> \mathrm{exp}_M>> \mathrm{H}^1_{\varphi,\gamma}(M)_e\rightarrow0\\
 @VV\mathrm{id}V @VV x \mapsto x V @ VVx\mapsto (0,x) V@VV x\mapsto x V\\
  0\rightarrow\mathrm{H}^0_{\varphi,\gamma}(M)@>x\mapsto x>> \bold{D}^K_{\mathrm{cris}}(M)
 @> d_5>> \bold{D}^K_{\mathrm{cris}}(M)\oplus t_M(K)@>d_6>> \mathrm{H}^1_{\varphi,\gamma}(M)_{f}\rightarrow 0
 \end{CD}
  \end{equation}
  with 
 $$d_5(x,y)=((1-\varphi)x,\overline{ \iota(x)})\,\,\text{ and } d_6=\mathrm{exp}_{f,M}\oplus \mathrm{exp}_M,$$
 where we also denote $\mathrm{exp}_M:t_M(K)\rightarrow \mathrm{H}^1_{\varphi,\gamma}(M)$ which is naturally induced by $\mathrm{exp}_M:\bold{D}^K_{\mathrm{dR}}(M)\rightarrow \mathrm{H}^1_{\varphi,\gamma}(M)$.

By the proof of Proposition \ref{2.22}, 
we obtain the following explicit formulae of $\mathrm{exp}_M$ and 
$\mathrm{exp}_{f,M}$, which are very useful in applications.

\begin{prop}\label{2.23}
We have the following formulae.
\begin{itemize}
\item[(1)]For $x\in \bold{D}^{K}_{\mathrm{dR}}(M)$, take $\tilde{x}\in M^{(n)}[1/t]^{\Delta}$ $(n\geqq n(M))$ such that 
$$\iota_m(\tilde{x})-x\in \bold{D}^+_{\mathrm{dif},m}(M)$$ for any $m\geqq n$ $($ such $\tilde{x}$ exists 
by the exact sequence $(\ref{4e})$ in the proof of Proposition $\ref{2.22})$. Then we have 
$$\mathrm{exp}_M(x)=[(\gamma-1)\tilde{x},(\varphi-1)\tilde{x}]\in \mathrm{H}^1_{\varphi,\gamma}(M).$$
\item[(2)]For $x\in \bold{D}^K_{\mathrm{cris}}(M)$, take $\tilde{x}\in M^{(n)}[1/t]^{\Delta}$ $(n\geqq n(M))$ 
such that 
$$\iota_n(\tilde{x})\in \bold{D}^+_{\mathrm{dif},n}(M)$$ and 
$$\iota_{n+k}(\tilde{x})-\sum_{l=1}^k\iota_{n+l}(\varphi^n(x))\in \bold{D}^+_{\mathrm{dif},n+k}(M)$$ for any $k\geqq 1$$ ($ we remark that  we have $\varphi^n(x)\in M^{(n)}[1/t]$ by Lemma $\ref{2.19}$ and that such $\tilde{x}$ exists 
by the exact sequence $(\ref{4e})$$)$. 
Then we have
$$\mathrm{exp}_{f,M}(x)=[(\gamma-1)\tilde{x}, (\varphi-1)\tilde{x}+\varphi^n(x)]\in \mathrm{H}^1_{\varphi,\gamma}(M).$$

\end{itemize}

\end{prop}
\begin{proof}
These formulae directly  follow from simple but a little bit long diagram chases in the proof of Proposition \ref{2.22}. For the convenience of the reader, we 
give a proof of these formulae. 

We first prove the formula (1). By the proof of Proposition \ref{2.22}, the above
exact triangle in this proposition is obtained by  taking the limit of 
the composition of the quasi-isomorphism 
$$\widetilde{C}^{\bullet}_{\varphi,\gamma}(M^{(n)})\isom 
\mathrm{Cone}(\widetilde{C}^{\bullet}_{\varphi,\gamma}(M^{(n)}[1/t]\oplus 
\widetilde{C}^{\bullet}_{\varphi,\gamma}(\bold{D}^+_{\mathrm{dif},n}(M))\rightarrow 
\widetilde{C}^{\bullet}_{\varphi,\gamma}(\bold{D}_{\mathrm{dif},n}(M)))[-1]:=C_1^{\bullet}$$
( which is obtained by the exact sequence (\ref{6e}) ) with the inverse of the quasi-isomorphism 
$$C_2^{\bullet}:=\mathrm{Cone}(\widetilde{C}^{\bullet}_{\varphi,\gamma}(M^{(n)}[1/t]\oplus 
C^{\bullet}_{\gamma}(\bold{D}^+_{\mathrm{dif},n}(M))\rightarrow 
C^{\bullet}_{\gamma}(\bold{D}_{\mathrm{dif},n}(M)))[-1]\isom C_1^{\bullet}$$
induced by  by the quasi-isomorphism
 $C^{\bullet}_{\gamma}(\bold{D}^{(+)}_{\mathrm{dif},n}(M))
\rightarrow \widetilde{C}^{\bullet}_{\varphi,\gamma}(\bold{D}^{(+)}_{\mathrm{dif},n}(M))
:x\mapsto (x)_{m\geqq n}$ of (\ref{9e}).

By definition of $\mathrm{exp}_M(-)$,  for $x\in \mathrm{H}^0_{\gamma}
(\bold{D}_{\mathrm{dif},n}(M))$, these quasi-isomorphisms send 
$\mathrm{exp}_M(x)$ ( 
which we see as an element of $\mathrm{H}^1(\widetilde{C}^{\bullet}_{\varphi,\gamma}(M^{(n)}))$)
 to the element $[0,0,x]\in \mathrm{H}^1(C_2^{\bullet})$ represented by 
 $(0,0,x)\in \widetilde{C}^1_{\varphi,\gamma}(M^{(n)}[1/t])\oplus 
 C^1_{\gamma}(\bold{D}^+_{\mathrm{dif},n}(M)
 )\oplus C^0_{\gamma}(\bold{D}_{\mathrm{dif},n}(M))$. 
 Take $\tilde{x}\in M^{(n)}[1/t]^{\Delta}$ satisfying the condition in (1), then it suffices to show that 
 $[(\gamma-1)\tilde{x},(\varphi-1)\tilde{x}]\in \mathrm{H}^1(\widetilde{C}^{\bullet}_{\varphi,\gamma}(M^{(n)}))$ and $[0,0,x]\in \mathrm{H}^1(C^{\bullet}_2)$ are the same element in 
 $\mathrm{H}^1(C_1^{\bullet})$. By definition, $[(\gamma-1)\tilde{x},(\varphi-1)\tilde{x}]$ is sent to 
 $$[((\gamma-1)\tilde{x}, (\varphi-1)\tilde{x}), ((\iota_m((\gamma-1)\tilde{x}))_{m\geqq n}, 
 (\iota_m((\varphi-1)\tilde{x}))_{m\geqq n+1}), 0]$$
 and $[0,0,x]$ is sent to 
 $$[0,0,(-x)_{m\geqq n}]$$ 
 in $\mathrm{H}^1(C^{\bullet}_1)$, both are represented by elements of 
 $\widetilde{C}^1_{\varphi,\gamma}(M^{(n)}[1/t])\oplus \widetilde{C}^1_{\varphi,\gamma}
 (\bold{D}^+_{\mathrm{dif},n}(M))\oplus \widetilde{C}^0_{\varphi,\gamma}(\bold{D}_{\mathrm{dif},n}(M))$ (we remark the sign; for $f:C^{\bullet}\rightarrow D^{\bullet}$, 
 we define $D^{\bullet-1}\rightarrow \mathrm{Cone}(C^{\bullet}\rightarrow D^{\bullet})[-1]$ by 
 $x\mapsto (-x, 0)$ and $\mathrm{Cone}(C^{\bullet}\rightarrow D^{\bullet})[-1]\rightarrow C^{\bullet}$ is defined by $(x,y)\mapsto y$). Then, it s easy to check that the difference of these two elements is the coboundary of the element 
 $$(\widetilde{x}, (\iota_m(\widetilde{x})-x)_{m\geqq n})\in C^0_1=M^{(n)}[1/t]^{\Delta}\oplus 
 \prod_{m\geqq n}\bold{D}^+_{\mathrm{dif},m}(M)^{\Delta},$$
 which proves (1). 
 
 We next prove (2). The below exact triangle in Proposition 
 \ref{2.22} is obtained by taking the limit of the composition of the quasi-isomorphism 
 $\widetilde{C}^{\bullet}_{\varphi,\gamma}(M^{(n)})\isom C_1^{\bullet}$ defined above with 
 the quasi-isomorphism 
 $$C_1^{\bullet}\isom \mathrm{Cone}(\widetilde{C}^{\bullet}_{\varphi,\gamma}(M^{(n)}[1/t])
 \oplus C^{\bullet}_{\gamma}(\bold{D}^+_{\mathrm{dif},n}(M))\rightarrow C^{\bullet}_{\gamma}
 (\bold{D}_{\mathrm{dif},n}(M)))[-1]:=C_3^{\bullet}$$
 induced by the map $\prod_{m\geqq n}\bold{D}^+_{\mathrm{dif},m}(M)\rightarrow 
 \bold{D}^+_{\mathrm{dif},n}(M):(x_m)_{m\geqq n}\rightarrow x_n$, with 
 the inverse of  the quasi-isomorphism 
 \begin{multline*}
 C^{\bullet}_3\isom \mathrm{Cone}(C^{\bullet}_{\gamma}(M^{(n)}[1/t])
 \oplus C^{\bullet}_{\gamma}(\bold{D}^+_{\mathrm{dif},n}(M))\\
 \rightarrow C^{\bullet}_{\gamma}(M^{(n+1)}[1/t])\oplus C^{\bullet}_{\gamma}(\bold{D}
 _{\mathrm{dif},n}(M)))[-1]:=C_4^{\bullet}
 \end{multline*}
 which is naturally obtained by the identity 
 $$\widetilde{C}^{\bullet}_{\varphi,\gamma}
 (M^{(n)}[1/t]))=\mathrm{Cone}(C^{\bullet}_{\gamma}(M^{(n)}[1/t])\xrightarrow{1-\varphi}
 C^{\bullet}_{\gamma}(M^{(n+1)}[1/t]))[-1].$$
 
 For $x'\in \mathrm{H}^0_{\gamma}(M^{(n+1)}[1/t])$, the image of $x'$ by the first boundary map of the cone $C_4^{\bullet}$ is equal to $[0,0,x',0]\in \mathrm{H}^1(C_4^{\bullet})$ which is represented by 
 the element $(0,0,x',0)\in C^{1}_{\gamma}(M^{(n)}[1/t])\oplus C^{1}_{\gamma}(\bold{D}^+_{\mathrm{dif},n}(M))
 \oplus C^0_{\gamma}(M^{(n+1)}[1/t])\oplus C^0_{\gamma}(\bold{D}_{\mathrm{dif},n}(M))$. 
 Take $\tilde{x}'\in M^{(n)}[1/t]^{\Delta}$ such that $\iota_n(\tilde{x}')\in \bold{D}^+_{\mathrm{dif},n}(M)$ and that $\iota_{n+k}(\tilde{x}')-\sum_{l=1}^k\iota_{n+l}(x')\in \bold{D}^+_{\mathrm{dif},n+k}(M)$ 
 for any $k\geqq 1$, then, by definition of the map $j_2:\bold{D}^K_{\mathrm{cris}}(M)
 \isom \mathrm{H}^{(\varphi),0}_{\gamma}(M[1/t])$ and $\mathrm{exp}_{f,M}$, it suffices to show that 
 the element $[(\gamma-1)\tilde{x}',(\varphi-1)\tilde{x}'+x']\in \mathrm{H}^1(\widetilde{C}^{\bullet}_{\varphi,\gamma}(M^{(n)}))$ is sent to $[0,0,x',0]\in \mathrm{H}^1(C_4^{\bullet})$ by the above quasi-
 isomorphisms. By definition, the element $[(\gamma-1)\tilde{x}',(\varphi-1)\tilde{x}'+x']$ is 
 sent to 
 $$[(\gamma-1)\tilde{x}', \iota_n((\gamma-1)\tilde{x}'), (\varphi-1)\tilde{x}'+x',0]\in \mathrm{H}^1(C^{\bullet}_4)$$ by the above quasi-isomorphism. Then, it is easy to check that 
 the difference of this element with 
 $[0,0,x',0]$ 
 is  the coboundary of the element 
 $$(\tilde{x}',\iota_n(\tilde{x}'))\in C^0_4=M^{(n)}[1/t]^{\Delta}\oplus \bold{D}^+_{\mathrm{dif},n}(M)^{\Delta},$$
 which proves the formula (2).

\end{proof}

We next generalize the Bloch-Kato duality concerning the finite cohomology for 
$(\varphi,\Gamma)$-modules. 
Let $L=A$ be a finite extension of $\mathbb{Q}_p$, and let $M$ be a $(\varphi,\Gamma)$-module 
over $\mathcal{R}_L(\pi_K)$. We say that $M$ is de Rham if the equality $\mathrm{dim}_{K}\bold{D}^K_{\mathrm{dR}}(M)=[L:\mathbb{Q}_p]\cdot r_M$ holds. When $M$ is de Rham, we have 
a natural $L$-bilinear perfect pairing 
\begin{multline}\label{13e}
[-,-]_{\mathrm{dR}}:\bold{D}^K_{\mathrm{dR}}(M^*)\times \bold{D}^K_{\mathrm{dR}}(M)
\xrightarrow{(f,x)\mapsto f(x)} \bold{D}^K_{\mathrm{dR}}(\mathcal{R}_L(1))\\
=L\otimes_{\mathbb{Q}_p}K\frac{1}{t}\bold{e}_1\xrightarrow{
\frac{a}{t}\bold{e}_1\mapsto \frac{1}{[K:\mathbb{Q}_p]}(\mathrm{id}\otimes\mathrm{tr}_{K/\mathbb{Q}_p})(a)}L,
\end{multline}
which induces a natural isomorphisms
$$\bold{D}^K_{\mathrm{dR}}(M)\isom 
\bold{D}^K_{\mathrm{dR}}(M^*)^{\vee}\,\,\text{ and } \bold{D}^K_{\mathrm{dR}}(M)^0\isom t_{M^*}(K)^{\vee}.$$

\begin{prop}\label{2.24}
Let $L=A$ be a finite extension of $\mathbb{Q}_p$, and 
let $M$ be a de Rham $(\varphi,\Gamma)$-module over $\mathcal{R}_L(\pi_K)$. 
Then, $\mathrm{H}^1_{\varphi,\gamma}(M)_{f}$  is the orthogonal 
complement of 
$\mathrm{H}^1_{\varphi,\gamma}(M^*)_{f}$ with respect to
the Tate duality pairing 
$\langle, \rangle:\mathrm{H}^1_{\varphi,\gamma}(M^*)\times \mathrm{H}^1_{\varphi,\gamma}(M)\rightarrow 
L$.
\end{prop}

\begin{proof}
We remark that we have $\mathrm{dim}_L\mathrm{H}^1_{\varphi,\gamma}(M)_{f}=\mathrm{dim}_L
(t_M(K)) +\mathrm{dim}_L\mathrm{H}^0_{\varphi,\gamma}(M)$ by the below exact sequence of (\ref{12e}). 
Using this formula for $M, M^*$, it is easy to check that we have 
$\mathrm{dim}_L\mathrm{H}^1_{\varphi,\gamma}(M)_f+
\mathrm{dim}_L\mathrm{H}^1_{\varphi,\gamma}(M^*)_f=\mathrm{dim}_L\mathrm{H}^1_{\varphi,\gamma}(M)$ under the assumption that $M$ is de Rham. Hence, it suffices to show that 
we have $\langle x,y\rangle=0$ for any $x\in \mathrm{H}^1_{\varphi,\gamma}(M^*)_{f}$ 
and $y\in \mathrm{H}^1_{\varphi,\gamma}(M)_{f}$ by comparing the dimensions. 
By definition of $\mathrm{H}^1_{\varphi,\gamma}(-)_f$, this claim follows from the following lemma \ref{2.25}.

\end{proof}
Let $M$ be a $(\varphi,\Gamma)$-module over $\mathcal{R}_A(\pi_K)$ (we don't need to assume that $M$ is de Rham). Using the isomorphism $j_2:\bold{D}^K_{\mathrm{cris}}(M^*)\isom \mathrm{H}^{(\varphi),0}_{\gamma}(M^*[1/t])$,
define an $A$-bilinear pairing
\begin{multline*}
h(-,-):(\bold{D}^K_{\mathrm{cris}}(M^*)\oplus \bold{D}^K_{\mathrm{dR}}(M^*))\times
(\mathrm{H}^{(\varphi), 1}_{\gamma}(M[1/t])\oplus \mathrm{H}^1_{\gamma}(\bold{D}^+_{\mathrm{dif}}(M)))\\
\rightarrow \mathrm{H}^{(\varphi), 1}_{\gamma}(
M^*\otimes M[1/t])\oplus \mathrm{H}^1_{\gamma}(\bold{D}_{\mathrm{dif}}(M^*\otimes M))
\end{multline*}
by $$h((x, y),([z],[w])):=([j_2(x)\otimes z], [y\otimes w]).$$ 

\begin{lemma}\label{2.25}
For $(x,y)\in \bold{D}^K_{\mathrm{cris}}(M^*)\oplus\bold{D}^K_{\mathrm{dR}}(M^*)$ and 
$z\in \mathrm{H}^1_{\varphi,\gamma}(M)$, we have 
$$f_2(h((x, y), g(z)))=(\mathrm{exp}_{f,M^*}(x)+\mathrm{exp}_{M^*}(y))\cup z \in \mathrm{H}^2_{\varphi,\gamma}(M^*\otimes M),$$ where 
$$g:\mathrm{H}^1_{\varphi,\gamma}(M)\rightarrow \mathrm{H}^{(\varphi),1}_{\gamma}(M)
\oplus \mathrm{H}^1_{\gamma}(\bold{D}^+_{\mathrm{dR}}(M))$$ is induced by $d_3$ and 
$$f_2:\mathrm{H}^{(\varphi), 1}_{\gamma}(
M^*\otimes M[1/t])\oplus \mathrm{H}^1_{\gamma}(\bold{D}_{\mathrm{dif}}(M^*\otimes M))
\rightarrow \mathrm{H}^2_{\varphi,\gamma}(M^*\otimes M)$$ is the second boundary map 
of the below exact triangle of Proposition \ref{2.22}.

\end{lemma}

\begin{proof}
The equality $\mathrm{exp}_{M^*}(y)\cup z =f_2(h((0,y), g(z)))$ for $y\in \bold{D}^K_{\mathrm{dR}}(M^*)$, $z\in \mathrm{H}^1_{\varphi,\gamma}(M)$ is proved in Lemma 2.13 
of \cite{Na14a}. 
Hence, it suffices to show the equality 
$$\mathrm{exp}_{f,M^*}(x)\cup z =f_2(h((x,0), g(z)))$$
for $x\in\bold{D}^K_{\mathrm{cris}}(M^*)$, 
whose proof is also just a diagram chase similar to  the proof of Proposition \ref{2.23}, hence 
we omit the proof.

\end{proof}
Finally, we compare our exponential map with the Bloch-Kato exponential map 
for $p$-adic representations. Here, we assume that $A=\mathbb{Q}_p$ for simplicity, we can 
do the same things for any $L=A$ a finite $\mathbb{Q}_p$-algebra.

For an $\mathbb{Q}_p$-representation $V$ of $G_{K}$, by Bloch-Kato \cite{BK90}, we have the following 
diagram of exact sequences
 \begin{equation}\label{14e}
 \begin{CD}
 0\rightarrow\mathrm{H}^0(K,V)@>{x\mapsto x}>> \bold{D}^K_{\mathrm{cris}}(V)^{\varphi=1}
 @> x\mapsto \overline{x} >> t_V(K)@> \mathrm{exp}_V>> \mathrm{H}^1_e(K,V)\rightarrow0\\
 @VV\mathrm{id}V @VV x \mapsto x V @ VVx\mapsto (0,x) V@VV x\mapsto x V\\
  0\rightarrow\mathrm{H}^0(K,V)@>x\mapsto x>> \bold{D}^K_{\mathrm{cris}}(V)
 @> d_5>> \bold{D}^K_{\mathrm{cris}}(V)\oplus t_V(K)@>d_6>> \mathrm{H}^1_f(K,V)\rightarrow 0
 \end{CD}
  \end{equation}
  with 
 $$d_5(x,y)=((1-\varphi)x,\overline{x})\,\,\text{ and } d_6=\mathrm{exp}_{f,V}\oplus \mathrm{exp}_V,$$
 which is associated to the exact sequences which is the tensor product of $V$ ( over $\mathbb{Q}_p$) with the 
 Bloch-Kato fundamental exact sequences
  \begin{equation*}
 \begin{CD}
 0@>>>\mathbb{Q}_p@>{x\mapsto (x,x)}>> \bold{B}^{\varphi=1}_{\mathrm{cris}}\oplus \bold{B}^+_{\mathrm{dR}}
 @>(x,y)\mapsto x-y>>  \bold{B}_{\mathrm{dR}}@>>> 0\\
 @.@VV\mathrm{id}V @VV (x,y) \mapsto (x,y)V @ VVx\mapsto (0,x) V\\
  0@>>>\mathbb{Q}_p@>x\mapsto (x,x)>> \bold{B}_{\mathrm{cris}}\oplus\bold{B}^+_{\mathrm{dR}}
 @> (x,y)\mapsto (x, (1-\varphi)x-y)>> \bold{B}_{\mathrm{cris}}\oplus\bold{B}_{\mathrm{dR}}@>>>  0.
 \end{CD}
  \end{equation*}
  We want to compare the diagram (\ref{14e}) with the diagram (\ref{12e}) for $M=\bold{D}_{\mathrm{rig}}(V)$. 
  More generally, as in \S2.4 of \cite{Na14a}, we compare the similar diagram defined below for a $B$-pair $W=(W_e,W^+_{\mathrm{dR}})$ 
  with the diagram (\ref{12e}) for the associated $(\varphi,\Gamma)$-module $\bold{D}_{\mathrm{rig}}(W)$. 
  For the definitions of $B$-pairs and the definition of the functor $W\mapsto \bold{D}_{\mathrm{rig}}(W)$ 
  which gives an equivalence between the category of $B$-pairs and that of $(\varphi,\Gamma)$-modules 
  over $\mathcal{R}(\pi_K)$ which we don't recall here, see \S2.5 \cite{Na14a} and \cite{Ber08a}.
  
  Let $W=(W_e,W^+_{\mathrm{dR}})$ be a $B$-pair for $K$. Put $W_{\mathrm{cris}}:=\bold{B}_{\mathrm{cris}}\otimes_{\bold{B}_e}W_e$, which is naturally equipped with an action of $\varphi$.  Since we have an exact sequence $0\rightarrow \bold{B}_{\mathrm{cris}}^{\varphi=1}\rightarrow 
  \bold{B}_{\mathrm{cris}}\xrightarrow{1-\varphi}\bold{B}_{\mathrm{cris}}\rightarrow 0$, we have a 
  natural quasi-isomorphism ( the vertical arrows) between the following two complexes of $G_K$-modules concentrated in 
  degree $[0,1]$
  \begin{equation*}
  \begin{CD}
  [W_e\oplus W^+_{\mathrm{dR}}@> (x,y)\mapsto x-y>> W_{\mathrm{dR}}]\\
  @VV (x,y)\mapsto (x,y) V@ VVx\mapsto (0,x) V\\
  [W_{\mathrm{cris}}\oplus W^+_{\mathrm{dR}}@> (x,y)\mapsto ((1-\varphi)x, x-y)>>
  W_{\mathrm{cris}}\oplus W_{\mathrm{dR}}].
  \end{CD}
  \end{equation*}
  Put $$C^{\bullet}_{\mathrm{cont}}(G_K,W):=\mathrm{Cone}(C^{\bullet}_{\mathrm{cont}}(G_K,W_e)\oplus C^{\bullet}_{\mathrm{cont}}(G_K, W^+_{\mathrm{dR}})
  \rightarrow C^{\bullet}_{\mathrm{cont}}(G_K,W_{\mathrm{dR}}))[-1]$$ and
   \begin{multline*}
   C^{\bullet}_{\mathrm{cont}}(G_K,W)':=\mathrm{Cone}(C^{\bullet}_{\mathrm{cont}}(G_K,W_{\mathrm{cris}})\oplus C^{\bullet}_{\mathrm{cont}}(G_K,W^+_{\mathrm{dR}})\\
   \rightarrow 
   C^{\bullet}_{\mathrm{cont}}(G_K,W_{\mathrm{cris}})\oplus  C^{\bullet}_{\mathrm{cont}}(G_K,W_{\mathrm{dR}}))[-1].
   \end{multline*}
   We identify 
   $$\mathrm{H}^i(K,W):=\mathrm{H}^i(C^{\bullet}_{\mathrm{cont}}(G_K,W))=\mathrm{H}^i(C^{\bullet}_{\mathrm{cont}}(G_K,W)')$$
    by the above quasi-isomorphism. Put $\bold{D}^K_{\mathrm{cris}}(W):=\mathrm{H}^0(K, W_{\mathrm{cris}})$, 
    $\bold{D}^K_{\mathrm{dR}}(W):=\mathrm{H}^0(K, W_{\mathrm{dR}})$, and $\bold{D}^K_{\mathrm{dR}}(W)^i:=\bold{D}^K_{\mathrm{dR}}(W)
    \cap t^iW^+_{\mathrm{dR}}$ for $i\in \mathbb{Z}$. Taking the cohomology of the mapping cones, then we obtain the similar diagram of exact sequences as in (\ref{14e}) for $W$.  By definition, it is clear that the diagram (\ref{14e}) for the associated $B$-pair $W(V):=(\bold{B}_e\otimes_{\mathbb{Q}_p}V, 
   \bold{B}^+_{\mathrm{dR}}\otimes_{\mathbb{Q}_p}V)$ is canonically isomorphic to 
   the diagram (\ref{14e}) for $V$ defined by Bloch-Kato.
   
   Our comparison result is the following.

\begin{prop}\label{2.27}
\begin{itemize}
\item[(1)]We have the following functorial isomorphisms 
\begin{itemize}
\item[(i)]$\mathrm{H}^i(K, W)\isom \mathrm{H}^i_{\varphi,\gamma}(\bold{D}_{\mathrm{rig}}(W))$, 
\item[(ii)]$\bold{D}^K_{\mathrm{dR}}(W)^j\isom \bold{D}^K_{\mathrm{dR}}(\bold{D}_{\mathrm{rig}}(W))^j$ for $j\in \mathbb{Z}$,
\item[(iii)]$\bold{D}^K_{\mathrm{cris}}(W)\isom \bold{D}^K_{\mathrm{cris}}(\bold{D}_{\mathrm{rig}}(W)).$

\end{itemize}
\item[(2)]The isomorphisms in $(1)$ induces an isomorphism from the diagram $(\ref{14e})$ for $W$ to 
the diagram $(\ref{12e})$ for $\bold{D}_{\mathrm{rig}}(W)$.
\end{itemize}

\end{prop}
\begin{proof}
We already proved (i), (ii) of (1) and the comparison of the above exact sequence in $(\ref{14e})$ for $W$ 
with that in $(\ref{12e})$ for $\bold{D}_{\mathrm{rig}}(W)$, see Theorem 2.21 of \cite{Na14a} or the references 
in the proof of this theorem. 

Moreover, the isomorphism (iii) may be well known to the experts, but we give 
a proof of this since we couldn't  find suitable references. In this proof, we freely use the notations in \S2.5 \cite{Na14a} or in \cite{Ber08a}; please see these references. We first note that the inclusion 
$(\widetilde{\bold{B}}^+_{\mathrm{rig}}[1/t]\otimes_{\bold{B}_e}W_e)^{G_K}\hookrightarrow 
\bold{D}^K_{\mathrm{cris}}(W)$ induced by the natural inclusion $\widetilde{\bold{B}}^+_{\mathrm{rig}}
:=\cap_{n\geqq 0}\varphi^n(\bold{B}^+_{\mathrm{cris}})\hookrightarrow \bold{B}^+_{\mathrm{cris}}$ is 
isomorphism since $\bold{D}^K_{\mathrm{cris}}(W)$ is a finite dimensional $\mathbb{Q}_p$-vector space 
on which $\varphi$ acts as an automorphism. Moreover, in the same way as the proof of 
Proposition 3.4 of \cite{Ber02}, we can show that the natural inclusion 
$(\widetilde{\bold{B}}^+_{\mathrm{rig}}[1/t]\otimes_{\bold{B}_e}W_e)^{G_K}\hookrightarrow 
(\widetilde{\bold{B}}^{\dagger}_{\mathrm{rig}}[1/t]\otimes_{\bold{B}_e}W_e)^{G_K}$ is also isomorphism. 
Since we have $\widetilde{\bold{B}}^{\dagger}_{\mathrm{rig}}[1/t]\otimes_{\bold{B}_e}W_e=
\widetilde{\bold{B}}^{\dagger}_{\mathrm{rig}}[1/t]\otimes_{\mathcal{R}(\pi_K)[1/t]}\bold{D}_{\mathrm{rig}}(W)[1/t]$ by definition of 
$\bold{D}_{\mathrm{rig}}(W)$, it suffices to show that the natural inclusion 
$\bold{D}^K_{\mathrm{cris}}(\bold{D}_{\mathrm{rig}}(W))\hookrightarrow 
(\widetilde{\bold{B}}^{\dagger}_{\mathrm{rig}}[1/t]\otimes_{\mathcal{R}(\pi_K)[1/t]}\bold{D}_{\mathrm{rig}}(W)[1/t])^{G_K}=:D_0$ is isomorphism. Moreover, it suffices to show that $D_0$ is contained in $\bold{D}_{\mathrm{rig}}(W)[1/t]$. This claim is proved as follows. Define $\mathcal{R}(\pi_K)\otimes_FD_0\subseteq \widetilde{\bold{B}}^{\dagger}_{\mathrm{rig}}\otimes_FD_0$ which are $(\varphi,\Gamma)$-module over $\mathcal{R}(\pi_K)$ (
respectively $(\varphi,G_K)$-module over $\widetilde{\bold{B}}^{\dagger}_{\mathrm{rig}}$). Then, by Th\'eor\`eme 1.2 of \cite{Ber09}, the natural map $\widetilde{\bold{B}}^{\dagger}_{\mathrm{rig}}\otimes_FD_0\rightarrow \widetilde{\bold{B}}^{\dagger}_{\mathrm{rig}}[1/t]\otimes_{\mathcal{R}(\pi_K)[1/t]}\bold{D}_{\mathrm{rig}}(W)[1/t]:a\otimes x\mapsto a\cdot x$ (which is actually an inclusion) of $(\varphi,G_K)$-modules factors through 
$\mathcal{R}(\pi_K)\otimes_FD_0\rightarrow \bold{D}_{\mathrm{rig}}(W)[1/t]$, in particular 
we have $D_0\subseteq \bold{D}_{\mathrm{rig}}(W)[1/t]$, which proves the claim.

We next prove that the below exact sequences in (\ref{14e}) for $W$ is isomorphic to 
that in (\ref{12e}) for $\bold{D}_{\mathrm{rig}}(W)$ by the isomorphisms in (1) of this proposition. 
Since the other commutativities are clear, or were already proved in Theorem 2.21 \cite{Na14a}, 
it suffices to show that the following diagram commutes
\begin{equation}
\begin{CD}
\bold{D}^K_{\mathrm{cris}}(\bold{D}_{\mathrm{rig}}(W))@> \mathrm{exp}_{f,\bold{D}_{\mathrm{rig}}(W)
}>> \mathrm{H}^1_{\varphi,\gamma}(\bold{D}_{\mathrm{rig}}(W)) \\
@VV \isom V @ VV \isom V\\
\bold{D}^K_{\mathrm{cris}}(W)@> \mathrm{exp}_{f,W} >>
\mathrm{H}^1(K, W).
\end{CD}
\end{equation}
In the same way as the proof of Theorem 2.21 \cite{Na14a}, 
we assume that $\Delta=\{1\}$, and using the canonical identifications
$$\mathrm{H}^1(K,W)\isom \mathrm{Ext}^1(B, W),\,\,\,\,\mathrm{H}^1_{\varphi,\gamma}(\bold{D}_{\mathrm{rig}}(W))
\isom \mathrm{Ext}^1(\mathcal{R}(\pi_K), \bold{D}_{\mathrm{rig}}(W))$$
( where we denote by $B=(\bold{B}_e,\bold{B}^+_{\mathrm{dR}})$ for the trivial $B$-pair), it suffices 
to show that, for $a\in \bold{D}^K_{\mathrm{cris}}(\bold{D}_{\mathrm{rig}}(W))$,  the extension corresponding to $\mathrm{exp}_{f,\bold{D}_{\mathrm{rig}}(W)}(a)$ is 
sent to the extension corresponding to $\mathrm{exp}_{f, W}(a)$ by the inverse functor $W(-)$ of $\bold{D}_{\mathrm{rig}}(-)$. We prove this claim as follows. Take $n\geqq 1$ sufficiently large such that 
$a\in (\bold{D}^{(n)}_{\mathrm{rig}}(W)[1/t])^{\Gamma_K}$. Take $\tilde{a}\in \bold{D}^{(n)}_{\mathrm{rig}}[1/t]$ satisfying the condition in (2) of Proposition \ref{2.23}. By (2) of this proposition, then the extension 
$D_a$ corresponding to $\mathrm{exp}_{f,\bold{D}_{\mathrm{rig}}(W)}(a)$ is written by 
$$[0\rightarrow \bold{D}_{\mathrm{rig}}(W)\xrightarrow{x\mapsto (x,0)} \bold{D}_{\mathrm{rig}}(W)\oplus 
\mathcal{R}(\pi_K)\bold{e}\xrightarrow{(x,y\bold{e})\mapsto y}\mathcal{R}(\pi_K)\rightarrow 0]$$
such that 
$$\varphi((x, y\bold{e}))=(\varphi(x)+\varphi(y)((\varphi-1)\tilde{a}+\varphi^n(a)), \varphi(y)\bold{e})$$
and 
$$\gamma((x, y\bold{e}))=(\gamma(x)+\gamma(y)(\gamma-1)\tilde{a}, \gamma(y)\bold{e})$$

(here, we remark that there is an mistake in \cite{Na14a}; in the proof  of Theorem 2.21 of \cite{Na14a}, 
$D_a$ should be defined by 
$$\varphi((x, y\bold{e}))=(\varphi(x)+\varphi(y)(\varphi-1)\tilde{a}, \varphi(y)\bold{e})$$
and 
$$\gamma((x, y\bold{e}))=(\gamma(x)+\gamma(y)(\gamma-1)\tilde{a}, \gamma(y)\bold{e})).$$

On the other hand, by definition of $\mathrm{exp}_{f,W}$, the extension 
$$W_a:=(W_{e,a}, W^+_{\mathrm{dR},a}:=W^+_{\mathrm{dR}}\oplus \bold{B}^+_{\mathrm{dR}}\bold{e}_{\mathrm{dR}})$$ corresponding to $\mathrm{exp}_{f,W}(a)$  is defined by 
$$g(x,y\bold{e}_{\mathrm{dR}})=(g(x), g(y)\bold{e}_{\mathrm{dR}})$$ 
for $x\in W^+_{\mathrm{dR}}$, $y\in \bold{B}^+_{\mathrm{dR}}$, $g\in G_K$, and $W_{e,a}$ is defined as the kernel 
of the following surjection $$W_{\mathrm{cris},a}:=W_{\mathrm{cris}}\oplus \bold{B}_{\mathrm{cris}}\bold{e}_{\mathrm{cris}}\rightarrow W_{\mathrm{cris},a}:(x,y\bold{e}_{\mathrm{cris}})\mapsto ((\varphi-1)x+\varphi(y)a, (\varphi-1)y\bold{e}_{\mathrm{cris}})$$
on which $G_K$ acts by $g(\bold{e}_{\mathrm{cris}})=\bold{e}_{\mathrm{cris}}$ 
(actually, this is equal to the kernel of the surjection
$$W_{\mathrm{rig},a}:=W_{\mathrm{rig}}\oplus \widetilde{\bold{B}}^+_{\mathrm{rig}}[1/t]\bold{e}_{\mathrm{cris}}\rightarrow W_{\mathrm{rig},a}:(x,y\bold{e}_{\mathrm{cris}})\mapsto ((\varphi-1)x+\varphi(y)a, (\varphi-1)y\bold{e}_{\mathrm{cris}})$$
where we define $W_{\mathrm{rig}}:=\widetilde{\bold{B}}^+_{\mathrm{rig}}[1/t]\otimes_{\bold{B}_e}W_e$)
and the isomorphism $\bold{B}_{\mathrm{dR}}\otimes_{\bold{B}_e}W_{e,a}\isom \bold{B}_{\mathrm{dR}}\otimes_{\bold{B}^+_{\mathrm{dR}}}W^+_{\mathrm{dR}}$ is defined by 
$$\bold{B}_{\mathrm{dR}}\otimes_{\bold{B}_e}W_{e,a}
=\bold{B}_{\mathrm{dR}}\otimes_{\bold{B}_{\mathrm{cris}}}W_{\mathrm{cris},a}\xrightarrow{(x,y\bold{e}_{\mathrm{cris}})\mapsto (x,y\bold{e}_{\mathrm{dR}})} \bold{B}_{\mathrm{dR}}\otimes_{\bold{B}^+_{\mathrm{dR}}}W^+_{\mathrm{dR}}.$$ 
Then, by definition of the functor $\bold{D}_{\mathrm{rig}}(-)$ in \S2.2 \cite{Ber08a} (where the notation $D(-)$ is 
used), $\widetilde{\bold{B}}^{\dagger,r_n}_{\mathrm{rig}}\otimes_{\mathcal{R}^{(n)}(\pi_K)}\bold{D}^{(n)}_{\mathrm{rig}}(W_a)$ is equal to 
\begin{equation}\label{16e}
\{x\in \widetilde{\bold{B}}^{\dagger,r_n}_{\mathrm{rig}}[1/t]\otimes_{\bold{B}_{e}}W_{e,a}
|\iota_m(x)\in W^+_{\mathrm{dR},a} \text{ for any } m\geqq n \}.
\end{equation}
Since we have $\widetilde{\bold{B}}^{\dagger,r_n}_{\mathrm{rig}}[1/t]\otimes_{\bold{B}_e}W_{e,a}=
\widetilde{\bold{B}}^{\dagger,r_n}_{\mathrm{rig}}[1/t]\otimes_{\widetilde{\bold{B}}^+_{\mathrm{rig}}[1/t]}W_{\mathrm{rig},a}$, 
$\varphi^{-m}(\bold{e}_{\mathrm{cris}})=\bold{e}_{\mathrm{cris}}-\sum_{k=1}^m\varphi^{-k}(a)$ 
for $m\geqq 1$ and we have $\iota_{n+k}\circ\varphi^n=\varphi^{-k}$, 
it is easy to see that the group (\ref{16e}) is equal to 
$$\widetilde{\bold{B}}^{\dagger,r_n}_{\mathrm{rig}}\otimes_{\mathcal{R}^{(n)}(\pi_K)}\bold{D}^{(n)}_{\mathrm{rig}}(W)
\oplus \widetilde{\bold{B}}^{\dagger,r_n}_{\mathrm{rig}}(\tilde{a}+\varphi^n(\bold{e}_{\mathrm{cris}})),$$
which is easily to seen to be isomorphic to $\widetilde{\bold{B}}^{\dagger,r_n}_{\mathrm{rig}}\otimes_{\mathcal{R}^{(n)}(\pi_K)}D_a^{(n)}$ as a $(\varphi,G_K)$-module. Therefore, we obtain the isomorphism 
$$\bold{D}_{\mathrm{rig}}(W_a)\isom D_a$$ as an extension by Th\'eor\`eme 1.2 of  \cite{Ber09}, which proves the proposition.

\end{proof}




\section{Local $\varepsilon$-conjecture for $(\varphi,\Gamma)$-modules over the Robba ring}
From now on, we assume that $K=\mathbb{Q}_p$, and we freely
omit the notation $\mathbb{Q}_p$, i.e. we use the notation $\Gamma$, 
$\mathcal{R}_A$, $\bold{D}_{\mathrm{dR}}(M)$, 
$\bold{D}_{\mathrm{cris}}(M)$, $t_M$ etc instead of $\Gamma_{\mathbb{Q}_p}$, $\mathcal{R}_A(\pi_{\mathbb{Q}_p})$, $\bold{D}^{\mathbb{Q}_p}_{\mathrm{dR}}(M)$, 
$\bold{D}^{\mathbb{Q}_p}_{\mathrm{cris}}(M), t_M(\mathbb{Q}_p)\cdots $. Moreover, since 
Kato's and our conjectures are formulated after fixing a $\mathbb{Z}_p$-basis $\zeta=\{\zeta_{p^n}\}_{n\geqq 0}$ of $\mathbb{Z}_p(1)$, we also fix a parameter $\pi:=\pi_{\zeta}$ of $\mathcal{R}_A$ and 
denote $t=\mathrm{log}(1+\pi)$ as in Notation \ref{2.2}.

In this section, we formulate a conjecture which is a natural 
generalization of Kato's ( $p$-adic ) local $\varepsilon$-conjecture, where the main objects were
$p$-adic or torsion representations of $G_{\mathbb{Q}_p}$,  for $(\varphi,\Gamma)$-modules over the relative Robba ring $\mathcal{R}_A$. 
Since the article  \cite{Ka93b} in which the conjecture was stated has been unpublished until now, and since the compatibility of our conjecture with his conjecture is an important 
 part of our conjecture, here we also recall his original conjecture.
 
\subsection{Determinant functor}
Kato's and our conjectures are formulated using the theory of the determinant functor (\cite{KM76}). 
In this subsection, we briefly recall this theory following \cite{KM76}, \S2.1 of \cite{Ka93a}.

Let $R$ be a commutative ring. We define a category $\mathcal{P}_{R}$ such that whose objects are 
pairs $(L,r)$ where $L$ is an invertible $R$-module and $r:\mathrm{Spec}(R)\rightarrow \mathbb{Z}$ is a
 locally constant function, whose morphisms are defined by $\mathrm{Mor}_{\mathcal{P}_R}((L,r), (M,s)):=
 \mathrm{Isom}_R(L, M)$ if $r=s$, or empty otherwise. We call the objects of this category graded invertible 
 $R$-modules. The category $\mathcal{P}_R$ is equipped with the structure of a ( tensor )  product defined by $(L,r)\boxtimes (M,s):=(L\otimes_R M, r+s)$ with the natural associativity constraint and the commutativity constraint $(L,r)\boxtimes (M,s)\isom (M,s)\boxtimes (L, r):l\otimes m\mapsto (-1)^{r s}m\otimes l$. 
  From now on, we always identify $(L,r)\boxtimes (M,s)=(M,s)\boxtimes (L, r)$ by this constraint isomorphism. 
  The unit object for the product 
 is $\bold{1}_R:=(R, 0)$. For each $(L,r)$, set $L^{\vee}:=\mathrm{Hom}_R(L, R)$, then 
 $(L,r)^{-1}:=(L^{\vee},-r)$ becomes an inverse of $(L, r)$ by the isomorphism 
 $i_{(L,r)}:(L,r)\boxtimes(L^{\vee},-r)\isom \bold{1}_R$ induced by the evaluation map 
 $L\otimes_R\mathrm{Hom}_R(L, R)\isom R:x\otimes f\mapsto f(x)$. We remark that we have 
 $i_{(L, r)^{-1}}=(-1)^ri_{(L,r)}$.  
 For a ring homomorphism $f:R\rightarrow R'$, one has a 
 base change functor $(-)\otimes_RR':\mathcal{P}_R\rightarrow\mathcal{P}_{R'}$ defined by $(L,r)\mapsto (L,r)\otimes_{R}R':=(L\otimes_RR',r\circ f^*)$ where 
 $f^*:\mathrm{Spec}(R')\rightarrow \mathrm{Spec}(R)$.
 
 For a category $\mathcal{C}$, denote by $(\mathcal{C},\mathrm{is})$ for the category such that 
 whose objects are the same as $\mathcal{C}$ and the morphisms are all isomorphisms in $\mathcal{C}$. 
 Define a functor
 $$\mathrm{Det}_R:(\bold{P}_{\mathrm{fg}}(R),\mathrm{is})\rightarrow \mathcal{P}_R:
 P\mapsto (\mathrm{det}_RP,\mathrm{rk}_RP)$$
 where $\mathrm{rk}_R:\bold{P}_{\mathrm{fg}}(R)\rightarrow \mathbb{Z}_{\geqq 0}$ is the $R$-rank of $P$ 
 and $\mathrm{det}_RP:=\wedge_R^{\mathrm{rk}_RP}P$. Note that $\mathrm{Det}_R(0)=\bold{1}_R$ is the unit object. For a short exact sequence $0\rightarrow P_1\rightarrow P_2\rightarrow P_3\rightarrow 0$ 
 in $\bold{P}_{\mathrm{fg}}(R)$, we always identify $\mathrm{Det}_R(P_1)\boxtimes \mathrm{Det}_R(P_3)$ with $\mathrm{Det}_R(P_2)$ by the following functorial isomorphism (put $r_i:=\mathrm{rk}_RP_i$)
 \begin{equation}\label{17d}
  \mathrm{Det}_R(P_1)\boxtimes \mathrm{Det}_R(P_3)\isom \mathrm{Det}_R(P_2)
  \end{equation}
 induced by $$ (x_1\wedge\cdots\wedge x_{r_1})\otimes(\overline{x_{r_1+1}}\wedge\cdots\wedge \overline{x_{r_2}})\mapsto x_1\wedge\cdots\wedge x_{r_1}\wedge x_{r_1+1}\wedge\cdots \wedge x_{r_2}
 $$ where $x_1,\cdots, x_{r_1}$ (resp. $\overline{x_{r_1+1}},\cdots, \overline{x_{r_2}}$) 
 are local sections of $P_1$ (resp. $P_3$) and $x_i\in P_2$ ($i=r_1+1, \cdot r_2$) is a lift of $\overline{{x}_i}\in P_3$.
 
  For a bounded complex $P^{\bullet}$ in $\bold{P}_{\mathrm{fg}}(R)$, define $\mathrm{Det}_R(P^{\bullet})\in 
  \mathcal{P}_R$ by 
  $$\mathrm{Det}_R(P^{\bullet}):=\boxtimes_{i\in \mathbb{Z}}\mathrm{Det}_R(P^i)^{(-1)^i}.$$ 
  For a short exact sequence 
  $0\rightarrow P_1^{\bullet}\rightarrow P_2^{\bullet}\rightarrow P_3^{\bullet}\rightarrow 0$ of bounded complexes in $\bold{P}_{\mathrm{fg}}(R)$, we define a 
  canonical isomorphism 
  \begin{equation}\label{17e}
  \mathrm{Det}_R(P_1^{\bullet})\boxtimes \mathrm{Det}_R(P_3^{\bullet})\isom \mathrm{Det}_R(P_2^{\bullet})
  \end{equation}
  by applying the isomorphism (\ref{17d}) to each exact sequence 
   $0\rightarrow P^i_1\rightarrow P^i_2\rightarrow P^i_3\rightarrow 0$. Moreover, if $P^{\bullet}$ is an 
   acyclic bounded complex in $\bold{P}_{\mathrm{fg}}(R)$, we can define a canonical isomorphism
   \begin{equation}\label{acyclic}
   h_{P^{\bullet}}:\mathrm{Det}_R(P^{\bullet})\isom \bold{1}_R,
   \end{equation}
   which is characterized by the following properties:
   when $P^{\bullet}:=[P^i\xrightarrow{f} P^{i+1}]$ is concentrated in degree $[i, i+1]$, we define this as the composite
  \begin{multline*}
  \mathrm{Det}_R(P^{\bullet})=\mathrm{Det}_R(P^{i})\boxtimes \mathrm{Det}_R(P^{i+1})^{-1}
  \\
  \xrightarrow{\mathrm{Det}(f)\boxtimes \mathrm{id}}\mathrm{Det}_R(P^{i+1})\boxtimes 
  \mathrm{Det}_R(P^{i+1})^{-1}\xrightarrow{\delta_{\mathrm{Det}_R(P^{i+1})}}\bold{1}_R
  \end{multline*}
  when $i$ is even (when $i$ is odd, we similarly define it using $f^{-1}:P^{i+1}\isom P^i$), and for a short exact sequence 
  $0\rightarrow P_1^{\bullet}\rightarrow P_2^{\bullet}\rightarrow P_3^{\bullet}\rightarrow 0$ of 
  acyclic bounded complexes of $\bold{P}_{\mathrm{fg}}(R)$, we have the following commutative 
  diagram 
   \begin{equation*}
 \begin{CD}
 \mathrm{Det}_R(P_1^{\bullet})\boxtimes \mathrm{Det}_R(P_3^{\bullet})@> \isom >> \mathrm{Det}_R(P_2^{\bullet}) \\
 @VVh_{P_1^{\bullet}}\boxtimes h_{P_3^{\bullet}}V @VV h_{P_2^{\bullet}}V \\
  \bold{1}_R\boxtimes \bold{1}_R@ > = >> \bold{1}_R.
  \end{CD}
  \end{equation*}
  The theory of determinants of \cite{KM76} enables us to uniquely (up to canonical isomorphism) extend $\mathrm{Det}_R(-)$ to a functor
  $$\mathrm{Det}_R:(\bold{D}^b_{\mathrm{perf}}(R),\mathrm{is})\rightarrow \mathcal{P}_R$$
  such that the isomorphism (\ref{17e}) extends to the following situation: 
  for any exact sequence $0\rightarrow P_1^{\bullet}
  \rightarrow P_2^{\bullet}\rightarrow P_3^{\bullet}\rightarrow 0$ of complexes of $R$-modules such that each $P_i^{\bullet}$ is quasi-isomorphic 
  to a bounded complex in $\bold{P}_{\mathrm{fg}}(R)$, there exists a canonical isomorphism 
   \begin{equation}
    \mathrm{Det}_R(P^{\bullet}_1)\boxtimes \mathrm{Det}_R(P^{\bullet}_3)\isom  \mathrm{Det}_R(P^{\bullet}_2).
   \end{equation}
  By this property, if 
  $P^{\bullet}\in \bold{D}^b_{\mathrm{perf}}(R)$ satisfies that $\mathrm{H}^i(P^{\bullet})[0]\in \bold{D}^b_{\mathrm{perf}}(R)$ for 
  any $i$, there exists a canonical isomorphism 
  $$\mathrm{Det}_R(P^{\bullet})\isom \boxtimes_{i\in \mathbb{Z}}\mathrm{Det}_R(
  \mathrm{H}^i(P^{\bullet})[0])^{(-1)^i}.$$

  For $(L,r)\in \mathcal{P}_R$, define $(L, r)^{\vee}:=(L^{\vee}, r)\in \mathcal{P}_R$, which induces an anti equivalnce $(-)^{\vee}:\mathcal{P}_R\isom \mathcal{P}_R$.
  For $P\in \bold{P}_{\mathrm{fg}}(R)$, then we have a canonical 
  isomorphism $\mathrm{Det}_R(P^{\vee})\isom \mathrm{Det}_R(P)^{\vee}$ defined by the isomorphism 
  \begin{multline*}
  \mathrm{det}_R(P^{\vee})\isom (\mathrm{det}_RP)^{\vee}\\
  f_1\wedge\cdots \wedge f_r\mapsto [x_1\wedge\cdots \wedge x_r\mapsto \sum_{\sigma\in \mathfrak{S}_r}
  \mathrm{sgn}(\sigma)f_1(x_{\sigma(1)})\cdots f_r(x_{\sigma(r)})].
  \end{multline*}
  This naturally extends to $(\bold{D}_{\mathrm{perf}}^b(R), \mathrm{is})$, i.e. for any $P^{\bullet}\in \bold{D}^b_{\mathrm{perf}}(R)$, 
  there exists a canonical isomorphism 
  \begin{equation}\label{20a}
\mathrm{Det}_R(\bold{R}\mathrm{Hom}_R(P^{\bullet}, R))\isom \mathrm{Det}_R(P^{\bullet})^{\vee}. \end{equation}

\subsection{Fundamental lines}

Both Kato's conjecture and ours concern with the existence of a compatible family of 
canonical trivialization of some graded invertible modules defined by using the determinants of the Galois cohomologies of 
Galois representations or $(\varphi,\Gamma)$-modules. We call these graded invertible modules the fundamental lines, 
of which we explain in this subsection.

Kato's conjecture concerns with pairs $(\Lambda,T)$ such that
\begin{itemize}
\item[(i)]
$\Lambda$ is a noetherian semi-local ring which is complete with respect to 
the $\mathfrak{m}_{\Lambda}$-adic topology (where 
$\mathfrak{m}_{\Lambda}$ is the Jacobson radial of $\Lambda$) such 
that $\Lambda/\mathfrak{m}_{\Lambda}$ is a finite ring 
with the order a power of $p$,
\item[(ii)]$T$ is a $\Lambda$-representation of $G_{\mathbb{Q}_p}$, 
i.e. a finite projective $\Lambda$-module equipped with a continuous $\Lambda$-linear action of 
 $G_{\mathbb{Q}_p}$.
\end{itemize}

Our conjecture conjecture concerns with pairs $(A,M)$ such that
\begin{itemize}
\item[(i)]$A$ is a $\mathbb{Q}_p$-affinoid algebra,
\item[(ii)]$M$ is a $(\varphi,\Gamma)$-module over $\mathcal{R}_A$.

\end{itemize}

For each pair $(B, N)=(\Lambda,T) $ or $(A, M)$ as above, we'll define graded invertible 
$\Lambda$-modules $\Delta_{B,i}(N)\in \mathcal{P}_B$ for $i=1,2$ as below, and the fundamental line 
will be defined as $\Delta_{B}(N):=\Delta_{B,1}(N)\boxtimes\Delta_{B,2}(N)\in \mathcal{P}_B$.

We first define $\Delta_{\Lambda,i}(T)$ for $(\Lambda,T)$. Denote by $C^{\bullet}_{\mathrm{cont}}(G_{\mathbb{Q}_p}, T)$ for the complex of continuous cochains of $G_{\mathbb{Q}_p}$ with values in $T$. It is known that 
$C^{\bullet}_{\mathrm{cont}}(G_{\mathbb{Q}_p}, T)\in \bold{D}^{-}(\Lambda)$ is contained in $\bold{D}^{b}_{\mathrm{perf}}(\Lambda)$ and that satisfies 
the similar properties (1), (2), (3), (4) 
in Theorem \ref{2.16}. In particular, we can define a graded invertible 
$\Lambda$-module 
$$\Delta_{\Lambda,1}(T):=\mathrm{Det}_{\Lambda}(C^{\bullet}_{\mathrm{cont}}(G_{\mathbb{Q}_p}, T)),$$ 
( whose degree is $-r_T:=-\mathrm{rk}_{\Lambda}T$ by the Euler-Poincar\'e formula) which satisfies the following properties:
\begin{itemize}
\item[(i)]For each continuous homomorphism 
$f:\Lambda\rightarrow \Lambda'$, there exists a canonical $\Lambda'$-linear isomorphism
$$\Delta_{\Lambda,1}(T)\otimes_{\Lambda}\Lambda'\isom \Delta_{\Lambda',1}(T\otimes_{\Lambda}\Lambda'),$$
\item[(ii)]For each exact sequence 
$0\rightarrow T_1\rightarrow T_2\rightarrow T_3\rightarrow 0$ of 
$\Lambda$-representations of $G_{\mathbb{Q}_p}$, there exists a canonical 
$\Lambda$-linear isomorphism 
$$\Delta_{\Lambda,1}(T_1)\boxtimes
\Delta_{\Lambda,1}(T_3)\isom\Delta_{\Lambda,1}(T_2),$$
\item[(iii)]The Tate duality $C^{\bullet}_{\mathrm{cont}}(G_{\mathbb{Q}_p},T)\isom 
\bold{R}\mathrm{Hom}_{\Lambda}(C^{\bullet}_{\mathrm{cont}}(G_{\mathbb{Q}_p},T^*),\Lambda)[-2]$  and the isomorphism 
(\ref{20a}) induce a canonical $\Lambda$-linear isomorphism $$\Delta_{\Lambda,1}(T)\isom \Delta_{\Lambda,1}(T^*)^{\vee}.$$

\end{itemize}

We next define $\Delta_{\Lambda,2}(T)$ as follows. 
For $a\in \Lambda^{\times}$, we define 
$$\Lambda_a:=\{x\in W(\overline{\mathbb{F}}_p)\widehat{\otimes}_{\mathbb{Z}_p}\Lambda| (\varphi\otimes \mathrm{id}_{\Lambda})(x)=(1\otimes a)x\},$$ which is an invertible $\Lambda$-module. In the same way as in Theorem \ref{2.9}, for any rank one $\Lambda$-representation $T_0$, there exists unique (up to isomorphism) pair $(\delta_{T_0}, \mathcal{L}_{T_0})$ where $\delta_{T_0}:\mathbb{Q}_p^{\times}\rightarrow 
\Lambda^{\times}$ is a continuous homomorphism and 
$\mathcal{L}_{T_0}$ is an invertible $\Lambda$-module such that
$T_0\isom \Lambda(\tilde{\delta}_{T_0})\otimes_{\Lambda}\mathcal{L}_{T_0}$,  where 
we denote by $\tilde{\delta}_{T_0}:G_{\mathbb{Q}_p}^{\mathrm{ab}}\rightarrow \Lambda^{\times}$ 
for  the continuous character which satisfies $\tilde{\delta}_{T_0}\circ \mathrm{rec}_{\mathbb{Q}_p}=
\delta_{T_0}$. Under these definitions, we define $a(T):=\delta_{\mathrm{det}_{\Lambda}T}(p)\in \Lambda^{\times}$, and define an invertible $\Lambda$-module 
$$\mathcal{L}_{\Lambda}(T):=\Lambda_{a(T)}\otimes_{\Lambda}\mathrm{det}_{\Lambda}T$$ and define a graded 
invertible $\Lambda$-module
$$\Delta_{\Lambda,2}(T):=(\mathcal{L}_{\Lambda}(T), 
r_T).$$
Since we have a canonical isomorphism $\Lambda_{a_1}\otimes_{\Lambda}\Lambda_{a_2}\isom \Lambda_{a_1a_2}:
x\otimes y\mapsto xy$ for any $a_1,a_2\in \Lambda$, $\Delta_{\Lambda,2}(T)$
also satisfies the similar properties:
\begin{itemize}
\item[(i)]For $f:\Lambda\rightarrow \Lambda'$, there exists a canonical isomorphism
$\Delta_{\Lambda,2}(T)\otimes_{\Lambda}\Lambda'\isom \Delta_{\Lambda,2}(T\otimes_{\Lambda}
\Lambda')$, 
\item[(ii)]For $0\rightarrow T_1\rightarrow T_2\rightarrow 
T_3\rightarrow 0$, there exists a canonical isomorphism 
$\Delta_{\Lambda,2}(T_1)\boxtimes
\Delta_{\Lambda,2}(T_3)\isom \Delta_{\Lambda,2}(T_2)$,
\item[(iii)]Let $r_T$ be the rank of $T$, then there exists a canonical isomorphism 
$$\Delta_{\Lambda,2}(T)\isom \Delta_{\Lambda,2}(T^*)^{\vee}\boxtimes (\Lambda(r_T),0)
$$ which is induced by the product of the isomorphisms
$\Lambda_{\delta_{\mathrm{det}_{\Lambda}T}(p)}\isom (\Lambda_{\delta_{\mathrm{det}_{\Lambda}T^*}(p)})^{\vee}:x\mapsto [y\mapsto y\otimes x]$ (remark that we have 
$\Lambda_{\delta_{\mathrm{det}_{\Lambda}T^*}(p)}\otimes \Lambda_{\delta_{\mathrm{det}_{\Lambda}T}(p)}\isom \Lambda:y\otimes x\mapsto yx$
since we have $\delta_{\mathrm{det}_{\Lambda}T}(p)
=\delta_{\mathrm{det}_{\Lambda}T^*}(p)^{-1}$) and the isomorphism 
$\mathrm{det}_{\Lambda}T\isom\mathrm{det}_{\Lambda}(T^*)^{\vee}\otimes_{\Lambda}\Lambda(r_T)
$ induced by the canonical isomorphism $T\isom (T^*)^{\vee}(1):x\mapsto [y\mapsto y(x)\otimes \bold{e}_{-1}]\otimes \bold{e}_1$.

\end{itemize}

Finally, we define 
$$\Delta_{\Lambda}(T):=\Delta_{\Lambda,1}(T)\boxtimes\Delta_{\Lambda,2}(T)\in \mathcal{P}_B,$$
then $\Delta_{\Lambda}(T)$ also satisfies the similar properties (i), (ii) as those for $\Delta_{\Lambda,i}(T)$ and 
\begin{itemize}
\item[(iii)]there exists a canonical isomorphism 
$$\Delta_{\Lambda}(T)\isom \Delta_{\Lambda}(T^*)^{\vee}\boxtimes(\Lambda(r_T), 0).$$

\end{itemize}

Next, we define the fundamental line $\Delta_{A}(M)$ for 
$(\varphi,\Gamma)$-modules $M$ over $\mathcal{R}_A$. 
Let $A$ be a $\mathbb{Q}_p$-affinoid algebra, and let 
$M$ be a $(\varphi,\Gamma)$-module over $\mathcal{R}_A$. 
By Theorem \ref{2.16} of Kedlaya-Pottharst-Xiao, 
we can define a graded invertible $A$-module 
$$\Delta_{A,1}(M):=\mathrm{Det}_AC^{\bullet}_{\varphi,\gamma}(M)\in \mathcal{P}_A$$
which satisfies the similar properties (i), (ii), (iii) as those for $\Delta_{\Lambda,1}(T)$.
We next define $\Lambda_{A,2}(M)$ as follows. By Theorem \ref{2.9} of Kedlaya-Pottharst-Xiao, 
there exists unique ( up to isomorphism ) pair
$(\delta_{\mathrm{det}_{\mathcal{R}_A}M}, \mathcal{L}_{\mathrm{det}_{\mathcal{R}_A}M})$ where 
$\delta_{\mathrm{det}_{\mathcal{R}_A}M}:\mathbb{Q}_p^{\times}\rightarrow A^{\times}$ is 
a continuous homomorphism and $\mathcal{L}_{\mathrm{det}_{\mathcal{R}_A}M}$ is an invertible 
$A$-module such that $\mathrm{det}_{\mathcal{R}_A}M\isom \mathcal{R}_A(\delta_{\mathrm{det}_{\mathcal{R}_A}M})\otimes_{\mathcal{R}_A}\mathcal{L}_{\mathrm{det}_{\mathcal{R}_A}M}$. Then, we define 
an $A$-module 
$$\mathcal{L}_A(M):=\{x\in \mathrm{det}_{\mathcal{R}_A}M|\varphi(x)=\delta_{\mathrm{det}_{\mathcal{R}_A}M}(p)x, 
\gamma(x)=\delta_{\mathrm{det}_{\mathcal{R}_A}M}(\chi(\gamma))x\,\,(\gamma\in \Gamma)\}$$
which is an invertible $A$ module since it is isomorphic to $\mathcal{L}_{\mathrm{det}_{\mathcal{R}_A}M}$, 
and define a graded invertible $A$-module 
$$\Delta_{A,2}(M):=(\mathcal{L}_A(M),r_M)\in \mathcal{P}_A.$$
 By definition, it is easy to check that $\Delta_{A,2}(M)$ satisfies 
the similar properties (i), (ii), (iii) as those  for $\Delta_{\Lambda,2}(T)$. Finally, we define a 
graded invertible $A$-module $\Delta_{A}(M)$ which we call the fundamental line by 
$$\Delta_{A}(M):=\Delta_{A,1}(M)\boxtimes\Delta_{A,2}(M)\in \mathcal{P}_A,$$
which also satisfies the similar properties (i), (ii), (iii) as those for $\Delta_{\Lambda}(T)$.

More generally, let $X$ be a rigid analytic space over $\mathbb{Q}_p$,  and let 
$M$ be a $(\varphi,\Gamma)$-module over $\mathcal{R}_{X}$. By the base change property (i) of $\Delta_A(M)$, we can also functorially define a graded invertible $\mathcal{O}_X$-module 
$$\Delta_{X}(M)\in\mathcal{P}_{\mathcal{O}_X}$$ on $X$ (we can natural generalize the notion of graded 
invertible modules in this setting) such that there exists a canonical isomorphism 
$$\Gamma(\mathrm{Max}(A), \Delta_X(M))\isom \Delta_A(M|_{\mathrm{Max}(A)})$$ for any affinoid open 
$\mathrm{Max}(A)\subseteq X$.

We next compare Kato's fundamental line $\Delta_{\Lambda}(T)$ 
with our fundamental line $\Delta_A(M)$. Let $f:\Lambda\rightarrow A$ 
be a continuous ring homomorphism, where $\Lambda$ is equipped with $\mathfrak{m}_{\Lambda}$-adic topology and $A$ is equipped with $p$-adic topology.
Let $T$ be a $\Lambda$-representation of $G_{\mathbb{Q}_p}$. 
Let denote by  $M:=\bold{D}_{\mathrm{rig}}(T\otimes_{\Lambda}A)$ for  the $(\varphi,\Gamma)$-module 
over $\mathcal{R}_A$ associated to the $A$-representation $T\otimes_{\Lambda}A$ of $G_{\mathbb{Q}_p}$. By Theorem 2.8 of \cite{Po13a}, there exists  a canonical 
quasi-isomorphism $C^{\bullet}_{\mathrm{cont}}(G_{\mathbb{Q}_p}, T)\otimes^{\bold{L}}_{\Lambda}A\isom 
C^{\bullet}_{\varphi,\gamma}(M)$, and this induces an $A$-linear isomorphism 
$$\Delta_{\Lambda,1}(T)\otimes_{\Lambda}A\isom \Delta_{A,1}(M).$$
We also have the following isomorphism.

\begin{lemma}\label{3.1}
In the above situation, there exists a canonical $A$-linear isomorphism 
$$\Delta_{\Lambda,2}\otimes_{\Lambda}A\isom \Delta_{A,2}(M).$$

\end{lemma}

\begin{proof}

By definition, it suffices to show the lemma when $T$ is rank one. Hence, 
we may assume that $T=\Lambda(\tilde{\delta})\otimes_{\Lambda}\mathcal{L}$ for 
a continuous homomorphism  $\delta:\mathbb{Q}_p^{\times}
\rightarrow \Lambda^{\times}$ and an invertible $\Lambda$-module $\mathcal{L}$ (where 
$\tilde{\delta}$ is the character of $G^{\mathrm{ab}}_{\mathbb{Q}_p}$ such that $\tilde{\delta}\circ \mathrm{rec}_{\mathbb{Q}_p}=\delta$). 
Moreover, since we have a canonical isomorphism 
$$\bold{D}_{\mathrm{rig}}(
(\Lambda(\tilde{\delta})\otimes_{\Lambda}\mathcal{L})\otimes_{\Lambda}A)\isom 
\bold{D}_{\mathrm{rig}}(\Lambda(\tilde{\delta})\otimes_{\Lambda}A)\otimes_{A}
(\mathcal{L}\otimes_{\Lambda}A)$$ by the exactness of $\bold{D}_{\mathrm{rig}}(-)$, it suffices to show the lemma when 
$\mathcal{L}=\Lambda$.

Since the image of $H_{\mathbb{Q}}:=\mathrm{Gal}(\overline{\mathbb{Q}}_p/\mathbb{Q}_{p,\infty})$ in $G_{\mathbb{Q}_p}^{\mathrm{ab}}$ is the closed subgroup 
which is topologically generated by $\mathrm{rec}_{\mathbb{Q}_p}(p)$, we have 
$$\bold{D}_{\mathrm{rig}}(\Lambda(\tilde{\delta})\otimes_{\Lambda}A)=
(W(\overline{\mathbb{F}}_p)\hat{\otimes}_{\mathbb{Z}_p}\Lambda(\tilde{\delta}))^{\mathrm{rec}_{\mathbb{Q}_p}(p)=1}\otimes_{\Lambda}\mathcal{R}_A,$$ 
by definition of 
$\bold{D}_{\mathrm{rig}}(-)$,
and the right hand side is isomorphic to $\mathcal{R}_{A}(f\circ \delta)$. Hence, we obtain
 \[
\begin{array}{ll}
\mathcal{L}_A(M)&=((W(\overline{\mathbb{F}}_p)\hat{\otimes}_{\mathbb{Z}_p}\Lambda(\tilde{\delta}))^{\mathrm{rec}_{\mathbb{Q}_p}(p)=1}\otimes_{\Lambda}\mathcal{R}_A)^{\varphi=f(\delta(p)),\Gamma=f\circ \delta\circ \chi}\\
&=(W(\overline{\mathbb{F}}_p)\hat{\otimes}_{\mathbb{Z}_p}\Lambda(\tilde{\delta}))^{\mathrm{rec}_{\mathbb{Q}_p}(p)=1}\otimes_{\Lambda}A= \mathcal{L}_{\Lambda}(T)\otimes_{\Lambda}A, 
\end{array}
\]
which proves the lemma.
\end{proof}

Taking the products of these two canonical isomorphisms, we obtain the following corollary.

\begin{corollary}\label{3.2}
In the above situation, there exists a canonical isomorphism 
$$\Delta_{\Lambda}(T)\otimes_{\Lambda}A
\isom \Delta_{A}(M).$$

\end{corollary}

\begin{exa}\label{3.3}
The typical example of the above base change property is the following. 
For $\Lambda$ as above, let denote by $X$ for the associated rigid analytic space. 
More precisely, $X$ is the union of affinoids $\mathrm{Max}(A_n)$ for $n\geqq 1$, where 
$A_n$ is the $\mathbb{Q}_p$-affinoid algebra defined by $A_n:=\Lambda[\frac{\mathfrak{m}_{\Lambda}^n}{p}]^{\wedge}[1/p]$ ( for a ring $R$,  denote by $R^{\wedge}$ for the $p$-adic completion). Let $T$ be a $\Lambda$-representation of $G_{\mathbb{Q}_p}$, and let
denote by $M_n:=\bold{D}_{\mathrm{rig}}(T\otimes_{\Lambda}A_n)$. Since 
$M_n$ is compatible with the base change with respect to the canonical map 
$A_n\rightarrow A_{n+1}$ for any $n$, $\{M_n\}_{n\geqq 1}$ defines a $(\varphi,\Gamma)$-module 
$\mathcal{M}$ over $\mathcal{R}_{X}$. Then, the canonical isomorphism 
$\Delta_{\Lambda}(T)\otimes_{\Lambda}A_n\isom \Delta_{A_n}(M_n)$ defined in the above corollary glues to an isomorphism 
$$\Delta_{\Lambda}(T)\otimes_{\Lambda}\mathcal{O}_X\isom \Delta_{X}(\mathcal{M}).$$
Moreover, using the terminology of coadmissible modules (\cite{ST03}), we can define this comparison isomorphism without using sheaves. Let define $A_{\infty}:=\Gamma(X,\mathcal{O}_X)$ and 
$\Delta_{A_{\infty}}(M_{\infty}):=\varprojlim_{n}\Delta_{A_n}(M_n)$. Taking the limit of 
the isomorphism $\Delta_{\Lambda}(T)\otimes_{\Lambda}A_n\isom \Delta_{A_n}(M_n)$ we obtain 
an $A_{\infty}$-linear isomorphism 
$$\Delta_{\Lambda}(T)\otimes_{\Lambda}A_{\infty}\isom \Delta_{A_{\infty}}(M_{\infty}).$$
Then, the theory of coadmissible modules (Corollary 3.3 of \cite{ST03}) says that to consider the isomorphism 
$\Delta_{\Lambda}(T)\otimes_{\Lambda}\mathcal{O}_X\isom \Delta_{X}(\mathcal{M})$ is the same 
 as to consider the isomorphism $\Delta_{\Lambda}(T)\otimes_{\Lambda}A_{\infty}\isom \Delta_{A_{\infty}}(M_{\infty}).$ In fact, we will frequently use the latter object $\Delta_{A_{\infty}}(M_{\infty})$ in \S4.

\end{exa}

\subsection{de Rham $\varepsilon$-isomorphism}
In this subsection, we assume that $L=A$ is a  finite extension of 
$\mathbb{Q}_p$. We define a trivialization
$$\varepsilon^{\mathrm{dR}}_{L,\zeta}(M):\bold{1}_L\isom\Delta_L(M)$$
which we call the de Rham $\varepsilon$-isomorphism
 for each de Rham 
$(\varphi,\Gamma)$-module $M$ over $\mathcal{R}_L$ and for each 
$\mathbb{Z}_p$-basis $\zeta=\{\zeta_{p^n}\}_{n\geqq 0}$ of $\mathbb{Z}_p(1)$.

Let $M$ be a de Rham $(\varphi,\Gamma)$-module over $\mathcal{R}_L$. 
We first recall the definition of Deligne-Langlands' and Fontaine-Perrin-Riou's 
$\varepsilon$-constant associated to $M$ (\cite{De73}, \cite{FP94}). 

We first briefly  recall the theory of $\varepsilon$-constants of Deligne and Langlands (\cite{De73}).
Let $W_{\mathbb{Q}_p}\subseteq G_{\mathbb{Q}_p}$ be the Weil group of $\mathbb{Q}_p$. 
Let $E$ be a field of characteristic zero, and let $V=(V, \rho)$ be an $E$ representation of $W_{\mathbb{Q}_p}$, i.e. $V$ is a finite dimensional $E$-vector space equipped with a smooth $E$-linear action $\rho$ of $W_{\mathbb{Q}_p}$. Let denote by $V^{\vee}$ for the dual 
$(\mathrm{Hom}_E(V,E),\rho^{\vee})$ of $V$. Denote by $E(|x|)$ the rank one $E$-representation 
of $W_{\mathbb{Q}_p}$ corresponding to the continuos homomorphism 
$|x|:\mathbb{Q}_p^{\times}
\rightarrow E^{\times}:p\mapsto 1/p, a\mapsto 1 (a\in \mathbb{Z}_p^{\times})$ via the local class field theory. 
Put $V^{\vee}(|x|):=V^{\vee}\otimes_E E(|x|)$. Assume that $E$ is a field which contains 
$\mathbb{Q}(\zeta_{p^{\infty}})$. The definition of the $\varepsilon$-constants depends on 
the choice of an additive character of $\mathbb{Q}_p$ and a Haar measure on $\mathbb{Q}_p$. In this article, we fix the Haar measure $dx$ on $\mathbb{Q}_p$ for which 
$\mathbb{Z}_p$ has measure $1$. For each $\mathbb{Z}_p$-basis
$\zeta=\{\zeta_{p^n}\}_{n\geqq 0}$ of $\mathbb{Z}_p(1)$, we define an additive character 
$\psi_{\zeta}:\mathbb{Q}_p\rightarrow E^{\times}$ such that 
 $\psi_{\zeta}(1/p^n):=\zeta_{p^n}$ for $n\geqq 1$. In this article, 
 we don't recall the precise definition of $\varepsilon$-constants, but we recall here some of their 
 basic properties under the fixed additive character $\psi_{\zeta}$ 
 and the fixed Haar measure $dx$. Under these fixed datum, we can attach a constant 
$\varepsilon(V, \psi_{\zeta}, dx)\in E^{\times}$ 
for each $V$ as above which satisfies the following  properties 
(we denote $\varepsilon(V,\zeta):=\varepsilon(V, \psi_{\zeta}, dx)$ for simplicity):
\begin{itemize}
\item[(1)]For each exact sequence $0\rightarrow V_1\rightarrow V_2\rightarrow V_3\rightarrow 0$ of 
finite dimensional $E$-vector spaces with continuous actions of $W_{\mathbb{Q}_p}$, we have 
$$\varepsilon(V_2,\zeta)=\varepsilon(V_1,\zeta)\varepsilon(V_3,\zeta).$$
\item[(2)]For each $a\in \mathbb{Z}_p^{\times}$, we define $\zeta^a:=\{\zeta_{p^n}^a\}_{n\geqq 1}$. 
Then, we have 
$$\varepsilon(V, \zeta^a)=\mathrm{det}_EV(\mathrm{rec}_{\mathbb{Q}_p}(a))\varepsilon(V,\zeta).$$
\item[(3)]$\varepsilon(V,\zeta)\varepsilon(V^{\vee}(|x|),\zeta^{-1})=1.$ 
\item[(4)]$\varepsilon(V,\zeta)=1$ if $V$ is unramified. 
\item[(5)]If $\mathrm{dim}_EV=1$ and corresponds to a locally constant homomorphism 
$\delta:\mathbb{Q}^{\times}_p\rightarrow E^{\times}$ via the local class field theory, then 
$$\varepsilon(V,\zeta)=\delta(p)^{n(\delta)}(\sum_{i\in (\mathbb{Z}/p^{n(\delta)}\mathbb{Z})^{\times}}
\delta(i)^{-1}\zeta_{p^{n(\delta)}}^i), $$
where $n(\delta)\geqq 0$ is the conductor of $\delta$, i.e. the minimal integer $n\geqq 0$ such that $\delta|_{(1+p^n\mathbb{Z}_p)
\cap\mathbb{Z}_p^{\times}}=1$ (then $\delta|_{\mathbb{Z}_p^{\times}}$ factors through $ (\mathbb{Z}/p^{n(\delta)}\mathbb{Z})^{\times}$).
\end{itemize}
For a Weil-Deligne  representation $W=(V,\rho, N)$ of $W_{\mathbb{Q}_p}$ defined over $E$, we set 
$$\varepsilon(W,\zeta):=\varepsilon((V,\rho), \zeta)\cdot \mathrm{det}_E(-\mathrm{Fr}_p|V^{L_p}/(V^{N=0})^{I_p}),$$ which also satisfies 
$$\varepsilon(W, \zeta)\cdot \varepsilon(W^{\vee}(|x|),\zeta^{-1})=1.$$

Next, we define the $\varepsilon$-constant for each de Rham $(\varphi,\Gamma)$-module 
over $\mathcal{R}_L$ following Fontaine-Perrin-Riou (\cite{FP94}). Let $M$ be a de Rham $(\varphi,\Gamma)$-module 
over $\mathcal{R}_L$. Then $M$ is potentially semi-stable by the result of Berger (for example, see Th\'eor\`eme III.2.4 of \cite{Ber08b}) which is based on the Crew's conjecture proved by Andr\'e, Mebkhout, Kedlaya. Hence, we can define a filtered $(\varphi,N,G_{\mathbb{Q}_p})$-module 
$\bold{D}_{\mathrm{pst}}(M):=\cup_{K\subseteq \overline{\mathbb{Q}}_p}\bold{D}^K_{\mathrm{st}}(M|_K)$ which is a free $\mathbb{Q}_p^{\mathrm{ur}}\otimes_{\mathbb{Q}_p}L$-module whose rank is $r_M$, where $K$ run through all the finite extensions of 
$\mathbb{Q}_p$ and we define $\bold{D}^K_{\mathrm{st}}(M|_K):=(\mathcal{R}_L(\pi_K)[\mathrm{log}(\pi),1/t]\otimes_{\mathcal{R}_{L}}M)^{\Gamma_K=1}$. Set $\bold{D}_{\mathrm{st}}(M):=\bold{D}_{\mathrm{st}}^{\mathbb{Q}_p}(M)$. 
Following Fontaine, one can define a Weil-Deligne representation $W(M):=(W(M)
,\rho,N)$ of $W_{\mathbb{Q}_p}$ defined over $\mathbb{Q}_p^{\mathrm{ur}}\otimes_{\mathbb{Q}_p}L$ such that $N$ is the natural one and $\rho(g)(x):=\varphi^{v(g)}(g\cdot x)$ for  
$g\in W_{\mathbb{Q}_p}$ and $x\in W(M)$, where we denote by $g\cdot x$ 
for the natural action of $G_{\mathbb{Q}_p}$ on $W(M)$ and 
$v:W_{\mathbb{Q}_p}\twoheadrightarrow W^{\mathrm{ab}}_{\mathbb{Q}_p}\xrightarrow{\mathrm{rec}_{\mathbb{Q}_p}^{-1}}\mathbb{Q}_p^{\times}\xrightarrow{v_p} \mathbb{Z}$. Taking the base change of $W(M)$ by 
the natural inclusion $\mathbb{Q}_p^{\mathrm{ur}}\otimes_{\mathbb{Q}_p}L\hookrightarrow 
\mathbb{Q}_p^{\mathrm{ab}}\otimes_{\mathbb{Q}_p}L$, and decomposing $\mathbb{Q}_p^{\mathrm{ab}}\otimes_{\mathbb{Q}_p}L\isom \prod_{\tau}L_{\tau}$ into a finite product of 
 fields $L_{\tau}$. We obtain a Weil-Deligne representation $W(M)_{\tau}$ of $W_{\mathbb{Q}_p}$  
 defined over $L_{\tau}$ for each $\tau$. Hence, we can define 
 the $\varepsilon$-constant $\varepsilon(W(M)_{\tau}, \tau(\zeta))\in L^{\times}_{\tau}$, where $\tau(\zeta)$ is the image of $\zeta$ in $L_{\tau}$ by the projection 
 $\mathbb{Q}_p^{\mathrm{ab}}\otimes_{\mathbb{Q}_p}L\rightarrow L_{\tau}$. Then, the product
 $$\varepsilon_L(W(M),\zeta):=(\varepsilon(W(M)_{\tau}, \tau(\zeta)))_{\tau}\in \prod_{\tau}L_{\tau}^{\times}$$ is contained in $L_{\infty}^{\times}:=(\mathbb{Q}_p(\zeta_{p^{\infty}})\otimes_{\mathbb{Q}_p}L)^{\times}
 \subseteq (\mathbb{Q}_p(\zeta_{p^{\infty}})\otimes_{\mathbb{Q}_p}\mathbb{Q}_p^{ur}\otimes_{\mathbb{Q}_p}L)^{\times}$ since it is easy to check that $\varepsilon_L(W(M),\zeta)$ is fixed by $1\otimes \varphi\otimes 1$.


Using this definition, for each de Rham $(\varphi,\Gamma)$-module $M$ over 
$\mathcal{R}_L$, we construct a trivialization
$\varepsilon^{\mathrm{dR}}_{L,\zeta}(M):\bold{1}_L\isom\Delta_L(M)$ as follows. 
We will first define two isomorphisms
 $$\theta_L(M):\bold{1}_L\isom \Delta_{L,1}(M)\boxtimes\mathrm{Det}_L(\bold{D}_{\mathrm{dR}}(M))$$ and 
 $$\theta_{\mathrm{dR},L}(M,\zeta):\mathrm{Det}_L(\bold{D}_{\mathrm{dR}}(M))\isom\Delta_{L,2}(M)
 $$ (we remark that $\theta_{\mathrm{dR},L}(M,\zeta)$ depends on the choice of $\zeta$), and then define $\varepsilon_{L,\xi}^{\mathrm{dR}}(M)$ as the following composites
  \begin{multline*}
  \varepsilon_{L,\xi}^{\mathrm{dR}}(M):\bold{1}_L\xrightarrow{\Gamma_L(M)\cdot\theta_L(M)}\Delta_{L,1}(M)\boxtimes\mathrm{Det}_L(\bold{D}_{\mathrm{dR}}(M))\\
  \xrightarrow{\mathrm{id}\boxtimes \theta_{\mathrm{dR},L}(M,\zeta)}\Delta_{L,1}(M)
  \boxtimes \Delta_{L,2}(M)=\Delta_L(M),
  \end{multline*}
  where $\Gamma_L(M)\in \mathbb{Q}^{\times}$ is defined by 
  $$\Gamma_L(M):=\prod_{r\in\mathbb{Z}}\Gamma^*(r)^{-\mathrm{dim}_L\mathrm{gr}^{-r}\bold{D}_{\mathrm{dR}}(M)},$$
  where we set 
  $$\Gamma^*(r):=\begin{cases} (r-1)! &  (r\geqq 1)\\
  \frac{(-1)^r}{(-r)!} & (r\leqq 0)\end{cases}.$$

We first define $\theta_L(M):\bold{1}_L\isom\Delta_{L,1}(M)\boxtimes\mathrm{Det}_L(\bold{D}_{\mathrm{dR}}(M))$. 
By the result of \S 2.2, we have the following exact sequence of $L$-vector spaces
\begin{equation}\label{exact}
0\rightarrow \mathrm{H}^0_{\varphi,\gamma}(M_0)\rightarrow \bold{D}_{\mathrm{cris}}(M_0)_1
\xrightarrow{x\mapsto ((1-\varphi)x, \overline{x})} \bold{D}_{\mathrm{cris}}(M_0)_2\oplus t_{M_0}\xrightarrow{
\mathrm{exp}_{f,M_0}\oplus \mathrm{exp}_{M_0}} \mathrm{H}^1_{\varphi,\gamma}(M_0)_{f}\rightarrow 0
\end{equation}
 for $M_0=M, M^*$, where we denote $\bold{D}_{\mathrm{cris}}(M_0)_i=\bold{D}_{\mathrm{cris}}(M_0)$ for $i=1,2$.

Using Tate duality and the 
de Rham duality 
$$\bold{D}_{\mathrm{dR}}(M)^0\isom t_{M^*}^{\vee}:x\mapsto [\overline{y}\mapsto [y,x]_{\mathrm{dR}}]$$ (here $y\in \bold{D}_{\mathrm{dR}}(M^*)$ is a lift of $\bar{y}$) and Proposition \ref{2.24}, we define 
a map 
\begin{multline*}
\mathrm{exp}^*_{M^*}:\mathrm{H}^1_{\varphi,\gamma}(M)_{/f}:=\mathrm{H}^1_{\varphi,\gamma}(M)/\mathrm{H}^1_{\varphi,\gamma}(M)_f\xrightarrow{x\mapsto [y\mapsto \langle y,x\rangle]} 
\mathrm{H}^1_{\varphi,\gamma}(M^*)_f^{\vee}\\
\xrightarrow{\mathrm{exp}_{M^*}^{\vee}} t_{M^*}^{\vee}\isom \bold{D}_{\mathrm{dR}}(M)^0
\end{multline*}
which is called the dual exponential map and was studied in \S2.4 of \cite{Na14a}. 
Using this map, as the dual of the exact sequence $(\ref{exact})$ for $M_0=M^*$, we obtain 
an exact sequence
\begin{equation}\label{18e}
0\rightarrow \mathrm{H}^1_{\varphi,\gamma}(M)_{/f}\xrightarrow{\mathrm{exp}_{f,M^*}^{\vee}
\oplus\mathrm{exp}^*_{M^*}}
\bold{D}_{\mathrm{cris}}(M^*)_2^{\vee}\oplus \bold{D}_{\mathrm{dR}}(M)^0\xrightarrow{(*)} \bold{D}_{\mathrm{cris}}(M^*)_1^{\vee}
\rightarrow \mathrm{H}^2_{\varphi,\gamma}(M)\rightarrow 0,
\end{equation}
where the map $\bold{D}_{\mathrm{cris}}(M^*)_2^{\vee}\rightarrow \bold{D}_{\mathrm{cris}}(M^*)_1^{\vee}$ in $(*)$ is the dual of $(1-\varphi)$.
 Therefore, as the composite of the exact sequences (\ref{exact}) for $M_0=M$ and (\ref{18e}), we obtain the following exact sequence 
 \begin{multline}\label{exact2}
 0\rightarrow \mathrm{H}^0_{\varphi,\gamma}(M)\rightarrow \bold{D}_{\mathrm{cris}}(M)_1\xrightarrow{x\mapsto ((1-\varphi)x, \bar{x})} 
 \bold{D}_{\mathrm{cris}}(M)_2\oplus t_M\rightarrow \mathrm{H}^1_{\varphi,\gamma}(M)\\
 \rightarrow \bold{D}_{\mathrm{cris}}(M^*)_2^{\vee}\oplus \bold{D}_{\mathrm{dR}}(M)^0\rightarrow 
 \bold{D}_{\mathrm{cris}}(M^*)^{\vee}_1\rightarrow \mathrm{H}_{\varphi,\gamma}^2(M)\rightarrow 0
 \end{multline}
 Applying the trivialization (\ref{acyclic}) to this exact sequence and using the canonical isomorphisms 
 $i_{\mathrm{Det}_L(\bold{D}_{\mathrm{cris}}(M)_1)}:
\mathrm{Det}_L(\bold{D}_{\mathrm{cris}}(M)_2)\boxtimes\mathrm{Det}_L(\bold{D}_{\mathrm{cris}}(M)_1)^{-1}\isom \bold{1}_L$ and 
$i_{\mathrm{Det}_L(\bold{D}_{\mathrm{cris}}(M^*)^{\vee}_1)}:
\mathrm{Det}_L(\bold{D}_{\mathrm{cris}}(M^*)_2^{\vee})\boxtimes
\mathrm{Det}_L(\bold{D}_{\mathrm{cris}}(M^*)_1^{\vee})^{-1}\isom \bold{1}_L$
and $\mathrm{Det}_L(\bold{D}^0_{\mathrm{dR}}(M))\boxtimes \mathrm{Det}_L(t_M)\isom \mathrm{Det}_L(\bold{D}_{\mathrm{dR}}(M))$, we obtain a canonical 
 isomorphism 
 
 $$\theta_L(M):\bold{1}_L\isom \Delta_{L,1}(M)\boxtimes \mathrm{Det}_L(\bold{D}_{\mathrm{dR}}(M)).$$

Next, we define an isomorphism 
$\theta_{\mathrm{dR},L}(M,\zeta):
\mathrm{Det}_L(\bold{D}_{\mathrm{dR}}(M))\isom \Delta_{L,2}(M)$. 
To define this, we show the following lemma.
\begin{lemma}\label{3.4}
Let $\{h_1,h_2,\cdots,h_{r_M}\}$ be the set of Hodge-Tate weights of $M$ $($with multiplicity$)$. Put 
$h_M:=\sum_{i=1}^{r_M}h_i$. For any $n\geqq n(M)$ such that $\varepsilon_L(W(M),\zeta)\in L_n:=\mathbb{Q}_p(\zeta_{p^n})\otimes_{\mathbb{Q}_p}L$, the map
\begin{multline*}
\mathcal{L}_{L}(M)\rightarrow \bold{D}_{\mathrm{dif},n}(\mathrm{det}_{\mathcal{R}_L}M)=L_n((t))\otimes_{\iota_n, \mathcal{R}^{(n)}_L}(\mathrm{det}_{\mathcal{R}_L}M)^{(n)}: \\
x\mapsto \frac{1}{\varepsilon_L(W(M),\zeta)}\cdot\frac{1}{t^{h_M}}\otimes\varphi^n(x)
\end{multline*}

induces an isomorphism 
$$f_{M,\zeta}:\mathcal{L}_L(M)\isom \bold{D}_{\mathrm{dR}}(\mathrm{det}_{\mathcal{R}_L}M),$$ 
and doesn't depend on the choice of $n$.
\end{lemma}
\begin{proof}
The independence of $n$ follows from the definition of the transition map 
$\bold{D}_{\mathrm{dif},n}(-)\hookrightarrow \bold{D}_{\mathrm{dif},n+1}(-)$. 

We show that $f_{M,\zeta}$ is isomorphism. Comparing the dimensions, it suffices to show that the image of the map in the lemma is contained in $\bold{D}_{\mathrm{dR}}(\mathrm{det}_{\mathcal{R}_L}M)$, i.e. is fixed by  
the action of $\Gamma$.
Since we have $\varepsilon_L(W(M),\zeta)/\varepsilon_L(
W(\mathrm{det}_{\mathcal{R}_L}M),\zeta)\in L^{\times}(\subseteq L_{\infty}^{\times}),$ 
it suffices to show the claim when $M$ is of rank one. We assume that $M$ is of rank one.  
By the classification of rank one de Rham 
$(\varphi,\Gamma)$-modules, there exists a locally constant homomorphism 
$\tilde{\delta}:\mathbb{Q}_p^{\times}\rightarrow L^{\times}$ such that $M\isom \mathcal{R}_L(\tilde{\delta}\cdot x^{h_M})$. The corresponding representation $W(M)$ 
of $W_{\mathbb{Q}_p}$ 
is given by the homomorphism $\tilde{\delta}\cdot|x|^{h_M}:\mathbb{Q}_p^{\times}\rightarrow L^{\times}
$ via the local class field theory. By the property (2) of $\varepsilon$-constants, then 
we have $\gamma(\varepsilon_L(\bold{D}_{\mathrm{pst}}(M),\zeta))=\tilde{\delta}(\chi(\gamma))\varepsilon_L(W(M),\zeta)$ 
for $\gamma\in \Gamma$, which proves the claim since we have $\gamma(\varphi^n(x))=\tilde{\delta}(\chi(\gamma))\chi(\gamma)^{h_M}\varphi^n(x)$ for $x\in \mathcal{L}_L(M), \gamma\in \Gamma$ by definition.

\end{proof}

Since we have a canonical isomorphism $\bold{D}_{\mathrm{dR}}(\mathrm{det}_{\mathcal{R}_L}M)
\isom \mathrm{det}_L\bold{D}_{\mathrm{dR}}(M)$, the isomorphism $f_{M,\zeta}$ induces an isomorphism 
$f_{M,\zeta}:\Delta_{L,2}(M)\isom \mathrm{Det}_L(\bold{D}_{\mathrm{dR}}(M))$. We define 
the isomorphism $\theta_{\mathrm{dR},L}(M,\zeta)$ as the inverse
$$\theta_{\mathrm{dR},L}(M,\zeta):=f_{M,\zeta}^{-1}: \mathrm{Det}_L(\bold{D}_{\mathrm{dR}}(M))\isom
\Delta_{L,2}(M).$$

\begin{rem}\label{3.5}
The isomorphism $f_{M,\zeta}$, and hence the isomorphism $\theta_{\mathrm{dR},L}(M,\zeta)$ depend on the choice of $\zeta$. If we choose another $\mathbb{Z}_p$-basis of $\mathbb{Z}_p(1)$ which can be written as 
$\zeta^a:=\{\zeta^a_{p^n}\}_{n\geqq 0}$ for unique $a\in \mathbb{Z}_p^{\times}$, then $f_{M,\zeta^a}$ is defined using $\varepsilon_{L}(W(M),\zeta^a)$ and the parameter 
$\pi_{\zeta^a}$ (see remark \ref{2.1}) and $t_{a}:=\mathrm{log}(1+\pi_{\zeta^a})$. Since we have 
$\varepsilon_{L}(W(M),\zeta^a)=\mathrm{det}W(M)(\mathrm{rec}_{\mathbb{Q}_p}(a))
\varepsilon_{L}(W(M),\zeta)$ and $\pi_{\zeta^a}=(1+\pi)^a-1$ and 
$t_a=at$, we have $f_{M,\zeta^a}=\frac{1}{\delta_{\mathrm{det}_{\mathcal{R}_L}M}(a)}f_{M,\zeta}$, and hence we also have 
$$\theta_{\mathrm{dR},L}(M,\zeta^a)=\delta_{\mathrm{det}_{\mathcal{R}_L}M}(a)\cdot
\theta_{\mathrm{dR},L}(M,\zeta),$$
and also obtain 
$$\varepsilon^{\mathrm{dR}}_{L,\zeta^a}(M)=\delta_{\mathrm{det}_{\mathcal{R}_L}M}(a)\cdot \varepsilon^{\mathrm{dR}}_{L,\zeta}(M).$$

\end{rem}



\begin{rem}\label{3.7}
In \cite{Ka93b} or \cite{FK06}, Kato and Fukaya-Kato defined their de Rham $\varepsilon$-isomorphism $\varepsilon^{\mathrm{dR}}_{L,\zeta}(V)':\bold{1}_L\isom\Delta_L(V)$ 
(using a different notation) for each de Rham $L$-representation $V$ of $G_{\mathbb{Q}_p}$ 
using the original Bloch-Kato exponential map. Using Proposition \ref{2.27}, we can compare 
our $\varepsilon_{L,\zeta}^{\mathrm{dR}}(\bold{D}_{\mathrm{rig}}(V))$ with their $\varepsilon^{\mathrm{dR}}_{L,\zeta}(V)'$  under the canonical isomorphism $\Delta_{L}(V)\isom \Delta_L(\bold{D}_{\mathrm{rig}}(V))$ defined in 
Corollary \ref{3.2}. We remark that our's and their's are different since they used (in our notation) the 
$\varepsilon$-constant $\varepsilon_{L}((\bold{D}_{\mathrm{pst}}(V), \rho),\zeta)$ associated to the representation $(\bold{D}_{\mathrm{pst}}(V), \rho)$ of $W_{\mathbb{Q}_p}$ instead of $W(V)$. Since one has 
$$\varepsilon_L(W(V),\zeta)=\varepsilon_L((\bold{D}_{\mathrm{pst}}(V), \rho),\zeta)\cdot 
\mathrm{det}_L(-\varphi|\bold{D}_{\mathrm{st}}(V)/\bold{D}_{\mathrm{cris}}(V)),$$
the correct relation between our's and their's are 
\begin{equation}
\varepsilon_{L,\zeta}^{\mathrm{dR}}(\bold{D}_{\mathrm{rig}}(V))=
\mathrm{det}_L(-\varphi|\bold{D}_{\mathrm{st}}(V)/\bold{D}_{\mathrm{cris}}(V))\cdot \varepsilon_{L,\zeta}^{\mathrm{dR}}(V)'.
\end{equation} Moreover, we insist that our's are correct one, since we show in Lemma 
\ref{3.8} below that our $\varepsilon_{L,\zeta}^{\mathrm{dR}}(M)$ is compatible with exact sequence (but $\varepsilon^{\mathrm{dR}}_{L,\zeta}(V)'$ may not satisfy this compatibility).

\end{rem}

Finally in this subsection, we prove a lemma on the compatibility of 
the de Rham $\varepsilon$-isomorphism with exact sequences and the Tate duality.
\begin{lemma}\label{3.8}
\begin{itemize}
\item[(1)]For any exact sequence $0\rightarrow M_1\rightarrow M_2\rightarrow M_3\rightarrow 0$, 
we have 
$$\varepsilon_{L,\zeta}^{\mathrm{dR}}(M_2)=\varepsilon_{L,\zeta}^{\mathrm{dR}}(M_1)\boxtimes \varepsilon_{L,\zeta}^{\mathrm{dR}}(M_3)$$ under the canonical isomorphism 
$\Delta_L(M_2)\isom \Delta_L(M_1)\boxtimes\Delta_{L}(M_3)$.
\item[(2)]
One has the following commutative diagram of isomorphisms 
\begin{equation*}
 \begin{CD}
 \Delta_L(M)@> \mathrm{can} >> \Delta_L(M^*)^{\vee}\boxtimes (L(r_M), 0)\\
 @A \varepsilon_{L,\zeta^{-1}}^{\mathrm{dR}}(M) AA @VV 
 \varepsilon_{L,\zeta}^{\mathrm{dR}}(M^*)^{\vee}\boxtimes[\bold{e}_{r_M}\mapsto 1] V\\
  \bold{1}_L@> \mathrm{can}>> \bold{1}_L\boxtimes\bold{1}_L.
  \end{CD}
  \end{equation*}

\end{itemize}

\end{lemma}
\begin{proof}
We first prove (1). The proof is identical to that of Proposition 3.3.8 of \cite{FK06}, but we give a proof for convenience of the readers. We first remark that we have 
\begin{equation}\label{26.111}
\Gamma_L(M)\cdot \Gamma_L(M^*)=(-1)^{h_M+\mathrm{dim}_Lt_M}
\end{equation}
since we have 
$$\Gamma^*(r)\cdot \Gamma^*(1-r)=\begin{cases} (-1)^{r-1} & (r\geqq 1)\\
 (-1)^r & (r\leqq 0). \end{cases}$$
 
 We next remark that one has the following commutative diagram
 \begin{equation}\label{27.111}
 \begin{CD}
 \bold{1}_L@> \theta_L(M)>> \Delta_{L,1}(M)\boxtimes\mathrm{Det}_L(M)\\
 @V (-1)^{\mathrm{dim}_L t_M} VV @VV \mathrm{can} V\\
  \bold{1}_L@< \theta_L(M^*)^{\vee}<< \Delta_{L,1}(M^*)^{\vee}\boxtimes\mathrm{Det}_L(M^*)^{\vee},
  \end{CD}
  \end{equation}
  where the right vertical arrow is induced by the Tate duality, since one has the following commutative diagram
  
  \begin{equation*}
 \begin{CD}
 t_M @> -\mathrm{exp}_M>> \mathrm{H}^1_{\varphi,\gamma}(M) @> \mathrm{exp}^*_{M^*} >> \bold{D}_{\mathrm{dR}}(M)^0\\
 @V \bar{x}\mapsto [y, \mapsto [y,x]_{\mathrm{dR}}] VV @V x\mapsto [y\mapsto \langle y,x\rangle ]  VV@ VVx\mapsto [\bar{y}\mapsto [y,x]_{\mathrm{dR}}] V\\
  \bold{D}_{\mathrm{dR}}(M^*)^{\vee}@> (\mathrm{exp}^*_M)^{\vee} >> \mathrm{H}^1_{\varphi,\gamma}(M^*)^{\vee}@> (\mathrm{exp}_{M^*})^{\vee}>> (t_{M^*})^{\vee}.
  \end{CD}
  \end{equation*}
  Finally, we remark that one has the following commutative diagram 
   \begin{equation}\label{28.111}
 \begin{CD}
 \mathrm{Det}_L(M)@> \theta_{\mathrm{dR},L}(M,\zeta^{-1})>> \Delta_{L,2}(M)\\
 @V (-1)^{h_M}\cdot\mathrm{can}VV @VV \mathrm{can} V\\
  \mathrm{Det}_L(M^*)^{\vee}= \mathrm{Det}_L(M^*)^{\vee}\boxtimes\bold{1}_L@< \theta_{\mathrm{dR},L}(M^*,\zeta)^{\vee}\boxtimes[\bold{e}_{r_M}\mapsto 1]<< \Delta_{L,2}(M^*)^{\vee}
  \boxtimes (L(r_M),0),
  \end{CD}
  \end{equation}
  where the vertical maps $\mathrm{can}$ are also defined by the duality, since we have 
  $$\varepsilon_{L}(W(V),\zeta^{-1})\cdot \varepsilon_L(W(V^*), \zeta)=1.$$ Then, (1) follows from 
  the commutative diagrams (\ref{26.111}), (\ref{27.111}) and (\ref{28.111}).

We next prove (1). We first define an isomorphism
$$\theta_L(M)':\bold{1}_L\isom \Delta_{L,1}(M)\boxtimes \mathrm{Det}_L(\bold{D}_{\mathrm{dR}}(M))$$ in the same way as $\theta_L(M)$ using the following exact sequence
  \begin{equation}\label{exact}
0\rightarrow \mathrm{H}^0_{\varphi,\gamma}(M)\rightarrow \bold{D}_{\mathrm{cris}}(M)_1
\xrightarrow{x\mapsto ((1-\varphi^{-1})x, \overline{x})} \bold{D}_{\mathrm{cris}}(M_0)_2\oplus t_{M_0}\xrightarrow{
\mathrm{exp}_{f,M}\oplus \mathrm{exp}_{M}} \mathrm{H}^1_{\varphi,\gamma}(M)_{f}\rightarrow 0
\end{equation} (we use $\varphi^{-1}$ instead of $\varphi$) and (\ref{18e}), and define 
$$\theta_{\mathrm{dR},L}(M,\zeta)':\mathrm{Det}_L(\bold{D}_{\mathrm{dR}}(M))\isom \Delta_{L,2}(M)$$ in the same way as $\theta_{\mathrm{dR},L}(M,\zeta)'$ using the constant 
$$\varepsilon_{L}(W(M),\zeta)\cdot \mathrm{det}_L(-\varphi|\bold{D}_{\mathrm{cris}}(M))
=\varepsilon_{L}((\bold{D}_{\mathrm{pst}}(M),\rho),\zeta)\cdot \mathrm{det}_L(-\varphi|\bold{D}_{\mathrm{st}}(M))$$ 
instead of $\varepsilon_{L}(W(V),\zeta)$. Since we have $\theta_L(M)'=\theta_L(M)\cdot 
\mathrm{det}_L(-\varphi^{-1}|\bold{D}_{\mathrm{cris}}(M))$, $\varepsilon_{L,\zeta}^{\mathrm{dR}}(M)$ can be defined using the triple $(\Gamma_L(M), \theta_L(M)', \theta_{\mathrm{dR},L}(M,\zeta)')$ instead of $(\Gamma_L(M), \theta_L(M), \theta_{\mathrm{dR},L}(M,\zeta))$. 

Let $0\rightarrow M_1\rightarrow M_2\rightarrow M_3\rightarrow 0$ be an exact sequence as in (1). 
Since one has $\Gamma(M_2)=\Gamma(M_1)\cdot \Gamma(M_3)$, it suffices to show that both 
$\theta_{L}(-)'$ and $\theta_{\mathrm{dR},L}(-)'$ are compatible with the exact sequence. 

Since we have 
$$\varepsilon_{L}((\bold{D}_{\mathrm{pst}}(M_2),\rho),\zeta)=\varepsilon_{L}((\bold{D}_{\mathrm{pst}}(M_1),\rho),\zeta)\cdot \varepsilon_{L}((\bold{D}_{\mathrm{pst}}(M_3),\rho),\zeta)$$ 
and 
$$ \mathrm{det}_L(-\varphi|\bold{D}_{\mathrm{st}}(M_2))=\mathrm{det}_L(-\varphi|\bold{D}_{\mathrm{st}}(M_1))\cdot \mathrm{det}_L(-\varphi|\bold{D}_{\mathrm{st}}(M_3))$$
(since $\bold{D}_{\mathrm{pst}}(-)$ and $\bold{D}_{\mathrm{st}}(-)$ are exact for de Rham 
$(\varphi,\Gamma)$-modules), the isomorphism $\theta_{\mathrm{dR},L}(-)'$ is compatible with the exact sequence.
  
 We remark that the functor $\bold{D}_{\mathrm{cris}}(-)$ is not exact (in general) for de Rham 
 $(\varphi,\Gamma)$-modules, but we have the following exact sequence
\begin{multline*}
0\rightarrow \bold{D}_{\mathrm{cris}}(M_1)\rightarrow \bold{D}_{\mathrm{cris}}(M_2)\rightarrow \bold{D}_{\mathrm{cris}}(M_3)\\
\xrightarrow{(*)}\bold{D}_{\mathrm{cris}}(M_1^*)^{\vee}\rightarrow \bold{D}_{\mathrm{cris}}(M_2^*)^{\vee}\rightarrow \bold{D}_{\mathrm{cris}}(M_3^*)^{\vee}\rightarrow 0
\end{multline*} such that the boundary map $(*)$ satisfies the following commutative diagram

 \begin{equation*}
 \begin{CD}
 \bold{D}_{\mathrm{cris}}(M_3)@> (*) >> \bold{D}_{\mathrm{cris}}(M_1^*)^{\vee}\\
 @VV \varphi^{-1}V @VV \varphi^{\vee} V\\
  \bold{D}_{\mathrm{cris}}(M_3)@> (*) >> \bold{D}_{\mathrm{cris}}(M_1^*)^{\vee},
  \end{CD}
  \end{equation*}
  from which the compatibility of $\theta_L(-)'$ with the exact sequence follows, which finishes the 
  proof of the lemma.

 \end{proof}

 
 \subsection{Formulation of the local $\varepsilon$-conjecture}
 In this subsection, using the definitions in the previous subsections, we formulate 
 the following conjecture which we call local $\varepsilon$-conjecture. 
 This conjecture is a combination of Kato's original $\varepsilon$-conjecture for $(\Lambda, T)$ with 
 our conjecture for $(A,M)$. To state both situations in a same time, we use notation 
 $(B,N)$ for $(\Lambda,T)$ or $(A,M)$, and $f:B\rightarrow B'$ for 
 $f:\Lambda\rightarrow \Lambda'$ or $f:A\rightarrow A'$.
 
 \begin{conjecture}\label{3.9}
 We can uniquely define a $B$-linear isomorphism 
 $$\varepsilon_{B,\zeta}(N):\bold{1}_B\isom \Delta_{B}(N)$$ for each pair $(B, N)$ as above and for each 
 $\mathbb{Z}_p$-basis $\zeta$ of $\mathbb{Z}_p(1)$ satisfying the following conditions.
\begin{itemize}
\item[(i)]Let $f:B\rightarrow B'$ be a continuous homomorphism. Then, we have 
$$\varepsilon_{B,\zeta}(N)\otimes\mathrm{id}_{B'}=\varepsilon_{B',\zeta}(N\otimes_{B}B')$$
under the canonical isomorphism $\Delta_{B}(N)\otimes_{B}B'\isom \Delta_{B'}
(N\otimes_{B}B')$.
\item[(ii)]Let $0\rightarrow N_1\rightarrow N_2\rightarrow N_3\rightarrow 0$ be an exact sequence. Then, we have 
$$\varepsilon_{B,\zeta}(N_1)\boxtimes
\varepsilon_{B,\zeta}(N_3)=\varepsilon_{B,\zeta}(N_2)$$
under the canonical isomorphism 
$\Delta_{B}(N_1)\boxtimes\Delta_{B}(N_3)\isom\Delta_{B}(N_2).$

\item[(iii)]For any $a\in \mathbb{Z}_p^{\times}$, we have 
$$\varepsilon_{B,\zeta^a}(N)=\delta_{\mathrm{det}_{B}(N)}(a)\cdot\varepsilon_{B,\zeta}(N).$$
\item[(iv)]One has the following commutative diagram of isomorphisms 
\begin{equation*}
 \begin{CD}
 \Delta_B(N)@> \mathrm{can} >> \Delta_B(N^*)^{\vee}\boxtimes (L(r_N), 0)\\
 @A \varepsilon_{B,\zeta^{-1}}(N) AA @VV
 \varepsilon_{B,\zeta}(N^*)^{\vee}\boxtimes[\bold{e}_{r_N}\mapsto 1] V\\
  \bold{1}_B@> \mathrm{can}>> \bold{1}_B\boxtimes\bold{1}_B.
  \end{CD}
  \end{equation*}

\item[(v)]Let $f:\Lambda\rightarrow A$ be a continuous homomorphism, and let 
$M:=\bold{D}_{\mathrm{rig}}(T\otimes_{\Lambda}A)$ be the associated 
$(\varphi,\Gamma)$-module obtained by the base change of  $T$ with respect to $f$. Then, 
we have 
$$\varepsilon_{\Lambda,\zeta}(T)\otimes \mathrm{id}_A=\varepsilon_{A,\zeta}(M)$$ under the canonical isomorphism $\Delta_{\Lambda}(T)\otimes_{\Lambda}A\isom \Delta_{A}(M)$ defined in 
Corollary \ref{3.2}.
\item[(vi)]Let $L=A$ be a finite extension of $\mathbb{Q}_p$, and let $M$ be a de Rham 
$(\varphi,\Gamma)$-module over $\mathcal{R}_L$. Then we have 
$$\varepsilon_{L,\zeta}(M)=\varepsilon^{\mathrm{dR}}_{L,\zeta}(M).$$
\end{itemize}
\end{conjecture}
\begin{rem}\label{3.10}
Kato's original conjecture (\cite{Ka93b}) is just the restriction of the above conjecture 
to the pairs $(\Lambda, T)$. As we explained in Remark \ref{3.7}, we insist that the condition 
(v) should be stated using $\varepsilon_{L,\zeta}^{\mathrm{dR}}(\bold{D}_{\mathrm{rig}}(V))$ (or $\varepsilon_{L,\zeta}^{\mathrm{dR}}(V):=\varepsilon^{\mathrm{dR}}_{L,\zeta}(V)'\cdot 
\mathrm{det}_L(-\varphi|\bold{D}_{\mathrm{st}}(V)/\bold{D}_{\mathrm{cris}}(V))$) instead of $\varepsilon^{\mathrm{dR}}_{L,\zeta}(V)'$.

\end{rem}

\begin{rem}\label{3.11}
In Kato's conjecture, the uniqueness of 
$\varepsilon$-isomorphism was not explicitly predicted. Recently, it is known that the de Rham points 
(even crystalline points) are Zariski dense in ``universal'' families of $p$-adic representations, or 
$(\varphi,\Gamma)$-modules in many cases (\cite{Co08},\cite{Kis10} for two dimensional case, \cite{Ch13},\cite{Na14b} for general case), hence we add the uniqueness assertion 
in our conjecture.
\end{rem}

In \cite{Ka93b}, Kato proved his conjecture for the rank one case (remark that one has $\bold{D}_{\mathrm{st}}(V)=\bold{D}_{\mathrm{cris}}(V)$ for the rank one case, hence one 
also has $\varepsilon_{L,\zeta}^{\mathrm{dR}}(V)'=\varepsilon_{L,\zeta}^{\mathrm{dR}}(V)$). 
As a generalization of his theorem, our main theorem of this article is the following, whose 
proof is given in the next section. 

\begin{thm}\label{3.12}
The conjecture $\ref{3.9}$ is true for  the rank case. More precisely, we can uniquely
 define a $B$-linear isomorphism  $\varepsilon_{B,\zeta}(N):\bold{1}_B\isom\Delta_B(N)$ 
  for each pair $(B, N)$ such that $N$ is of rank one and for each 
 $\mathbb{Z}_p$-basis $\zeta$ of $\mathbb{Z}_p(1)$ satisfying the
 conditions $(\mathrm{i}), (\mathrm{iii}), (\mathrm{iv}), (\mathrm{v}), (\mathrm{vi})$.

\end{thm}

Before passing to the proof of this theorem in the next section, we prove two easy corollaries concerning the trianguline case. We say that a $(\varphi,\Gamma)$-module $M$ over $\mathcal{R}_A$ is trianguline if $M$ has a
 filtration $\mathcal{F}:0:=M_0\subseteq M_1\subseteq \cdots \subseteq M_n:=M$ 
 whose graded quotients $M_i/M_{i-1}$ are rank one $(\varphi,\Gamma)$-modules over $\mathcal{R}_A$ for all $1\leqq i\leqq n$. We call the filtration $\mathcal{F}$ a triangulation of $M$.
 
 \begin{corollary}\label{3.13}
 Let $M$ be a trianguline $(\varphi,\Gamma)$-module over $\mathcal{R}_A$ of rank $n$ with a triangulation $F$ as above. The isomorphism 
 $$\varepsilon_{\mathcal{F}, A, \zeta}(M):\bold{1}_A\xrightarrow{\boxtimes_{i=1}^{n}\varepsilon_{A, \zeta}(M_i/M_{i-1})}\boxtimes_{i=1}^n\Delta_A(M_i/M_{i-1})\isom \Delta_A(M)
 $$ defined 
 as the product of the isomorphisms $\varepsilon_{A, \zeta}(M_i/M_{i-1}):\bold{1}_A\isom
 \Delta_A(M_i/M_{i-1})$ which are defined in Theorem $\ref{3.12}$ satisfies the following properties. 
 \begin{itemize}
 \item[(i)']For any $f:A\rightarrow A'$, we have 
$$\varepsilon_{\mathcal{F}, A,\zeta}(M)\otimes\mathrm{id}_{A'}=\varepsilon_{\mathcal{F}', A',\zeta}(M\otimes_{A}A')$$
where $\mathcal{F}'$ is the base change of the triangulation $\mathcal{F}$ by $f$.
\item[(iii)']For any $a\in \mathbb{Z}_p^{\times}$, we have 
$$\varepsilon_{\mathcal{F}, A,\zeta^a}(M)=\delta_{\mathrm{det}_{A}(M)}(a)\cdot\varepsilon_{\mathcal{F}, A,\zeta}(M).$$
\item[(iv)']One has the following commutative diagram of isomorphisms 
\begin{equation*}
 \begin{CD}
 \Delta_A(M)@> \mathrm{can} >> \Delta_A(M^*)^{\vee}\boxtimes (A(r_M), 0)\\
 @A \varepsilon_{\mathcal{F},A,\zeta}(M) AA @VV
 \varepsilon_{\mathcal{F}^*, A,\zeta}(M^*)^{\vee}\boxtimes[\bold{e}_{r_M}\mapsto (-1)^{r_M}] V\\
  \bold{1}_A@> \mathrm{can}>> \bold{1}_A\boxtimes\bold{1}_A,
  \end{CD}
  \end{equation*}
where $\mathcal{F}^*$ is the Tate dual of the triangulation $\mathcal{F}$.
\item[(vi)']Let $L=A$ be a finite extension of $\mathbb{Q}_p$, and let $M$ be a de Rham and trianguline 
$(\varphi,\Gamma)$-module over $\mathcal{R}_L$. Then, for any triangulation $\mathcal{F}$ of $M$,  we have 
$$\varepsilon_{\mathcal{F}, L,\zeta}(M)=\varepsilon^{\mathrm{dR}}_{L,\zeta}(M).$$
In particular, in this case, $\varepsilon_{\mathcal{F}, L,\zeta}(M)$ does not depend on $\mathcal{F}$.
\end{itemize}
\end{corollary}
\begin{proof}
This corollary immediately follows from Theorem \ref{3.12} since $\varepsilon^{\mathrm{dR}}_{L,\zeta}(M)$ is multiplicative with respect to exact sequences.
\end{proof}

Finally, we compare Corollary \ref{3.13} with the previous known results on Kato's $\varepsilon$-conjecture for twists of crystalline case. Let $F$ be a finite unramified extension of $\mathbb{Q}_p$. Let $V$ be a crystalline $L$-representation of $G_{F}$, and let $T\subseteq V$ be a $G_{F}$-stable $\mathcal{O}_L$-lattice of $V$. In \cite{BB08} and \cite{LVZ14}, they defined $\varepsilon$-isomorphisms for some twists of $T$. Here, for simplicity, we only recall the result of \cite{BB08} under the additional assumption that $F=\mathbb{Q}_p$, since other cases can be proven in the same way. Let $\mathcal{O}_L[[\Gamma]]$ be the Iwasawa algebra with coefficient $\mathcal{O}_L$. We define a $\mathcal{O}_L[[\Gamma]]$-representation $\bold{Dfm}(T):=T\otimes_{\mathcal{O}_L}\mathcal{O}_L[[\Gamma]]$ on which $G_{\mathbb{Q}_p}$ acts by $g(x\otimes y):=g(x)\otimes [\bar{g}]^{-1}y$ for any $g\in G_{\mathbb{Q}_p}, x\in T, y\in \mathcal{O}_L[[\Gamma]]$.
In \cite{BB08}, by studying the associated Wach modules very carefully, they essentially showed that Perrin-Riou's big exponential map induces an $\varepsilon$-isomorphism which we denote by 
$$\varepsilon_{\mathcal{O}_L[[\Gamma]],\zeta}^{\mathrm{BB}}(\bold{Dfm}(T)):\bold{1}_{\mathcal{O}_L[[\Gamma]]}\isom\Delta_{\mathcal{O}_L[[\Gamma]]}(\bold{Dfm}(T))$$ 
satisfying the conditions in Conjecture \ref{3.9}. Let $\bold{D}_{\mathrm{rig}}(V)$ be the 
$(\varphi,\Gamma)$-module over $\mathcal{R}_L$ associated to $V$. Then, applying Example \ref{3.3} to $(\Lambda, T)=(\mathcal{O}_L[[\Gamma]], T)$, we obtain a canonical isomorphism 
\begin{equation}\label{isom1}
\Delta_{\mathcal{O}_L[[\Gamma]]}(\bold{Dfm}(T))\otimes_{\mathcal{O}_L[[\Gamma]]}\mathcal{R}_L^{\infty}(\Gamma)\isom \Delta_{\mathcal{R}^{\infty}(\Gamma)}(\bold{Dfm}(\bold{D}_{\mathrm{rig}}(V)))
\end{equation}
(see the next section for the definitions of $\mathcal{R}_L^{\infty}(\Gamma)$ and $\bold{Dfm}(\bold{D}_{\mathrm{rig}}(V))$). Since $\bold{D}_{\mathrm{rig}}(V)$ is crystalline, after extending scalars, we may assume that it is trianguline with a triangulation $\mathcal{F}$. Then, $\bold{Dfm}(\bold{D}_{\mathrm{rig}}(V))$ is also trianguline with a triangulation $\mathcal{F}':=\bold{Dfm}(\mathcal{F})$. Hence, by Corollary 
\ref{3.13}, we obtain an isomorphism $$\varepsilon_{\mathcal{F}', \mathcal{R}_L^{\infty}(\Gamma), \zeta}(\bold{Dfm}(\bold{D}_{\mathrm{rig}}(V))):\bold{1}_{\mathcal{R}_L^{\infty}(\Gamma)}\isom 
\Delta_{\mathcal{R}^{\infty}(\Gamma)}(\bold{Dfm}(\bold{D}_{\mathrm{rig}}(V)))
.$$
Under this situation, we easily obtain the following corollary.

\begin{corollary}\label{3.14}
Under the isomorphism $(\ref{isom1})$, we have 
$$\varepsilon^{\mathrm{BB}}_{\mathcal{O}_L[[\Gamma]],\zeta}(\bold{Dfm}(T))\otimes \mathrm{id}_{\mathcal{R}_L^{\infty}(\Gamma)}=\varepsilon_{\mathcal{F}', \mathcal{R}_L^{\infty}(\Gamma), \zeta}(\bold{Dfm}(\bold{D}_{\mathrm{rig}}(V))).$$
In particular, the isomorphism $\varepsilon_{\mathcal{F}', \mathcal{R}_L^{\infty}(\Gamma), \zeta}(\bold{Dfm}(\bold{D}_{\mathrm{rig}}(V)))$ does not depend on $\mathcal{F}$.

\end{corollary}
\begin{proof}
By \cite{BB08} and Theorem \ref{3.12}, the base changes of the both sides in Corollary 
\ref{3.14} by the continuous $L$-algebra morphism 
$f_{\delta}:\mathcal{R}_L^{\infty}(\Gamma)\rightarrow L:[\gamma]\rightarrow \delta(\gamma)^{-1}$ are equal to $\varepsilon_{L,\zeta}^{\mathrm{dR}}(\bold{D}_{\mathrm{rig}}(V(\delta)))$ for any potentially crystalline character $\delta:\Gamma\rightarrow L^{\times}$. Since the points corresponding to such characters are Zariski dense in the rigid analytic space associated to $\mathrm{Spf}(\mathcal{O}_L[[\Gamma]])$, we obtain the equality in the corollary.

\end{proof}

\section{Rank one case}

In \cite{Ka93b}, Kato proved his $\varepsilon$-conjecture using the theory 
of  Coleman homomorphism which interpolates the exponential maps and the dual exponential maps of rank one de Rham $p$-adic representations of $G_{\mathbb{Q}_p}$. In particular, 
so called the explicit reciprocity law, which is the explicit formula of it's interpolation property, was very important in his proof. 

In this final section, we first construct the $\varepsilon$-isomorphism
$$\varepsilon_{A,\zeta}(M):\bold{1}_A\isom\Delta_A(M)$$ for any rank one $(\varphi,\Gamma)$-module $M$ by interpreting the theory of Coleman homomorphism in terms of $p$-adic Fourier transform (e.g. Amice transform, Colmez transform), which seems to be standard for the experts of the theory of $(\varphi,\Gamma)$-modules. Then, we prove that this isomorphism satisfies the de Rham condition (vi) by establishing the ``explicit reciprocity law" of our Coleman homomorphism using our theory 
of Bloch-Kato's exponential map developed in \S2.2.

\subsection{Construction of the $\varepsilon$-isomorphism}

We first recall the theory of analytic Iwasawa cohomology of 
$(\varphi,\Gamma)$-modules over the Robba ring after \cite{Po13b} and \cite{KPX14}. 
Let $\Lambda(\Gamma):=\mathbb{Z}_p[[\Gamma]]$ be the Iwasawa algebra of $\Gamma$ with 
coefficients in $\mathbb{Z}_p$,  and let $\mathfrak{m}$ be the Jacobson radical of $\Lambda(\Gamma)$. 
For each $n\geqq 1$, define a $\mathbb{Q}_p$-affinoid algebra $\mathcal{R}^{[1/p^n,\infty]}(\Gamma):=(\Lambda(\Gamma)[\frac{\mathfrak{m}^n}{p}])^{\wedge}[1/p]$ where, for any ring $R$, let denote by $R^{\wedge}$ the $p$-adic completion of $R$. 
Let $X_n:=\mathrm{Max}(\mathcal{R}^{[1/p^n,\infty]}(\Gamma))$ be the associated affinoid. Define $X:=\cup_{n\geqq 1}X_n$ 
, which is a disjoint union of open unit discs. For $n\geqq 1$, consider the rank one $(\varphi,\Gamma)$-module 
$$\bold{Dfm}_n:=\mathcal{R}^{[1/p^n,\infty]}(\Gamma)\hat{\otimes}_{\mathbb{Q}_p}\mathcal{R}\bold{e}=
\mathcal{R}_{\mathcal{R}^{[1/p^n,\infty]}(\Gamma)}\bold{e}$$ with
$$\varphi(1\hat{\otimes}\bold{e})=1\hat{\otimes}\bold{e}\text{ and }
\gamma(1\hat{\otimes}\bold{e})=[\gamma]^{-1}\hat{\otimes}\bold{e}\text{ for }\gamma\in \Gamma.$$ 
Put $\bold{Dfm}:=\varprojlim_n\bold{Dfm}_n$; this is a $(\varphi,\Gamma)$-module over the relative Robba ring over $X$. For $M$ a $(\varphi,\Gamma)$-module over $\mathcal{R}_A$, we define the cyclotomic deformation of $M$ by $$\bold{Dfm}(M):=\varprojlim_n\bold{Dfm}_n(M)$$ with 
$$\bold{Dfm}_n(M):=
M\hat{\otimes}_{\mathcal{R}}\bold{Dfm}_n\isom M\hat{\otimes}_A\mathcal{R}^{[1/p^n,\infty]}_A(\Gamma)\bold{e},$$ which is a $(\varphi,\Gamma)$-module over the relative Robba ring 
over $\mathrm{Max}(A)\times X$. This $(\varphi,\Gamma)$-module is the universal cyclotomic deformation of $M$ in the sense that, for each continuous homomorphism 
$\delta_0:\Gamma\rightarrow A^{\times}$, we have a natural isomorphism 
$$\bold{Dfm}(M)\otimes_{\mathcal{R}_A^{\infty}(\Gamma),f_{\delta_0}}A\isom M(\delta_0)
:(x\hat{\otimes} \eta\bold{e})\otimes a\mapsto f_{\delta_0}(\eta)ax\bold{e}_{\delta_0}$$
for $x\in M$ and $\eta\bold{e}\in \mathcal{R}^{\infty}_A(\Gamma)\bold{e}$ and $a\in A$, 
where $$f_{\delta_0}:\mathcal{R}_A^{\infty}(\Gamma)\rightarrow A$$ is the continuous $A$-algebra 
homomorphism defined by $$f_{\delta_0}([\gamma]):=\delta_0(\gamma)^{-1}$$ for $\gamma\in \Gamma$ (and recall that $M(\delta_0):=M\otimes_A A\bold{e}_{\delta_0}=M\bold{e}_{\delta_0}$ is defined by $\varphi(x\bold{e}_{\delta_0})=\varphi(x)\bold{e}_{\delta_0}$ and $\gamma(x\bold{e}_{\delta_0}):=\delta_0(\gamma)\gamma(x)\bold{e}_{\delta_0}$ for $x\in M$ and $\gamma\in \Gamma$).

By Theorem 4.4.8 of \cite{KPX14}, we have a natural 
quasi-isomorphism of $\mathcal{R}^{\infty}_A(\Gamma)$-modules 
$$g_{\gamma}:C^{\bullet}_{\psi,\gamma}(\bold{Dfm}(M))\isom C^{\bullet}_{\psi}(M):=[M^{\Delta}\xrightarrow{\psi-1}M^{\Delta}]$$
where the latter complex is concentrated in degree $[1,2]$.
This quasi-isomorphism is obtained as a composite of (a system of) quasi-isomorphisms
$$C^{\bullet}_{\psi,\gamma}(\bold{Dfm}_n(M))\isom C^{\bullet}_{\psi}(M)\widehat{\otimes}_{\mathcal{R}^{\infty}_A(\Gamma)}\mathcal{R}^{[1/p^n,\infty]}_A(\Gamma)$$

which are naturally induced by the following diagrams of short exact sequences of $\mathcal{R}_A^{[1/p^n,\infty]}(\Gamma)$-modules for $n\geqq 1$

\begin{equation}\label{20e}
\begin{CD}
0@>>> \bold{Dfm}_n(M)^{\Delta}@>\gamma-1 >>  \bold{Dfm}_n(M)^{\Delta}@> f _{\gamma}>> M\widehat{\otimes}_{\mathcal{R}^{\infty}_A(\Gamma)}\mathcal{R}^{[1/p^n,\infty]}_A(\Gamma)@>>> 0\\
@. @V \psi-1 VV @V \psi-1 VV @V\psi-1VV@. \\
0@>>>   \bold{Dfm}_n(M)^{\Delta}@>\gamma-1 >>  \bold{Dfm}_n(M)^{\Delta}
@> f _{\gamma}>> M\widehat{\otimes}_{\mathcal{R}^{\infty}_A(\Gamma)}\mathcal{R}^{[1/p^n,\infty]}_A(\Gamma) @>>> 0,
\end{CD}
\end{equation}
where 
$$f_{\gamma}(\sum_i x_i\hat{\otimes}\eta_i \bold{e}):=\frac{1}
{|\Gamma_{\mathrm{tor}}|\mathrm{log}_0(\chi(\gamma))}\sum_ix_i\hat{\otimes} \eta_i$$
for $x_i\in M, \eta_i \bold{e}\in \mathcal{R}_A^{[1/p^n,\infty]}(\Gamma)\bold{e}$,  
with the inverse of the natural quasi-isomorphism 
$$C^{\bullet}_{\psi}(M)\isom \varprojlim_nC^{\bullet}_{\psi}(M\otimes_{\mathcal{R}^{\infty}_A(\Gamma)}
\mathcal{R}_A^{[1/p^n,\infty]}(\Gamma))\isom \varprojlim_n C^{\bullet}_{\psi}(M\widehat{\otimes}_{\mathcal{R}^{\infty}_A(\Gamma)}\mathcal{R}^{[1/p^n,\infty]}_A(\Gamma))$$
(see Theorem 4.4.8 of \cite{KPX14} and Theorem 2.8 (3) of \cite{Po13b} for the proof).
This quasi-isomorphism is canonical in the sense that, for another 
$\gamma'\in \Gamma$ whose image in $\Gamma/\Delta$ is a topological generator, we have the following commutative diagram
\begin{equation}\label{21e}
\begin{CD}
C^{\bullet}_{\psi,\gamma}(\bold{Dfm}(M))@> g_{\gamma}>>C^{\bullet}_{\psi}(M) \\
@V \iota_{\gamma,\gamma'} VV @V \mathrm{id} VV \\
C^{\bullet}_{\psi,\gamma'}(\bold{Dfm}(M))@> g_{\gamma'}>>C^{\bullet}_{\psi}(M).
\end{CD}
\end{equation}

For $\delta_0:\Gamma\rightarrow A^{\times}$, using the natural isomorphism 
$\bold{Dfm}(M)\otimes_{\mathcal{R}^{\infty}_A(\Gamma),f_{\delta_0}}A\isom M(\delta_0)$ and 
the quasi-isomorphism $g_{\gamma}$, we obtain the following quasi-isomorphism 
\begin{multline}\label{22e}
g_{\gamma,\delta_0}:C^{\bullet}_{\psi,\gamma}(M(\delta_0))\isom C^{\bullet}_{\psi,\gamma}
(\bold{Dfm}(M)\otimes_{\mathcal{R}^{\infty}_A(\Gamma),f_{\delta_0}}A)\\
\isom
C^{\bullet}_{\psi,\gamma}(\bold{Dfm}(M))\otimes^{\bold{L}}_{\mathcal{R}^{\infty}_A(\Gamma),f_{\delta_0}}A
\isom C^{\bullet}_{\psi}(M)\otimes^{\bold{L}}_{\mathcal{R}_A^{\infty}(\Gamma),f_{\delta_0}}A,
\end{multline}
where the second isomorphism follows from the fact that any $(\varphi,\Gamma)$-module $M_0$ over $\mathcal{R}_{A_0}$ is flat over $A_0$ for any $A_0$ (see Corollary 2.1.7 of \cite{KPX14}). 
This quasi-isomorphism can be written in a more explicit way as follows. 
To recall this, we see $A$ as a $\mathcal{R}_A^{\infty}(\Gamma)$-module by the map $f_{\delta_0}$. 
Then, we can take 
the following projective resolution of $A$,
$$0\rightarrow \mathcal{R}_A^{\infty}(\Gamma)\cdot p_{\delta_0}\xrightarrow{d_{1,\gamma}}\mathcal{R}_A^{\infty}(\Gamma)\cdot p_{\delta_0}\xrightarrow{d_{2,\gamma}} A\rightarrow 0,$$
where $p_{\delta_0}:=\frac{1}{|\Delta|}\sum_{\sigma\in \Delta}\delta_0^{-1}(\sigma)[\sigma]\in 
\mathcal{R}_A^{\infty}(\Gamma)$ (this is an idempotent ) and 
$$d_{1,\gamma}(\eta):=
(\delta_0(\gamma)[\gamma]-1)\eta,\text{ and }d_{2,\gamma}(\eta):=\frac{1}{|\Gamma_{\mathrm{tor}}|\mathrm{log}_0(\chi(\gamma))}f_{\delta_0}(\eta).$$

This resolution induces a canonical isomorphism 
$$
C^{\bullet}_{\psi}(M)\otimes_{\mathcal{R}_A^{\infty}(\Gamma)}
[\mathcal{R}_A^{\infty}(\Gamma)\cdot p_{\delta_0}\xrightarrow{d_{1,\gamma}}\mathcal{R}_A^{\infty}(\Gamma)\cdot p_{\delta_0}]\isom C^{\bullet}_{\psi}(M)\otimes^{\bold{L}}_{\mathcal{R}^{\infty}_A(\Gamma),f_{\delta_0}}A.$$
Moreover, using the isomorphism 
 $$M\otimes_{\mathcal{R}_A^{\infty}(\Gamma)}\mathcal{R}_A^{\infty}(\Gamma)\cdot p_{\delta_0}\isom M(\delta)^{\Delta}:m\otimes \lambda p_{\delta_0}\mapsto \lambda p_{\delta_0}(m\bold{e}_{\delta_0}),$$
 we obtain a natural isomorphism 
 $$C^{\bullet}_{\psi}(M)\otimes_{\mathcal{R}_A^{\infty}(\Gamma)}
[\mathcal{R}_A^{\infty}(\Gamma)\cdot p_{\delta_0}\xrightarrow{d_{1,\gamma}}\mathcal{R}_A^{\infty}(\Gamma)\cdot p_{\delta_0}]\isom C^{\bullet}_{\psi,\gamma}(M(\delta_0)).$$ 
Composing both, we obtain a natural quasi-isomorphism 
$$C^{\bullet}_{\psi}(M)\otimes^{\bold{L}}_{\mathcal{R}_A^{\infty}(\Gamma),f_{\delta_0}}A\isom C^{\bullet}_{\psi,\gamma}(M(\delta_0)),$$
which is easily to seen that this is equal to $g_{\gamma,\delta_0}$.

Using the theory of analytic Iwasawa cohomology recalled as above, we can describe the fundamental line 
$\Delta_{\mathcal{R}_A^{\infty}(\Gamma)}(\bold{Dfm}(M))$ as follows. 
The quasi-isomorphism $g_{\gamma}:C^{\bullet}_{\psi,\gamma}(\bold{Dfm}(M))\isom C^{\bullet}_{\psi}(M)$ and the quasi-isomorphism 
$C^{\bullet}_{\varphi,\gamma}(\bold{Dfm}(M))\isom C^{\bullet}_{\psi,\gamma}(\bold{Dfm}(M))$ 
induce a natural isomorphism in $\mathcal{P}_{\mathcal{R}^{\infty}_A(\Gamma)}$ 
$$\Delta_{\mathcal{R}^{\infty}_A(\Gamma),1}(\bold{Dfm}(M))
\isom \mathrm{Det}_{\mathcal{R}^{\infty}_A(\Gamma)}(C^{\bullet}_{\psi,\gamma}(\bold{Dfm}(M)))
\isom \mathrm{Det}_{\mathcal{R}^{\infty}_A(\Gamma)}(C^{\bullet}_{\psi}(M)).$$
Moreover, since we have 
 \begin{multline*}
\Delta_{\mathcal{R}^{\infty}_A(\Gamma),2}(\bold{Dfm}(M))=
\varprojlim_{n}\Delta_{\mathcal{R}^{[1/p^n,\infty]}_A(\Gamma),2}(\bold{Dfm}_n(M))\\
\isom\varprojlim_n (\Delta_{A,2}(M)\otimes_A\mathcal{R}^{[1/p^n,\infty]}_A(\Gamma)\bold{e}^{\otimes r_M})
=\Delta_{A,2}(M)\bold{e}^{\otimes r_M}\otimes_A\mathcal{R}^{\infty}_A(\Gamma)\\
\isom \Delta_{A,2}(M)\otimes_A\mathcal{R}^{\infty}_A(\Gamma),
\end{multline*}
where the last isomorphism is just the division  by $\bold{e}^{\otimes r_M}$, we obtain a canonical isomorphism 
 
\begin{equation}\label{23.5e}
\Delta_{\mathcal{R}^{\infty}_A(\Gamma)}(\bold{Dfm}(M))\isom \mathrm{Det}_{\mathcal{R}^{\infty}_A(\Gamma)}(C^{\bullet}_{\psi}(M))\boxtimes(\Delta_{A,2}(M)\otimes_A\mathcal{R}^{\infty}_A(\Gamma)).
\end{equation}
Under this canonical isomorphism, we will first define an isomorphism 
$$\theta_{\zeta}(M): \mathrm{Det}_{\mathcal{R}^{\infty}_A(\Gamma)}(C^{\bullet}_{\psi}(M))^{-1}\isom(\Delta_{A,2}(M)\otimes_A\mathcal{R}^{\infty}_A(\Gamma)), $$ 
and then define $\varepsilon_{\mathcal{R}_A^{\infty}(\Gamma),\zeta}(\bold{Dfm}(M))$ as the following 
composite
\begin{multline*}
\varepsilon_{\mathcal{R}_A^{\infty}(\Gamma),\zeta}(\bold{Dfm}(M)):
\bold{1}_{\mathcal{R}_A^{\infty}(\Gamma)}\xrightarrow{\mathrm{can}}\mathrm{Det}_{\mathcal{R}^{\infty}_A(\Gamma)}(C^{\bullet}_{\psi}(M))\boxtimes\mathrm{Det}_{\mathcal{R}^{\infty}_A(\Gamma)}(C^{\bullet}_{\psi}(M))^{-1} \\
\xrightarrow{\mathrm{id}\boxtimes \theta_{\zeta}(M)}\mathrm{Det}_{\mathcal{R}^{\infty}_A(\Gamma)}(C^{\bullet}_{\psi}(M))\boxtimes(\Delta_{A,2}(M)\otimes_A\mathcal{R}^{\infty}_A(\Gamma))
\isom \Delta_{\mathcal{R}^{\infty}_A(\Gamma)}(\bold{Dfm}(M))
\end{multline*}
for the following special rank one $(\varphi,\Gamma)$-modules $M$. 


For $\lambda\in A^{\times}$, define the ``unramified" continuous homomorphism $\delta_{\lambda}:\mathbb{Q}_p^{\times}\rightarrow A^{\times}$ by $\delta_{\lambda}(p):=\lambda$ and $\delta_{\lambda}|_{\mathbb{Z}_p^{\times}}:=1.$ We define an isomorphism 
$\theta_{\zeta}(M)$ for $M=\mathcal{R}_A(\delta_{\lambda})$ by the following steps,
which are based on the re-interpretation of 
  the theory of the Coleman homomorphism in terms of the $p$-adic Fourier transform.

Let $\mathrm{LA}(\mathbb{Z}_p, A)$ be the set of $A$-valued locally analytic functions on $\mathbb{Z}_p$, and define the action of $(\varphi,\psi,\Gamma)$ on it by 
 $$\varphi(f)(y):=f\left(\frac{y}{p}\right) \,\,(\,y\in p\mathbb{Z}_p) \text{ and } \varphi(f)|_{\mathbb{Z}_p^{\times}}:=0,$$

$$\psi(f)(y):=f(py)\text{ and }\,\, \gamma(f)(y):=\frac{1}{\chi(\gamma)}f\left(\frac{y}{\chi(\gamma)}\right) \, (\gamma\in \Gamma ).$$ 
One has a $(\varphi,\psi,\Gamma)$-equivariant $A$-linear surjection which we call the Colmez transform
\begin{equation}\label{Col}
\mathrm{Col}:\mathcal{R}_A\rightarrow \mathrm{LA}(\mathbb{Z}_p,A)
\end{equation} 
defined by 
$$\mathrm{Col}(f(\pi))(y):=\mathrm{Res}_0\left((1+\pi)^yf(\pi)\frac{d\pi}{(1+\pi)}\right),$$
where $\mathrm{Res}_0:\mathcal{R}_A\rightarrow A$ is defined by 
$\mathrm{Res}_0(\sum_{n\in \mathbb{Z}}a_n\pi^n):=a_{-1}$
(remark that $\mathrm{Col}$ depends on the choice of the parameter $\pi$, i.e. the choice of $\zeta$).
By this map,  we obtain the following short exact sequence 
$$0\rightarrow \mathcal{R}^{\infty}_A\rightarrow \mathcal{R}_A\xrightarrow{\mathrm{Col}} \mathrm{LA}(\mathbb{Z}_p, A)\rightarrow 0.$$
Twisting the action of $(\varphi,\psi,\Gamma)$ by $\delta_{\lambda}$, we obtain the following 
$(\varphi,\psi,\Gamma)$-equivariant exact sequence
$$0\rightarrow \mathcal{R}^{\infty}_A(\delta_{\lambda})\rightarrow \mathcal{R}_A(\delta_{\lambda})
\xrightarrow{\mathrm{Col}\otimes\bold{e}_{\delta_{\lambda}}} \mathrm{LA}(\mathbb{Z}_p, A)(\delta_{\lambda})\rightarrow 0,$$
from which we obtain the following exact sequences of complexes of $\mathcal{R}_A^{\infty}(\Gamma)$-modules
\begin{equation}\label{24e}
0\rightarrow C^{\bullet}_{\psi}(\mathcal{R}^{\infty}_A(\delta_{\lambda}))\rightarrow C^{\bullet}_{\psi}(\mathcal{R}_A(\delta_{\lambda}))\rightarrow C^{\bullet}_{\psi}(\mathrm{LA}(\mathbb{Z}_p, A)(\delta_{\lambda}))
\rightarrow 0.
\end{equation}

For each $k\geqq 0$, we define the algebraic function 
$$y^k:\mathbb{Z}_p\rightarrow A: a\mapsto a^k,$$ then 
$Ay^{k}\bold{e}_{\delta_{\lambda}}\subseteq \mathrm{LA}(\mathbb{Z}_p,A)(\delta_{\lambda})$ is 
a $\psi$-stable sub $\mathcal{R}^{\infty}_A(\Gamma)$-module. By Lemme 2.9 of \cite{Ch13}, the natural inclusion 
\begin{equation}\label{25e}
C^{\bullet}_{\psi}(\oplus_{0=k}^NAy^k\bold{e}_{\delta_{\lambda}})\hookrightarrow C^{\bullet}_{\psi}(\mathrm{LA}(\mathbb{Z}_p, A)(\delta_{\lambda}))
\end{equation}
is quasi-isomorphism for sufficiently large $N$. 

Set $P_i^k:=Ay^k\bold{e}_{\delta_{\lambda}}$ for $i=1,2$. Since we have $Ay^k\bold{e}_{\delta_{\lambda}}[0]\in \bold{D}^{[-1,0]}_{\mathrm{perf}}(\mathcal{R}_A^{\infty}(\Gamma))$ for any
$k\geqq 0$, 
the natural exact sequence 
$$0\rightarrow P^{k}_1[-1]\rightarrow C^{\bullet}_{\psi}(Ay^k\bold{e}_{\delta_{\lambda}})\rightarrow P^{k}_2[-2]\rightarrow 0$$ 
induces a canonical isomorphism 
$$
g_k:\mathrm{Det}_{\mathcal{R}_A^{\infty}(\Gamma)}(C^{\bullet}_{\psi}(Ay^k\bold{e}_{\delta_{\lambda}}))
\isom \mathrm{Det}_{\mathcal{R}_A^{\infty}(\Gamma)}(P^{k}_2)\boxtimes 
\mathrm{Det}_{\mathcal{R}_A^{\infty}(\Gamma)}(P^{k}_1)^{-1}\\
\xrightarrow{i_{\mathrm{Det}_{\mathcal{R}_A^{\infty}(\Gamma)}(P^{k}_1)}}\bold{1}_{\mathcal{R}_A^{\infty}(\Gamma)}.$$ We remark that, if the complex $C^{\bullet}_{\psi}(Ay^k\bold{e}_{\delta_{\lambda}})$ is acyclic, then the composite of this isomorphism with the inverse of the 
canonical trivialization isomorphism $h_{\mathrm{Det}_{\mathcal{R}_A^{\infty}(\Gamma)}(C^{\bullet}_{\psi}(Ay^k\bold{e}_{\delta_{\lambda}}))}:\mathrm{Det}_{\mathcal{R}_A^{\infty}(\Gamma)}(C^{\bullet}_{\psi}(Ay^k\bold{e}_{\delta_{\lambda}}))\isom \bold{1}_{\mathcal{R}_A^{\infty}(\Gamma)}$ is the identity map. Hence, if we define the isomorphism 
\begin{equation}\label{27e}
g^N:=\boxtimes_{0=k}^Ng_k:\mathrm{Det}_{\mathcal{R}_A^{\infty}(\Gamma)}(C^{\bullet}_{\psi}(\oplus_{k=0}^NAy^k\bold{e}_{\delta_{\lambda}}))\isom \bold{1}_{\mathcal{R}_A^{\infty}(\Gamma)}, 
\end{equation}
then, by (\ref{25e}) and (\ref{27e}) (for sufficiently large $N$), 
we obtain an isomorphism 
 \begin{equation}\label{29e}
\iota_0:\mathrm{Det}_{\mathcal{R}^{\infty}_A(\Gamma)}(C^{\bullet}_{\psi}(\mathrm{LA}(\mathbb{Z}_p, A)(\delta_{\lambda}))\isom \bold{1}_{\mathcal{R}^{\infty}_A(\Gamma)}, 
\end{equation}
which is independent of the choice of (sufficiently large) $N$.


Since $C^{\bullet}_{\psi}(\mathrm{LA}(\mathbb{Z}_p, A)(\delta_{\lambda}))$  and $C^{\bullet}_{\psi}(\mathcal{R}_A(\delta_{\lambda}))$ are both perfect complexes, 
we also have $$C^{\bullet}_{\psi}(\mathcal{R}^{\infty}_A(\Gamma))\in \bold{D}^b_{\mathrm{perf}}(\mathcal{R}^{\infty}_{A}(\Gamma))$$ by the exact sequence (\ref{24e}), and then we obtain an isomorphism 
 \begin{multline}\label{35f}
 \iota_1:\mathrm{Det}_{\mathcal{R}_A^{\infty}(\Gamma)}(C^{\bullet}_{\psi}(\mathcal{R}_A(\delta_{\lambda})))\isom \mathrm{Det}_{\mathcal{R}^{\infty}_{A}(\Gamma)}(C^{\bullet}_{\psi}(\mathcal{R}^{\infty}_A(\delta_{\lambda})))\boxtimes\mathrm{Det}_{\mathcal{R}^{\infty}_A(\Gamma)}(C^{\bullet}_{\psi}(\mathrm{LA}(\mathbb{Z}_p, A)(\delta_{\lambda})))\\
\xrightarrow{\mathrm{id}\boxtimes \iota_0} \mathrm{Det}_{\mathcal{R}^{\infty}_A(\Gamma)}(C^{\bullet}_{\psi}(\mathcal{R}^{\infty}_A(\delta_{\lambda})))\isom\mathrm{Det}_{\mathcal{R}_A^{\infty}(\Gamma)}(\mathcal{R}_A^{\infty}(\delta_{\lambda})^{\psi=1}[0])^{-1}, 
\end{multline}
where the last isomorphism is the isomorphism which is naturally induced by the exact sequence 
$0\rightarrow \mathcal{R}^{\infty}_A(\delta_{\lambda})^{\psi=1}\rightarrow\mathcal{R}_A^{\infty}(\delta_{\lambda})\xrightarrow{\psi-1} \mathcal{R}_A^{\infty}(\delta_{\lambda})\rightarrow 0$ (where the subjectivity is proved in Lemma 
2.9 (v) of \cite{Ch13}).

We next consider the complex $C^{\bullet}_{\psi}(\mathcal{R}^{\infty}_A(\delta_{\lambda}))$. 
For a $\mathcal{R}^{\infty}_A(\Gamma)$-module $M$ with linear
 actions of $\varphi$ and $\psi$, define a complex 
$$C^{\bullet}_{\widetilde{\psi}}(M):=[M\xrightarrow{\psi} M]\in \bold{D}^{[1,2]}(\mathcal{R}^{\infty}_A(\Gamma)),$$ 
and define a map of complexes $\alpha_M:C^{\bullet}_{\psi}(M)\rightarrow C^{\bullet}_{\tilde{\psi}}(M)$ 
by 
\begin{equation}
\begin{CD}
C^{\bullet}_{\psi}(M): [ @.M @>\psi-1>>M ] \\
 @VV\alpha_M V@ VV 1-\varphi V @VV  \mathrm{id}_M V \\
 C^{\bullet}_{\widetilde{\psi}}(M):[ @.M @> \psi >>M] .
  \end{CD}
  \end{equation}
  
  For $N\geqq 0$, set $D_N:=\oplus_{0\leqq k\leqq N} At^k\bold{e}_{\delta_{\lambda}}$. 
  Since we have $At^k\bold{e}_{\delta_{\lambda}}[0]\in \bold{D}_{\mathrm{perf}}^{[-1,0]}(\mathcal{R}_A^{\infty}(\Gamma))$, we can define a canonical isomorphism 
  \begin{equation}\label{37a}
  \mathrm{Det}_{\mathcal{R}_A^{\infty}(\Gamma)}(C^{\bullet}_{\psi}(D_N))\isom \bold{1}_{\mathcal{R}_A^{\infty}(\Gamma)}
  \end{equation}
  in the same way as the isomorphism (\ref{27e}).
  Then, the natural exact sequence $0\rightarrow C^{\bullet}_{\psi}(D_N) \rightarrow 
  C^{\bullet}_{\psi}(\mathcal{R}^{\infty}_A(\delta_{\lambda}))\rightarrow 
  C^{\bullet}_{\psi}(\mathcal{R}^{\infty}_A(\delta_{\lambda})/D_N)\rightarrow 0$ induces a canonical isomorphism 
  \begin{multline}\label{39a}
   \mathrm{Det}_{\mathcal{R}_A^{\infty}(\Gamma)}(C^{\bullet}_{\psi}(\mathcal{R}_A^{\infty}(\delta_{\lambda})))\isom  \mathrm{Det}_{\mathcal{R}_A^{\infty}(\Gamma)}(C^{\bullet}_{\psi}(D_N))\boxtimes \mathrm{Det}_{\mathcal{R}_A^{\infty}(\Gamma)}(C^{\bullet}_{\psi}(\mathcal{R}_A^{\infty}(\delta_{\lambda})/D_N))\\
   \isom \mathrm{Det}_{\mathcal{R}_A^{\infty}(\Gamma)}(C^{\bullet}_{\psi}(\mathcal{R}_A^{\infty}(\delta_{\lambda})/D_N))
   \end{multline}
   where the last isomorphism is induced by the isomorphism (\ref{37a}).
 
 Since the map $1-\varphi:\mathcal{R}_A^{\infty}(\delta_{\lambda})/D_N\rightarrow \mathcal{R}_A^{\infty}(\delta_{\lambda})/D_N$ is isomorphism for sufficiently large $N$ by Lemme 2.9 (ii) of \cite{Ch13}, the map 
 $\alpha_{(\mathcal{R}^{\infty}_A(\delta_{\lambda})/D_N)}$ 
 is also isomorphism for sufficiently large $N$. Hence, for sufficiently large $N$, we obtain a canonical isomorphism 
 \begin{equation}\label{40a}
 \mathrm{Det}_{\mathcal{R}_A^{\infty}(\Gamma)}(C^{\bullet}_{\psi}(\mathcal{R}_A^{\infty}(\delta_{\lambda})/D_N))\isom \mathrm{Det}_{\mathcal{R}_A^{\infty}(\Gamma)}(C^{\bullet}_{\widetilde{\psi}}(\mathcal{R}_A^{\infty}(\delta_{\lambda})/D_N)).
  \end{equation}
Since the complex $C^{\bullet}_{\widetilde{\psi}}(D_N)$ 
 is acyclic (since $\psi:At^k\bold{e}_{\delta_{\lambda}}\rightarrow At^k\bold{e}_{\delta_{\lambda}}$ is isomorphism for any $k\geqq 0$), the natural exact sequence $0\rightarrow C^{\bullet}_{\widetilde{\psi}}(D_N) \rightarrow 
  C^{\bullet}_{\widetilde{\psi}}(\mathcal{R}^{\infty}_A(\delta_{\lambda}))\rightarrow 
  C^{\bullet}_{\widetilde{\psi}}(\mathcal{R}^{\infty}_A(\delta_{\lambda})/D_N)\rightarrow 0$ induces a canonical isomorphism 
  \begin{multline}\label{40f}
  \mathrm{Det}_{\mathcal{R}_A^{\infty}(\Gamma)}(C^{\bullet}_{\widetilde{\psi}}(\mathcal{R}_A^{\infty}(\delta_{\lambda})/D_N))\isom  \mathrm{Det}_{\mathcal{R}_A^{\infty}(\Gamma)}(C^{\bullet}_{\widetilde{\psi}}(D_N))\boxtimes \mathrm{Det}_{\mathcal{R}_A^{\infty}(\Gamma)}(C^{\bullet}_{\widetilde{\psi}}(\mathcal{R}_A^{\infty}(\delta_{\lambda})/D_N))\\
 \isom \mathrm{Det}_{\mathcal{R}_A^{\infty}(\Gamma)}(C^{\bullet}_{\widetilde{\psi}}(\mathcal{R}_A^{\infty}(\delta_{\lambda}))),
  \end{multline} 
  where the first isomorphism is induced by the inverse of the 
  isomorphism $h_{C^{\bullet}_{\widetilde{\psi}}(D_N)}:\mathrm{Det}_{\mathcal{R}_A^{\infty}(\Gamma)}(C^{\bullet}_{\widetilde{\psi}}(D_N))\isom \bold{1}_{\mathcal{R}_A^{\infty}(\Gamma)}.$

  Moreover, the exact sequence $0\rightarrow \mathcal{R}_A^{\infty}(\delta_{\lambda})^{\psi=0}
  \rightarrow \mathcal{R}^{\infty}_A(\delta_{\lambda})\xrightarrow{\psi}\mathcal{R}^{\infty}_A(\delta_{\lambda})\rightarrow 0$ and the isomorphism 
  \begin{equation}\label{zeta}
  \mathcal{R}^{\infty}_A(\Gamma)\bold{e}_{\delta_{\lambda}}\isom \mathcal{R}^{\infty}_A(\delta_{\lambda})^{\psi=0}:\lambda\bold{e}_{\delta_{\lambda}}\mapsto (\lambda\cdot(1+\pi)^{-1})\bold{e}_{\delta_{\lambda}}
  \end{equation}
  (remark that this isomorphism depends on the choice of $\zeta$) naturally induces the following isomorphism
  
    \begin{equation}\label{41f}
  \mathrm{Det}_{\mathcal{R}_A^{\infty}(\Gamma)}(C^{\bullet}_{\widetilde{\psi}}(\mathcal{R}_A^{\infty}(\delta_{\lambda})))^{-1}\isom  \mathrm{Det}_{\mathcal{R}_A^{\infty}(\Gamma)}(\mathcal{R}_A^{\infty}(\delta_{\lambda})^{\psi=0})\isom (\mathcal{R}^{\infty}_A(\Gamma)\bold{e}_{\delta_{\lambda}},1).
  \end{equation}
  
  Finally, as the composites of the inverses of the isomorphisms (\ref{35f}), (\ref{39a}), (\ref{40a}), (\ref{40f}), and the isomorphism (\ref{41f}), we define the desired isomorphism 
  $$\theta_{\zeta}(\mathcal{R}_A(\delta_{\lambda})):\mathrm{Det}_{\mathcal{R}_A^{\infty}(\Gamma)}(C^{\bullet}_{\psi}
  (\mathcal{R}_A(\delta_{\lambda})))^{-1}\isom (\mathcal{R}^{\infty}_A(\Gamma)\bold{e}_{\delta_{\lambda}},1)=\Delta_{A,2}(\mathcal{R}_A(\delta_{\lambda}))\otimes_A\mathcal{R}_A^{\infty}(\Gamma).$$

\begin{defn}
Using the isomorphism (\ref{23.5e}), for $M=\mathcal{R}_A(\delta_{\lambda})$, we define the $\varepsilon$-isomorphism by 

\begin{multline*}
\varepsilon_{\mathcal{R}_A^{\infty}(\Gamma),\zeta}(\bold{Dfm}(M)):
\bold{1}_{\mathcal{R}_A^{\infty}(\Gamma)}\xrightarrow{\mathrm{can}}\mathrm{Det}_{\mathcal{R}^{\infty}_A(\Gamma)}(C^{\bullet}_{\psi}(M))\boxtimes\mathrm{Det}_{\mathcal{R}^{\infty}_A(\Gamma)}(C^{\bullet}_{\psi}(M))^{-1} \\
\xrightarrow{\mathrm{id}\boxtimes \theta_{\zeta}(M)}\mathrm{Det}_{\mathcal{R}^{\infty}_A(\Gamma)}(C^{\bullet}_{\psi}(M))\boxtimes(\Delta_{A,2}(M)\otimes_A\mathcal{R}^{\infty}_A(\Gamma))
\isom \Delta_{\mathcal{R}^{\infty}_A(\Gamma)}(\bold{Dfm}(M)).
\end{multline*}

 \end{defn}

Before defining the $\varepsilon$-isomorphism for general rank one case, we check that 
the isomorphism $\varepsilon_{\mathcal{R}_A^{\infty}(\Gamma),\zeta}(\bold{Dfm}(\mathcal{R}_A(\delta_{\lambda})))$ defined above satisfies the properties (i) and (iii) in Conjecture \ref{3.9}

For the property (i),  it is clear that, for each continuous homomorphism $f:A\rightarrow A'$ (and set $\lambda'=f(\lambda)$), we have 
$$\varepsilon_{\mathcal{R}^{\infty}_A(\Gamma),\zeta}(\bold{Dfm}(\mathcal{R}_A(\delta_{\lambda})))\otimes \mathrm{id}_{A'}=\varepsilon_{\mathcal{R}^{\infty}_{A'}(\Gamma),\zeta}(\bold{Dfm}(\mathcal{R}_{A'}(\delta_{\lambda'})))$$
 under the canonical isomorphism 
 \[
\begin{array}{ll}
\Delta_{\mathcal{R}^{\infty}_A(\Gamma)}(\bold{Dfm}(\mathcal{R}_A(\delta_{\lambda})))\otimes_A
A'&\isom \Delta_{\mathcal{R}^{\infty}_{A'}(\Gamma)}(\bold{Dfm}(\mathcal{R}_{A}(\delta_{\lambda})
\widehat{\otimes}_{A}A'))\\
&\isom  \Delta_{\mathcal{R}^{\infty}_{A'}(\Gamma)}(\bold{Dfm}(\mathcal{R}_{A'}(\delta_{\lambda'}
))),
\end{array}
\]
where the last isomorphism is induced by the isomorphism 
$$\mathcal{R}_A(\delta_{\lambda})\widehat{\otimes}_{A}A'\isom \mathcal{R}_{A'}(\delta_{\lambda'}): g(\pi)\bold{e}_{\delta_{\lambda}}\widehat{\otimes}a\mapsto 
ag^{f}(\pi)\bold{e}_{\delta_{\lambda'}}$$
(here we define $g^f(\pi):=\sum_{n\in \mathbb{Z}}f(a_n)\pi^n\in \mathcal{R}_{A'}$ 
 for $g(\pi)=\sum_{n\in \mathbb{Z}}a_n\pi^n\in \mathcal{R}_A$). The property (iii) easily follows from (\ref{zeta}) since one has $(1+\pi_{\zeta^a})=(1+\pi_{\zeta})^a=[\sigma_a]\cdot (1+\pi_{\zeta})$ for $a\in \mathbb{Z}_p^{\times}$.

Next, we consider a rank one $(\varphi,\Gamma)$-module over $\mathcal{R}_A$ of the 
form $\mathcal{R}_A(\delta)$ for a general continuous homomorphism $\delta:\mathbb{Q}_p^{\times}
\rightarrow A^{\times}$. Set 
$$\lambda:=\delta(p)  \text{ and } \delta_0:=\delta|_{\mathbb{Z}_p^{\times}}$$ which we freely see as a homomorphism $\delta_0:\Gamma\rightarrow A^{\times}$ by identifying 
$\chi:\Gamma\isom \mathbb{Z}_p^{\times}$.
We define the continuous $A$-algebra homomorphism 
$$f_{\delta_0}:\mathcal{R}^{\infty}_A(\Gamma)\rightarrow A$$ 
which is uniquely characterized by $f_{\delta_0}([\gamma])=\delta_{0}(\gamma)^{-1}$ for any $\gamma\in \Gamma$.
Then, we have a canonical 
isomorphism
$$\bold{Dfm}(\mathcal{R}_A(\delta_{\lambda}))\otimes_{\mathcal{R}^{\infty}_A(\Gamma),f_{\delta_0}}A\isom \mathcal{R}_A(\delta)$$
defined by 
$$(f(\pi)\bold{e}_{\delta_{\lambda}}\widehat{\otimes}\eta\bold{e})\otimes a:=af_{\delta_0}(\eta)f(\pi)\bold{e}_{\delta}$$ 
for $f(\pi)\in \mathcal{R}_A$, $\eta\in \mathcal{R}^{\infty}_A(\Gamma)$ , $a\in A$, 
which also induces a canonical  isomorphism 
$$\Delta_{\mathcal{R}^{\infty}_A(\Gamma)}(\bold{Dfm}(\mathcal{R}_A(\delta_{\lambda})))\otimes_{\mathcal{R}^{\infty}_A(\Gamma),f_{\delta_0}}A\isom \Delta_{A}(\mathcal{R}_A(\delta)).$$

\begin{defn}We define the isomorphism
$$\varepsilon_{A,\zeta}(\mathcal{R}_A(\delta)):\bold{1}_A\isom 
\Delta_A(\mathcal{R}_A(\delta))$$ 
by
$$\varepsilon_{A,\zeta}(\mathcal{R}_A(\delta)):=\varepsilon_{\mathcal{R}^{\infty}_A(\Gamma),\zeta}
(\bold{Dfm}(\mathcal{R}_A(\delta_{\lambda})))\otimes \mathrm{id}_A$$
under the above isomorphism.
\end{defn}

Next, we consider a rank one $(\varphi,\Gamma)$-module of the form 
$\mathcal{R}_A(\delta)\otimes_A\mathcal{L}$ for an invertible $A$-module $\mathcal{L}$.
\begin{lemma}\label{4.3}
Let $M$ be a $(\varphi,\Gamma)$-module 
over $\mathcal{R}_A$ (of any rank),  and let $\mathcal{L}$ be an invertible $A$-module.
Then, there exist a canonical $A$-linear isomorphism 
$$\Delta_A(M\otimes_A\mathcal{L})\isom\Delta_A(M).$$
\end{lemma}
\begin{proof}
The natural isomorphism 
$C^{\bullet}_{\varphi,\gamma}(M\otimes_A\mathcal{L})
\isom C^{\bullet}_{\varphi,\gamma}(M)\otimes_A\mathcal{L}$ induces 
an isomorphism
 
$$\Delta_{A,1}(M\otimes_A\mathcal{L})
\isom\Delta_{A,1}(M)\boxtimes(\mathcal{L}^{\otimes- r_M},0).$$
Since we also have a natural isomorphism 
$\mathcal{L}_{A}(M\otimes_A\mathcal{L})\isom \mathcal{L}_{A}(M)\otimes_A\mathcal{L}^{\otimes r_M}$, we obtain a natural  isomorphism 
$$\Delta_{A,2}(M\otimes_A\mathcal{L})\isom \Delta_{A,2}(M)\boxtimes(\mathcal{L}^{\otimes r_M},0).$$
Then, the isomorphism in the lemma is obtained by taking the products of these isomorphisms with the canonical isomorphism $i_{(\mathcal{L}^{\otimes r_M},0)}:(\mathcal{L}^{\otimes r_M},0)
\boxtimes (\mathcal{L}^{\otimes -r_M},0)\isom \bold{1}_A$.

\end{proof}
\begin{defn}
We define the isomorphism
$$\varepsilon_{A,\zeta}(\mathcal{R}_A(\delta)\otimes_A\mathcal{L}):\bold{1}_A\isom
\Delta_A(\mathcal{R}_A(\delta)\otimes_A\mathcal{L})$$ 
by $$\varepsilon_{A,\zeta}(\mathcal{R}_A(\delta)\otimes_A\mathcal{L}):=\varepsilon_{A,\zeta}(\mathcal{R}_A(\delta))
$$ under the above isomorphism $\Delta_A(M\otimes_A\mathcal{L})\isom \Delta_A(M).$

\end{defn}

Finally, let $M$ be a general rank one $(\varphi,\Gamma)$-module over $\mathcal{R}_A$. 
By Theorem \ref{2.9}, there exists unique pair $(\delta,\mathcal{L})$  such that $g:M\isom \mathcal{R}(\delta)\otimes_A\mathcal{L}$. This isomorphism induces an isomorphism 
$g_{*}:\Delta_A(M)\isom \Delta_A(\mathcal{R}_A(\delta)\otimes_A\mathcal{L})$. 
\begin{defn}
Under the above situation, we define 
$$\varepsilon_{A,\zeta}(M):=
\varepsilon_{A,\zeta}(\mathcal{R}_A(\delta)\otimes_A\mathcal{L}))\circ g_{*}:\bold{1}_A\isom 
\Delta_A(M).$$
\end{defn}
\begin{lemma}\label{4.6}
The isomorphism $\varepsilon_{A,\zeta}(M)$ is well-defined, i.e. does not depend on $g$.

\end{lemma}
\begin{proof}
Since we have $\mathrm{Aut}(\mathcal{R}_A(\delta)\otimes_A\mathcal{L})=A^{\times}$ 
(where $\mathrm{Aut}(M)$ is the group of automorphisms of $M$ as $(\varphi,\Gamma)$-module over 
$\mathcal{R}_A$), it suffices to show the following lemma.
\end{proof}
\begin{lemma}\label{4.7}
Let $M$ be a $(\varphi,\Gamma)$-module over $\mathcal{R}_A$. 
For $a\in A^{\times}$, let denote by $g_a:M\isom M:x\mapsto ax$. Then, we have 
$$(g_a)_{*}=\mathrm{id}_{\Delta_A(M)}$$.
\end{lemma}
\begin{proof}
This lemma immediately follows from the fact that 
$g_a$ induces $\Delta_{1,A}(M)\isom \Delta_{A,1}(M):x\mapsto a^{-r_M}x$ (by Euler-Poincar\'e formula) and 
$\Delta_{A,2}(M)\isom \Delta_{A,2}(M):x\mapsto a^{r_M}x$ by definition. 
\end{proof}

\begin{rem}
By definition, it is clear that $\varepsilon_{A,\zeta}(M)$ constructed above satisfies the condition (i) and (iii) in Conjecture \ref{3.9}. It also seems to be easy 
to directly prove the condition (iv), (v) of Conjecture 
\ref{3.9}.  However, in the next subsection, we prove the conditions (iv) and (v) using density arguments in the process of 
verifying the condition (vi).
\end{rem}
\begin{rem}
Define $\mathcal{O}_{\mathcal{E}}:=\{\sum_{n\in \mathbb{Z}}a_n\pi^n|a_n\in \mathbb{Z}_p, a_{-n}\rightarrow 0 \,(n\rightarrow +\infty) \}$, 
$\mathcal{O}_{\mathcal{E}^{+}}:=\mathbb{Z}_p[[\pi]]$,
and $\mathcal{O}_{\mathcal{E}^{+},\Lambda}:=\mathcal{O}_{\mathcal{E}^{+}}\widehat{\otimes}_{\mathbb{Z}_p}\Lambda$. Define $\mathcal{C}^0(\mathbb{Z}_p, \Lambda)$ 
to be the $\Lambda$-modules of $\Lambda$-valued continuous functions on $\mathbb{Z}_p$. Using the similar exact sequence 
$$0\rightarrow \mathcal{O}_{\mathcal{E}^+,\Lambda}\rightarrow \mathcal{O}_{\mathcal{E},\Lambda}\xrightarrow{\mathrm{Col}}\mathcal{C}^0(\mathbb{Z}_p,\Lambda)\rightarrow 0,$$
and using the equivalence between the category of $\Lambda$-representations of $G_{\mathbb{Q}_p}$ with that of \'etale 
$(\varphi,\Gamma)$-modules over $\mathcal{O}_{\mathcal{E},\Lambda}$ \cite{Dee01}, it seems to be possible to define an epsilon isomorphism $\varepsilon_{\Lambda,\zeta}(\Lambda(\tilde{\delta}))$ for any $\tilde{\delta}:G_{\mathbb{Q}_p}^{\mathrm{ab}}\rightarrow \Lambda^{\times}$ in the same way as the definition of $\varepsilon_{A,\zeta}(\mathcal{R}_A(\delta))$. Using this $\varepsilon$-isomorphism, it is clear that 
our epsilon isomorphism $\varepsilon_{A,\zeta}(\mathcal{R}_A(\delta))$ satisfies the condition (v) in Conjecture \ref{3.9}. Moreover, if one know the Kato's construction of 
his epsilon isomorphism, one can easily compare the isomorphism $\varepsilon_{\Lambda,\zeta}(\Lambda(\tilde{\delta}))$ with the Kato's one.

\end{rem}



\subsection{Verification of the conditions (iv), (v), (vi)}
In this final subsection, we prove that our $\varepsilon$-isomorphism $\varepsilon_{A,\zeta}(M)$ constructed in the 
previous subsection satisfies the conditions (iv), (v), (vi) of  Conjecture \ref{3.9}. 
Of course, the essential part is to prove the condition (vi); the other conditions follow from 
it using density arguments. 

Therefore, in this subsection, we mainly concentrate on the case where 
$A=L$ is a finite extension of $\mathbb{Q}_p$. 
Before verifying the condition (vi), we describe the isomorphism $\varepsilon_{L,\zeta}
(\mathcal{R}_L(\delta)):\bold{1}_L\isom \Delta_L(\mathcal{R}_L(\delta))$ for any continuous homomorphism $\delta=\delta_{\lambda}\delta_0:\mathbb{Q}_p^{\times}\rightarrow L^{\times}$ in a more explicit way. 

For a $\mathcal{R}^{\infty}_L(\Gamma)$-module $N$, define a $\Gamma$-module $N(\delta_0):=N\bold{e}_{\delta_0}$ by $\gamma(x\bold{e}_{\delta_0})=\delta_0(\gamma)([\gamma]\cdot x)\bold{e}_{\delta_0}$ for any $\gamma\in \Gamma$.
Then, we have a natural quasi-isomorphism
$$N[-1]\otimes^{\bold{L}}_{\mathcal{R}^{\infty}_L(\Gamma),f_{\delta_0}}L\isom
N\otimes_{\mathcal{R}^{\infty}_L(\Gamma)}
[\mathcal{R}^{\infty}_L(\Gamma)p_{\delta_0}\xrightarrow{d_{1,\gamma}}\mathcal{R}^{\infty}_L(\Gamma)p_{\delta_0}]
\isom C^{\bullet}_{\gamma}(N(\delta_0)).$$
Hence, if $N[0]\in \bold{D}^b_{\mathrm{perf}}(\mathcal{R}_L^{\infty}(\Gamma))$, then we obtain a natural isomorphism
$$\mathrm{Det}_L(N[-1])\otimes_{\mathcal{R}^{\infty}_L(\Gamma),f_{\delta_0}}L\isom 
\mathrm{Det}_L(C^{\bullet}_{\gamma}(N(\delta_0)))\isom \boxtimes_{i=0,1}\mathrm{Det}_L(\mathrm{H}^i_{\gamma}(N(\delta_0)))^{(-1)^i}.
$$ 
Moreover, if $N$ is also equipped with a commuting linear action of $\psi$ such that $C^{\bullet}_{\psi}(M)\in 
\bold{D}^b_{\mathrm{perf}}(\mathcal{R}^{\infty}_L(\Gamma))$, then we obtain a natural isomorphism 
$$\mathrm{Det}_L(C^{\bullet}_{\psi}(N))\otimes_{\mathcal{R}^{\infty}_L(\Gamma),f_{\delta_0}}L\isom 
\mathrm{Det}_L(C^{\bullet}_{\psi,\gamma}(N(\delta_0)))\isom \boxtimes_{i=0}^2\mathrm{Det}_L(\mathrm{H}^i_{\psi,\gamma}(N(\delta_0)))^{(-1)^i}.
$$ 
In particular, the isomorphism $\overline{\theta}_{\zeta}(\mathcal{R}_L(\delta)):=\theta_{\zeta}(\mathcal{R}_L(\delta_{\lambda}))\otimes_{\mathcal{R}^{\infty}_L(\Gamma),f_{\delta_0}}\mathrm{id}_L$ can be seen as the following isomorphism 
\begin{equation}\label{isom}
\overline{\theta}_{\zeta}(\mathcal{R}_L(\delta)):\boxtimes_{i=0}^2\mathrm{Det}_L(\mathrm{H}^i_{\psi,\gamma}(\mathcal{R}_L(\delta)))^{ (-1)^{i+1}}
\isom (\mathcal{R}_L^{\infty}(\Gamma)\bold{e}_{\delta_{\lambda}}, 1)\otimes_{\mathcal{R}^{\infty}_L(\Gamma),f_{\delta_0}}L
\isom (L\bold{e}_{\delta}, 1),
\end{equation}
where the last isomorphism is induced by the isomorphism 
$$\mathcal{R}^{\infty}_L(\Gamma)\bold{e}_{\delta_{\lambda}}\otimes_{\mathcal{R}_L^{\infty}(\Gamma),f_{\delta_0}}L\isom L\bold{e}_{\delta}:(\eta\bold{e}_{\delta_{\lambda}})\otimes a\mapsto af_{\delta_0}(\eta)\bold{e}_{\delta}.$$

Therefore, to verify the condition (vi) when $\mathcal{R}_L(\delta)$ is de Rham, 
we need to relate the map $\overline{\theta}_{\zeta}(\mathcal{R}_L(\delta))$ with 
the Bloch-Kato's exponential map or the dual exponential map. 

To do so, we divide into the following three cases:
\begin{itemize}
\item[(1)]$\delta\not = x^{-k}, x^{k+1}|x|$ for any $k\in \mathbb{Z}_{\geqq 0}$ (which we call 
the generic case).
\item[(2)]$\delta=x^{-k}$ for $k\geqq 0$.
\item[(3)]$\delta=x^{k+1}|x|$ for $k\geqq 0$.

\end{itemize}
We will first verify the condition (vi) in the generic case via establishing a kind of 
explicit reciprocity laws (see Proposition \ref{4.14} and Proposition \ref{4.15}). 
Then, we will verify the conditions (iv) and (v) using the generic case 
by density argument. Finally, we will prove the condition (vi) in the case (2) via direct calculations, and 
reduce the case (3) to the case (2) using the duality condition (iv). 

In the remaining parts, we freely use the results of Colmez
and Chenevier concerning 
the calculations of cohomologies $\mathrm{H}^i_{\psi,\gamma}(\mathcal{R}_L(\delta)), 
\mathrm{H}^i_{\psi,\gamma}(\mathcal{R}^{\infty}_L(\delta))$ and $\mathrm{H}^i_{\psi,\gamma}
(\mathrm{LA}(\mathbb{Z}_p, \allowbreak L)(\delta))$; see Proposition 2.1 and 
Th\'eor\`eme 2.9 of \cite{Co08} and Lemme 2.9 and Corollaire 2.11 of \cite{Ch13}.

\subsubsection{Verification of the condition $(vi)$ in the generic case}

In this subsection, we assume that $\delta$ is generic. 
Then, we have 
$$\mathrm{H}^i_{\psi,\gamma}(Lt^k\bold{e}_{\delta})=\mathrm{H}^i_{\psi,\gamma}(Ly^k\bold{e}_{\delta})=\mathrm{H}^i_{\psi,\gamma}(\mathrm{LA}(\mathbb{Z}_p, L)(\delta))=0$$ for any $k\in \mathbb{Z}_{\geqq 0}$ and $i\in \{0,1,2\}$, and 
$$\mathrm{H}^i_{\psi,\gamma}(\mathcal{R}_L(\delta))=\mathrm{H}^{i}_{\psi,\gamma}(\mathcal{R}^{\infty}_L(\delta))=0$$ for $i=0,2$ ,and 
$$\mathrm{dim}_L\mathrm{H}^1_{\psi,\gamma}(\mathcal{R}_L(\delta))=\mathrm{dim}_L\mathrm{H}^1_{\psi,\gamma}(\mathcal{R}^{\infty}_L(\delta))=1.$$

Then, $\iota_{1,\delta}:=\iota_1\otimes_{\mathcal{R}^{\infty}_L(\Gamma),f_{\delta_0}}\mathrm{id}_L$ 
(see (\ref{35f})) is the isomorphism 
\begin{equation}
(\mathrm{H}^1_{\psi,\gamma}(\mathcal{R}_L(\delta)),1)^{-1}
\isom (\mathrm{H}^1_{\gamma}(\mathcal{R}_L^{\infty}(\delta)^{\psi=1}),1)^{-1}
\end{equation}
in $\mathcal{P}_L$ induced by the isomorphism 
$$\mathrm{H}^1_{\gamma}(\mathcal{R}_L^{\infty}(\delta)^{\psi=1})\isom
 \mathrm{H}^1_{\psi,\gamma}(\mathcal{R}_L(\delta)):[x]\mapsto [x,0].$$

Then, the base change by $f_{\delta_0}$ of the isomorphism 
$$\mathrm{Det}_{\mathcal{R}_L^{\infty}(\Gamma)}(C^{\bullet}_{\psi}(\mathcal{R}_L^{\infty}(\delta_{\lambda})))^{-1}\isom \mathrm{Det}_{\mathcal{R}_L^{\infty}(\Gamma)}(\mathcal{R}_L^{\infty}(\delta_{\lambda})^{\psi=0}[0])\isom (\mathcal{R}_L^{\infty}(\Gamma)\bold{e}_{\delta_{\lambda}}, 1)$$ 
which is induced by  (\ref{39a}) , (\ref{40a}), (\ref{40f}) and  (\ref{41f}) becomes the following isomorphisms
\begin{equation}\label{45f}
(\mathrm{H}^1_{\gamma}(\mathcal{R}_L^{\infty}(\delta)^{\psi=1}),1)\xrightarrow{[x]\mapsto [(1-\varphi)x]}
(\mathrm{H}^1_{\gamma}(\mathcal{R}_L^{\infty}(\delta)^{\psi=0}),1),
\isom (L\bold{e}_{\delta},1)
\end{equation}
 where the last isomorphism is explicitly defined as follows.
For an explicit definition of this isomorphism, it is useful to use the Amice transform.
 Let $D(\mathbb{Z}_p,L):=\mathrm{Hom}^{\mathrm{cont}}_L(\mathrm{LA}(\mathbb{Z}_p,L), L)$ be the algebra of $L$-valued distributions on $\mathbb{Z}_p$, where the multiplication is defined by the convolution. By 
the theorem of Amice, we have an isomorphism of topological $L$-algebras
$$D(\mathbb{Z}_p, L)\isom \mathcal{R}^{\infty}_L:\mu\mapsto f_{\mu}(\pi):=\sum_{n\geqq 0}
\mu\left(\binom yn\right)\pi^n$$ 
( which depends on the choice of $\pi$, i.e. the choice of $\zeta$) where $\binom yn:=\frac{y(y-1)\cdots (y-n+1)}{n!}$. Then, the action of $(\varphi,\Gamma,\psi)$ on $\mathcal{R}^{\infty}_L$ induces the action on 
$D(\mathbb{Z}_p, L)$ by
$$\int_{\mathbb{Z}_p}f(y)\varphi(\mu)(y):=\int_{\mathbb{Z}_p}f(py)\mu(y), \,\,\,\int_{\mathbb{Z}_p}f(y)\psi(\mu)(y):=\int_{p\mathbb{Z}_p}f\left(\frac{y}{p}\right)\mu(y)$$
and
$$\int_{\mathbb{Z}_p}f(y)\sigma_a(\mu)(y):=\int_{\mathbb{Z}_p}f(ay)\mu(y),$$ where, for $a\in \mathbb{Z}_p^{\times}$, we define $\sigma_a\in \Gamma$ such that $\chi(\sigma_a)=a$.

Using this notion, it is easy to see that the second isomorphism in (\ref{45f}) is defined by 
$$\mathrm{H}^1_{\gamma}(\mathcal{R}^{\infty}_L(\delta)^{\psi=0})\isom L\bold{e}_{\delta}:
[f_{\mu}\bold{e}_{\delta}]\mapsto \frac{\delta(-1)}{|\Gamma_{\mathrm{tor}}|\mathrm{log}_0(\chi(\gamma))}\cdot\int_{\mathbb{Z}_p^{\times}}\delta^{-1}(y)\mu(y)\bold{e}_{\delta},$$
where we remark that we have an isomorphism 
$D(\mathbb{Z}_p^{\times}, L)\bold{e}_{\delta}\isom \mathcal{R}^{\infty}_L(\delta)^{\psi=0}
:\mu\bold{e}_{\delta}\mapsto f_{\mu}\bold{e}_{\delta}$, since one has 
$$f_{\delta_0}(\lambda)=\int_{\mathbb{Z}_p^{\times}}\delta_0^{-1}(y)\mu_{\gamma}(y)$$ 
for any $\lambda\in \mathcal{R}_L^{\infty}(\Gamma)$ and any continuous homomorphism 
$\delta_0:\mathbb{Z}_p^{\times}\rightarrow L^{\times}$, where we define $\mu_{\gamma}\in D(\mathbb{Z}_p^{\times}, L)$ by $f_{\mu_{\gamma}}(\pi)=\lambda\cdot (1+\pi)$.

For a $\Gamma$-module $N$, we define $\mathrm{H}^1(\Gamma,N):=N/N_0$, where $N_0$ is the submodule generated by the set $\{(\gamma-1)n| \gamma\in \Gamma, n\in N\}$. 
Then, we have the following canonical isomorphism 
$$\mathrm{H}^1(\Gamma,\mathcal{R}^{\infty}_L(\delta)^{\psi=1})
\isom \mathrm{H}^1_{\gamma}(\mathcal{R}_L^{\infty}(\delta)^{\psi=1}):[f\bold{e}_{\delta}]\mapsto[ |\Gamma_{\mathrm{tor}}|\mathrm{log}_0(\chi(\gamma))
p_{\Delta}(f\bold{e}_{\delta})]$$
(where ``canonical" means that this is independent of $\gamma$, i.e. is compatible with the isomorphisms $\iota_{\gamma,\gamma'}$ for any $\gamma'\in \Gamma$). Composing this with the isomorphism (\ref{45f}), we obtain an isomorphism 
\begin{equation}\label{46f}
(\mathrm{H}^1(\Gamma,\mathcal{R}_L(\delta)^{\psi=1}),1)
\isom (L\bold{e}_{\delta},1)
\end{equation}
in $\mathcal{P}_L$. 
Concerning the explicit description of this isomorphism, we obtain the following lemma.
\begin{lemma}\label{4.13}
The isomorphism $(\ref{46f})$ is induced by the following isomorphism
$$\iota_{\delta}:\mathrm{H}^1(\Gamma,\mathcal{R}^{\infty}_L(\delta)^{\psi=1})\isom 
L\bold{e}_{\delta}: [f_{\mu}\bold{e}_{\delta}]\mapsto \delta(-1)\cdot\int_{\mathbb{Z}_p^{\times}}\delta^{-1}(y)\mu(y).$$
\end{lemma}
\begin{proof}
For $f_{\mu}\bold{e}_{\delta}\in \mathcal{R}_L^{\infty}(\delta)^{\psi=1}$, we have 
$(1-\varphi)(f_{\mu}\bold{e}_{\delta})=((1-\varphi\psi)f_{\mu})\cdot \bold{e}_{\delta}$. 
Then, the lemma follows from the formula
$$\int_{\mathbb{Z}_p}f(x)(1-\varphi\psi)\mu(x)=\int_{\mathbb{Z}_p^{\times}}f(x)\mu(x).$$
for $\mu\in D(\mathbb{Z}_p,L)$.

\end{proof}

Next, we furthermore assume that $\mathcal{R}_L(\delta)$ is de Rham. By the classification, it is equivalent to 
$\delta=\tilde{\delta}x^k$ for $k\in \mathbb{Z}$ and a locally constant homomorphism 
$\tilde{\delta}:\mathbb{Q}_p^{\times}\rightarrow L^{\times}$. In the generic  case, we have 
the following isomorphisms of one dimensional $L$-vector spaces.
\begin{itemize}
\item[(1)]$\mathrm{exp}^*_{\mathcal{R}_L(\delta)^*}:\mathrm{H}^1_{\psi,\gamma}(\mathcal{R}_L(\delta))\isom \bold{D}_{\mathrm{dR}}(\mathcal{R}_L(\delta))$ if $k\leqq 0$.
\item[(2)]$\mathrm{exp}_{\mathcal{R}_L(\delta)}:\bold{D}_{\mathrm{dR}}(\mathcal{R}_L(\delta))\isom 
\mathrm{H}^1_{\psi,\gamma}(\mathcal{R}_L(\delta))$ if $k\geqq 1$.

\end{itemize}

Let define $n(\delta)\in \mathbb{Z}_{\geqq 0}$ as the minimal integer such that 
$\tilde{\delta}|_{(1+p^n\mathbb{Z}_p)\cap \mathbb{Z}_p^{\times}}$ is trivial. 
Then, we have the following facts:
\begin{itemize}
\item[(1)]$n(\delta)=0$ if and only if $\mathcal{R}_L(\delta)$ is crystalline.
\item[(2)]$\varepsilon_L(W(\mathcal{R}_L(\delta)),\zeta)=1$ if $n(\delta)=0$.
\item[(3)]$\varepsilon_L(W(\mathcal{R}_L(\delta)),\zeta)=\tilde{\delta}(p)^{n(\delta)}\sum_{i\in 
(\mathbb{Z}/p^{n(\delta)}\mathbb{Z})^{\times}}\tilde{\delta}(i)^{-1}\zeta_{p^{n(\delta)}}^i$ if $n(\delta)\geqq 1$.

\item[(4)]$\varepsilon_L(W(\mathcal{R}_L(\delta)),\zeta)\cdot\varepsilon_L(W(\mathcal{R}_L(\delta)^*),\zeta)
=\tilde{\delta}(-1)$.

\end{itemize}

By definition of $\varepsilon_{L,\zeta}(\mathcal{R}_L(\delta))$ and $\varepsilon^{\mathrm{dR}}_{L,\zeta}(\mathcal{R}_L(\delta))$, and by Lemma \ref{4.13}, to verify the condition (vi), it suffices to show the following two propositions \ref{4.14} (for $k\leqq 0$) and \ref{4.15} (for $k\geqq 1$), which can be seen as a kind of explicit reciprocity laws.

\begin{prop}\label{4.14}
If $k\leqq 0$, 
then the following map 
$$ \mathrm{H}^1(\Gamma,\mathcal{R}^{\infty}_L(\delta)^{\psi=1})
\isom \mathrm{H}^1_{\psi,\gamma}(\mathcal{R}_L(\delta))
\xrightarrow{\mathrm{exp}^*_{\mathcal{R}_L(\delta)^*}} 
\bold{D}_{\mathrm{dR}}(\mathcal{R}_L(\delta))=\left(\frac{1}{t^k}L_{\infty}\bold{e}_{\delta}\right)^{\Gamma}$$ 
$($where the first isomorphism is defined by $[f\bold{e}_{\delta}]\mapsto 
[|\Gamma_{\mathrm{for}}|\mathrm{log}_0(\chi(\gamma))p_{\Delta}(f\bold{e}_{\delta}),0]$$)$ sends each element 
$[f_{\mu}\bold{e}_{\delta}]\in\mathrm{H}^1(\Gamma, \mathcal{R}^{\infty}_L(\delta)^{\psi=1})$ to

\begin{itemize}
\item[(1)]$\frac{(-1)^k}{(-k)!}\cdot\frac{\delta(-1)}{\varepsilon_L(W(\mathcal{R}_L(\delta)), \zeta)}\cdot \frac{1}{t^{k}}\cdot
\int_{\mathbb{Z}_p^{\times}}\delta^{-1}(y)\mu(y)\bold{e}_{\delta}$ if  $n(\delta)\not=0$,
\item[(2)]$\frac{(-1)^k}{(-k)!}\cdot\frac{\mathrm{det}_L(1-\varphi|\bold{D}_{\mathrm{cris}}(\mathcal{R}_L(\delta)^*))}{\mathrm{det}_L(1-\varphi|\bold{D}_{\mathrm{cris}}(\mathcal{R}_L(\delta)))}\cdot\frac{\delta(-1)}{t^{k}}\cdot\int_{\mathbb{Z}_p^{\times}}
\delta^{-1}(y)\mu(y)\bold{e}_{\delta}$ if $n(\delta)= 0.$ 
\end{itemize}

\end{prop}

\begin{proof}

Here, we prove the proposition only when $k=0$, i.e. $\delta=\tilde{\delta}$ is locally constant. 
We will prove it for general $k\leqq 0$ after some preparations on the differential operator $\partial$ 
(the proof for general $k$ will be given after Remark \ref{4.19f}). 

Hence, we assume that $k=0$. For such $\delta$, we define a map
$$g_{\mathcal{R}_L(\delta)}:\bold{D}_{\mathrm{dR}}(\mathcal{R}_L(\delta))\rightarrow \mathrm{H}^1_{\gamma}(
\bold{D}_{\mathrm{dif}}(\mathcal{R}_L(\delta))): x\mapsto [\mathrm{log}(\chi(\gamma))x],$$
which is easily seen to be isomorphism.
By Proposition 2.16 of \cite{Na14a}, one has the following commutative diagram
\begin{equation}
\begin{CD}
\mathrm{H}^1_{\psi,\gamma}(\mathcal{R}_L(\delta))
@> \mathrm{exp}^*_{\mathcal{R}_L(\delta)^*} >> \bold{D}_{\mathrm{dR}}(\mathcal{R}_L(\delta)) \\
@ V \mathrm{id} VV @ V g_{\mathcal{R}_L(\delta)} VV \\
\mathrm{H}^1_{\psi,\gamma}(\mathcal{R}_L(\delta))
@> \mathrm{can} >> \mathrm{H}^1_{\gamma}(\bold{D}_{\mathrm{dif}}(\mathcal{R}(\delta)))
\end{CD}
\end{equation}

Set $n_0:=\mathrm{max}\{n(\delta), 1\}$ if $p\not=2$, and set 
$n_0:=\mathrm{max}\{n(\delta), 2\}$ if $p=2$. Then, the image of 
$[f_{\mu}\bold{e}_{\delta}]\in \mathrm{H}^1(\Gamma,\mathcal{R}^{\infty}_L(\delta)^{\psi=1})
\isom \mathrm{H}^1_{\psi,\gamma}(\mathcal{R}_L(\delta))$ by the canonical map 
$\mathrm{can}:\mathrm{H}^1_{\psi,\gamma}(\mathcal{R}_L(\delta))
\rightarrow \mathrm{H}^1_{\gamma}(\bold{D}_{\mathrm{dif}}(\mathcal{R}(\delta)))$ 
is equal to 
$$[|\Gamma_{\mathrm{tor}}|\mathrm{log}_0(\chi(\gamma))p_{\Delta}(\iota_{n_0}(f_{\mu}\bold{e}_{\delta})))]\in  \mathrm{H}^1_{\gamma}(\bold{D}_{\mathrm{dif}}(\mathcal{R}(\delta))).$$
Hence, it suffices to calculate $g_{\mathcal{R}_L(\delta)}^{-1}([|\Gamma_{\mathrm{tor}}|\mathrm{log}_0(\chi(\gamma))p_{\Delta}(\iota_{n_0}(f_{\mu}\bold{e}_{\delta}))])$.
By definition of $g_{\mathcal{R}_L(\delta)}$, it is easy to check that we have 
\begin{multline*}
g_{\mathcal{R}_L(\delta)}^{-1}[|\Gamma_{\mathrm{tor}}|\mathrm{log}_0(\chi(\gamma))p_{\Delta}(\iota_{n_0}(f_{\mu}\bold{e}_{\delta}))])\\
=\frac{|\Gamma_{\mathrm{tor}}|\mathrm{log}_0(\chi(\gamma))}
{\mathrm{log}(\chi(\gamma))}\frac{1}{[\mathbb{Q}_p(\zeta_{p^{n_0}}):\mathbb{Q}_p]}
\sum_{i\in (\mathbb{Z}/p^{n_0}\mathbb{Z})^{\times}}\sigma_i(\iota_{n_0}(f_{\mu}\bold{e}_{\delta})|_{t=0})=:(*).
\end{multline*}

Concerning the right hand side, when $n(\delta)\geqq 1$ if $p\not=2$, or
$n(\delta)\geqq 2$ if $p=2$, 
one has the following equalities, from which the equality (1) follows in this case,
 \[
\begin{array}{ll}
(*)
&=\frac{|\Gamma_{\mathrm{tor}}|\mathrm{log}_0(\chi(\gamma))}
{\mathrm{log}(\chi(\gamma))} \frac{1}{[\mathbb{Q}_p(\zeta_{p^{n(\delta)}}):\mathbb{Q}_p]}
\sum_{i\in (\mathbb{Z}/p^{n(\delta)}\mathbb{Z})^{\times}}
\sigma_i(\iota_{n(\delta)}(f_{\mu}\bold{e}_{\delta})|_{t=0})\\
&=\frac{|\Gamma_{\mathrm{tor}}|\mathrm{log}_0(\chi(\gamma))}
{\mathrm{log}(\chi(\gamma))}\frac{p}{(p-1)}\frac{1}{p^{n(\delta)}}
\sum_{i\in (\mathbb{Z}/p^{n(\delta)}\mathbb{Z})^{\times}}
\sigma_i(\frac{1}{\delta(p)^{n(\delta)}}\int_{\mathbb{Z}_p}
\zeta_{p^{n(\delta)}}^y\mu(y)\bold{e}_{\delta})\\
&=
\frac{1}{(p\delta(p))^{n(\delta)}}
\sum_{i\in (\mathbb{Z}/p^{n(\delta)}\mathbb{Z})^{\times}}
\delta(i)\int_{\mathbb{Z}_p}\zeta_{p^{n(\delta)}}^{iy}\mu(y)\bold{e}_{\delta} \\
&=
\frac{1}{(p\delta(p))^{n(\delta)}}
\sum_{i\in (\mathbb{Z}/p^{n(\delta)}\mathbb{Z})^{\times}}
\delta(i)(\sum_{j\in \mathbb{Z}/p^{n(\delta)}\mathbb{Z}}
\zeta^{ij}_{p^{n(\delta)}}\int_{j+p^{n(\delta)}\mathbb{Z}_p}
\mu(y))\bold{e}_{\delta} \\
&=
\frac{1}{(p\delta(p))^{n(\delta)}}
\sum_{j\in \mathbb{Z}/p^{n(\delta)}\mathbb{Z}}
(\sum_{i\in (\mathbb{Z}/p^{n(\delta)}\mathbb{Z})^{\times}}
\delta(i)\zeta^{ij}_{p^{n(\delta)}})
\int_{j+p^{n(\delta)}\mathbb{Z}_p}
\mu(y)\bold{e}_{\delta} \\
&=
\frac{1}{(p\delta(p))^{n(\delta)}}
\sum_{j\in (\mathbb{Z}/p^{n(\delta)}\mathbb{Z})^{\times}}
(\sum_{i\in (\mathbb{Z}/p^{n(\delta)}\mathbb{Z})^{\times}}
\delta(i)\zeta^{ij}_{p^{n(\delta)}})
\int_{j+p^{n(\delta)}\mathbb{Z}_p}
\mu(y)\bold{e}_{\delta} \\
&=
\frac{1}{(p\delta(p))^{n(\delta)}}
(\sum_{i\in (\mathbb{Z}/p^{n(\delta)}\mathbb{Z})^{\times}}
\delta(i)\zeta^{i}_{p^{n(\delta)}})
\sum_{j\in (\mathbb{Z}/p^{n(\delta)}\mathbb{Z})^{\times}}
\delta(j)^{-1}
\int_{j+p^{n(\delta)}\mathbb{Z}_p}
\mu(y)\bold{e}_{\delta} \\
&=\varepsilon_L(W(\mathcal{R}_L(\delta)^*),\zeta)
\int_{\mathbb{Z}_p^{\times}}\delta^{-1}(y)\mu(y)\bold{e}_{\delta} \\
&=
\frac{\delta(-1)}{\varepsilon_L(W(\mathcal{R}_L(\delta)),\zeta)}
\int_{\mathbb{Z}_p^{\times}}\delta^{-1}(y)\mu(y)\bold{e}_{\delta},

\end{array}
\]
where the second equality follows form $\iota_{n(\delta)}(f_{\mu})|_{t=0}=f_{\mu}(\zeta_{p^{n(\delta)}}-1)=\int_{\mathbb{Z}_p}\zeta_{p^{n(\delta)}}^y\mu(y)$, the third equality follows from 
$\frac{|\Gamma_{\mathrm{for}}|\mathrm{log}_0(\chi(\gamma))}{\mathrm{log}(\chi(\gamma))}\frac{p}{p-1}=1$ (for any $p$), and the sixth equality follows from the fact that $(\sum_{i\in (\mathbb{Z}/p^{n(\delta)}\mathbb{Z})^{\times}}
\delta(i)\zeta^{ij}_{p^{n(\delta)}})=0$ if $p|j$, and the seventh and eighth follow from the property (4) of 
$\varepsilon$-constants listed before this proposition.

When $n(\delta)=0$, one has $n_0=1$ if $p\not=2$ and $n_0=2$ if $p=2$, 
 then one has the following equalities
\[
\begin{array}{ll}
(*)
&= \frac{1}{p^{n_0}}
\sum_{i\in (\mathbb{Z}/p^{n_0}\mathbb{Z})^{\times}}
\sigma_i(\iota_{n_0}(f_{\mu}\bold{e}_{\delta})|_{t=0})\\
&=
 \frac{1}{p^{n_0}}
\sum_{i\in (\mathbb{Z}/p^{n_0}\mathbb{Z})^{\times}}
\sigma_i(\frac{1}{\delta(p)^{n_0}}\int_{\mathbb{Z}_p}\zeta_{p^{n_0}}^y\mu(y)\bold{e}_{\delta})\\
&=\frac{1}{(p\delta(p))^{n_0}}
\sum_{i\in (\mathbb{Z}/p^{n_0}\mathbb{Z})^{\times}}
\int_{\mathbb{Z}_p}\zeta_{p^{n_0}}^{iy}\mu(y)\bold{e}_{\delta} \\
&=
\frac{1}{(p\delta(p))^{n_0}}
\sum_{i\in (\mathbb{Z}/p^{n_0}\mathbb{Z})^{\times}}
(\sum_{j\in \mathbb{Z}/p^{n_0}\mathbb{Z}}\zeta_{p^{n_0}}^{ij}\int_{j+p^{n_0}\mathbb{Z}_p}\mu(y))\bold{e}_{\delta} \\
&=
\frac{1}{(p\delta(p))^{n_0}}
\sum_{j\in \mathbb{Z}/p^{n_0}\mathbb{Z}}
(\sum_{i\in (\mathbb{Z}/p^{n_0}\mathbb{Z})^{\times}}\zeta_{p^{n_0}}^{ij})
\int_{j+p^{n_0}\mathbb{Z}_p}\mu(y)\bold{e}_{\delta},
\end{array}
\]

where the first equality follows form $\frac{|\Gamma_{\mathrm{tor}}|\mathrm{log}_0(\chi(\gamma))}
{\mathrm{log}(\chi(\gamma))}\frac{1}{[\mathbb{Q}_p(\zeta_{p^{n_0}}):\mathbb{Q}_p]}=\frac{1}{p^{n_0}}$ for any $p$.

When $p\not=2$, the last term is equal to 
$$
\frac{1}{p\delta(p)}
\left((p-1)\int_{p\mathbb{Z}_p}\mu(y) -\int_{\mathbb{Z}_p^{\times}}
\mu(y)\right)\bold{e}_{\delta}$$
since we have 
$\sum_{i\in (\mathbb{Z}/p\mathbb{Z})^{\times}}\zeta_{p}^{ij}=p-1$ if $p|j$ and 
$\sum_{i\in (\mathbb{Z}/p\mathbb{Z})^{\times}}\zeta_{p}^{ij}=-1$ if $p\not|j$.

Since $f_{\mu}\bold{e}_{\delta}\in \mathcal{R}^{\infty}(\delta)^{\psi=1}$, we have 
$\psi(f_{\mu})=\delta(p)f_{\mu}$, hence we have 
$$\int_{p\mathbb{Z}_p}\mu(y)=\int_{\mathbb{Z}_p}\psi(\mu)(y)
=\delta(p)\int_{\mathbb{Z}_p}\mu(y)=\delta(p)\left(\int_{\mathbb{Z}_p^{\times}}\mu(y)
+\int_{p\mathbb{Z}_p}\mu(y)\right),$$ 
and we have $$\int_{p\mathbb{Z}_p}\mu(y)=\frac{\delta(p)}{1-\delta(p)}\int_{\mathbb{Z}_p^{\times}}\mu(y)$$
since we have $\delta(p)\not=1$ by the generic assumption on $\delta$.

Therefore, we have
\[
\begin{array}{ll}

\frac{1}{p\delta(p)}
\left((p-1)\int_{p\mathbb{Z}_p}\mu(y) -\int_{\mathbb{Z}_p^{\times}}
\mu(y)\right)\bold{e}_{\delta}
&=
\frac{1}{p\delta(p)}
\left((p-1)\frac{\delta(p)}{1-\delta(p)}-1\right)\int_{\mathbb{Z}_p^{\times}}\mu(y)\bold{e}_{\delta} \\
&=
\frac{1}{p\delta(p)}
\frac{p\delta(p)-1}{1-\delta(p)}\int_{\mathbb{Z}_p^{\times}}\mu(y)\bold{e}_{\delta}\\
&=
\frac{1-\frac{1}{p\delta(p)}}{1-\delta(p))}\int_{\mathbb{Z}_p^{\times}}
\mu(y)\bold{e}_{\delta},

\end{array}
\]
from which we obtain the equality (2) for $p\not=2$.

When $p=2$, then the last term is equal to 
$$\frac{1}{(p\delta(p))^2}\left(2\int_{4\mathbb{Z}_2}\mu(y)-2\int_{2+4\mathbb{Z}_2}\mu(y)\right)\bold{e}_{\delta}=\frac{1}{p\delta(p)^2}\left(\int_{4\mathbb{Z}_2}\mu(y)-\int_{2\mathbb{Z}_2}\mu(y)\right)\bold{e}_{\delta}$$ since we have 
$\sum_{i\in (\mathbb{Z}/4\mathbb{Z})^{\times}}\zeta_{4}^{ij}$ is 
equal to $2$ if $j\equiv 0$ (mod $4$), is equal to $0$ if $j\equiv 1, 3$ (mod $4$),  and is 
equal to $-2$ if $j\equiv 2$ (mod $4$).
Since we have $\psi(f_{\mu})=\delta(p)f_{\mu}$, we have 
$$\int_{4\mathbb{Z}_p}\mu(y)=\int_{2\mathbb{Z}_2}\psi(\mu)(y)
=\delta(p)\int_{2\mathbb{Z}_2}\mu(y)
=\delta(p)\frac{\delta(p)}{1-\delta(p)}\int_{\mathbb{Z}^{\times}_2}\mu(y)$$
where the last equality follows from the same argument for $p\not=2$.

Therefore, we have 
\[
\begin{array}{ll}
\frac{1}{p\delta(p)^2}\left(\int_{4\mathbb{Z}_2}\mu(y)-\int_{2\mathbb{Z}_2}\mu(y)\right)\bold{e}_{\delta}
&=\frac{1}{p\delta(p)^2}\left(\delta(p)\frac{\delta(p)}{1-\delta(p)}-\frac{\delta(p)}{1-\delta(p)}\right)\int_{\mathbb{Z}_p^{\times}}\mu(y)\bold{e}_{\delta}\\
&=\frac{1}{p\delta(p)^2}\frac{2\delta(p)^2-\delta(p)}{1-\delta(p)}\int_{\mathbb{Z}_2^{\times}}\mu(y)\bold{e}_{\delta}\\
&=\frac{1-\frac{1}{p\delta(p)}}{1-\delta(p)}\int_{\mathbb{Z}_2^{\times}}\mu(y)\bold{e}_{\delta},
\end{array}
\]
from which we obtain the equality (2) for $p=2$.

\end{proof}

To prove the above proposition for general $k\leqq 0$, we need to recall and prove
some facts on the differential operator $\partial $ defined in \S2.4 of \cite{Co08}, which 
will be used to reduce the verification of the condition (vi) for general $k$ to that for 
$k=0,1$ (even for the non generic case).

Let $A$ be a $\mathbb{Q}_p$-affinoid algebra.
We define an $A$-linear differential operator $\partial:\mathcal{R}_A\rightarrow \mathcal{R}_A: 
f(\pi)\mapsto (1+\pi)\frac{df(\pi)}{d\pi}$. Let $\delta:\mathbb{Q}_p^{\times}\rightarrow A^{\times}$ be a continuous homomorphism, then $\partial$ naturally induces an $A$-linear and $(\varphi,\Gamma)$-equivariant morphism
$$\partial:\mathcal{R}_A(\delta)\rightarrow \mathcal{R}_A(\delta x): 
f(\pi)\bold{e}_{\delta}\mapsto \partial(f(\pi))\bold{e}_{\delta x},$$
which sits in the following  exact sequence
\begin{equation}\label{48f}
0\rightarrow A(\delta)\xrightarrow{a\bold{e}_{\delta}\mapsto a\bold{e}_{\delta}} \mathcal{R}_A(\delta)\xrightarrow{\partial} 
\mathcal{R}_A(\delta x)\xrightarrow{f\bold{e}_{\delta x}\mapsto 
\mathrm{Res}_0\left(f\frac{d\pi}{1+\pi}\right)\bold{e}_{\delta |x|^{-1}}} A(\delta |x|^{-1})\rightarrow 0.
\end{equation}
By this exact sequence, when $A=L$ is a finite extension of $\mathbb{Q}_p$, we immediately 
obtain the following lemma.
\begin{lemma}\label{4.12.5}
$\partial:C^{\bullet}_{\varphi,\gamma}(\mathcal{R}_L(\delta))\rightarrow 
C^{\bullet}_{\varphi,\gamma}(\mathcal{R}_L(\delta x))$ 
is quasi-isomorphism except when 
$\delta=\bold{1}, |x|$.

\end{lemma}

For general case, the exact sequence (\ref{48f}) induces the following canonical isomorphism 
\begin{equation}\label{49f}
\mathrm{Det}_A(C^{\bullet}_{\varphi,\gamma}(A(\delta)))\boxtimes \Delta_{A,1}(\mathcal{R}_A(\delta))^{-1}\boxtimes \Delta_{A,1}(\mathcal{R}_A(\delta x))\boxtimes 
\mathrm{Det}_A(C^{\bullet}_{\varphi,\gamma}(A(\delta|x|^{-1})))^{-1}\isom \bold{1}_A.
\end{equation}

For $\delta'=\delta, \delta |x|^{-1}$, since $A(\delta')$ is free $A$-module, the complex 
$C^{\bullet}_{\varphi,\gamma}(A(\delta')):[A(\delta')_1^{\Delta}\xrightarrow{(\gamma-1)\oplus (\varphi-1)}A(\delta')_2^{\Delta}\oplus A(\delta')_3^{\Delta}\xrightarrow{(\varphi-1)\oplus (1-\gamma)} A(\delta')_4^{\Delta}]$ (where $A(\delta')_i=A(\delta')$ for $i=1,\cdots, 4$) induces the following 
canonical isomorphism 
\begin{multline*}
\mathrm{Det}_A(C^{\bullet}_{\varphi,\gamma}(A(\delta')))=(\mathrm{Det}_A(A(\delta')_1^{\Delta})\boxtimes 
\mathrm{Det}_A(A(\delta')_3^{\Delta})^{-1})\boxtimes (\mathrm{Det}_A(A(\delta')_4^{\Delta})
\boxtimes \mathrm{Det}_A(A(\delta')_2^{\Delta})^{-1})\\
\xrightarrow{i_{\mathrm{Det}_A(A(\delta')_1^{\Delta})}\boxtimes i_{\mathrm{Det}_A(A(\delta')_4^{\Delta})}} \bold{1}_A.
\end{multline*}

Applying this isomorphism,  the isomorphism (\ref{49f}) becomes the isomorphism $\Delta_{A,1}(\mathcal{R}_A(\delta))^{-1}\boxtimes \Delta_{A,1}(\mathcal{R}_A(\delta x))\isom \bold{1}_A$, 
and then, multiplying by $\Delta_{A,1}(\mathcal{R}_A(\delta))$ on the both sides, we obtain the following isomorphism which we also denote by $\partial$
$$\partial:\Delta_{A,1}(\mathcal{R}_A(\delta))\isom \Delta_{A,1}(\mathcal{R}_A(\delta))
\boxtimes (\Delta_{A,1}(\mathcal{R}_A(\delta))^{-1}\boxtimes \Delta_{A,1}(\mathcal{R}_A(\delta x)))
\xrightarrow{i_{\Delta_{A,1}(\mathcal{R}_A(\delta))^{-1}}}\Delta_{A,1}(\mathcal{R}_A(\delta x)).$$

Taking the product of this isomorphism with the isomorphism 
$$\Delta_{A,2}(\mathcal{R}_A(\delta))\isom \Delta_{A,2}(\mathcal{R}_A(\delta x)):a\bold{e}_{\delta}\mapsto 
-a\bold{e}_{\delta x},$$ 
we obtain the isomorphism 
$$\partial:\Delta_A(\mathcal{R}_A(\delta))\isom \Delta_A(\mathcal{R}_A(\delta x)).$$
By definition, it is clear that this isomorphism is compatible with any base change 
$A\rightarrow A'$.

Concerning this isomorphism, we prove the following proposition.
\begin{prop}\label{4.13.5}
$$\varepsilon_{A,\zeta}(\mathcal{R}_A(\delta x))=\partial\circ\varepsilon_{A,\zeta}(\mathcal{R}_A(\delta )).$$

\end{prop}
 \begin{proof}
 The proof of this proposition is a typical density argument, which will be used several times later.
 
 Define the unramified homomorphism 
$\delta_{Y}:\mathbb{Q}_p^{\times}\rightarrow \Gamma(\mathbb{G}_m^{\mathrm{an}},
\mathcal{O}_{\mathbb{G}_m^{\mathrm{an}}})^{\times}$ by $\delta_Y(p):=Y$ (where 
$Y$ is the parameter of $\mathbb{G}_m^{\mathrm{an}}$), then $\mathcal{R}_A(\delta)$ is 
obtained as a base change of the ``universal" rank one $(\varphi,\Gamma)$-module 
$\bold{Dfm}(\mathcal{R}_{\mathbb{G}_m^{\mathrm{an}}}(\delta_Y))$ over 
$\mathcal{R}_{X\times \mathbb{G}_m^{\mathrm{an}}}$ ($X$ is the rigid analytic space associated to 
$\mathbb{Z}_p[[\Gamma]]$). Since the isomorphism 
$\partial:\Delta_A(\mathcal{R}_A(\delta))\isom \Delta_A(\mathcal{R}_A(\delta x))$ is compatible 
with any base change, it suffices to show the proposition for $\bold{Dfm}(\mathcal{R}_{\mathbb{G}_m^{\mathrm{an}}}(\delta_Y))$. Since 
 $X\times \mathbb{G}_m^{\mathrm{an}}$ is reduced, it suffices to show it for 
the Zariski dense subset $S_0$ of $X\times \mathbb{G}_m^{\mathrm{an}}$ defined by 
 $$S_0:=\{(\delta_0,\lambda)\in X(L)\times \mathbb{G}_m^{\mathrm{an}}(L)| 
 L \text{ is a finite extension of } \mathbb{Q}_p, \delta:=\delta_{\lambda}\delta_0 
 \text{ is generic } \}.$$ 
 For any $(\delta_0,\lambda)$ in $S_0(L)$, $\varepsilon_{L,\zeta}(\mathcal{R}_L(\delta))$ corresponds to the isomorphism 
 $$\iota_{\delta}:\mathrm{H}^1(\Gamma, \mathcal{R}^{\infty}_L(\delta)^{\psi=1})\isom L\bold{e}_{\delta}
 :[f_{\mu}\bold{e}_{\delta}]\mapsto \delta(-1)\cdot
 \int_{\mathbb{Z}_p^{\times}}\delta^{-1}(y)\mu(y)\bold{e}_{\delta}$$ 
 by Lemma \ref{4.13} and by the arguments before this lemma. Then, the equality 
 $\varepsilon_{L,\zeta}(\mathcal{R}_L(\delta x))=\partial\circ\varepsilon_{L,\zeta}(\mathcal{R}_L(\delta ))$ 
 is equivalent to the commutativity of the following diagram
 \begin{equation*}
  \begin{CD}
\mathrm{H}^1(\Gamma,\mathcal{R}^{\infty}_L(\delta)^{\psi=1}) @> \iota_{\delta} >> L\bold{e}_{\delta}  \\
@V \partial VV @V \bold{e}_{\delta}\mapsto -\bold{e}_{\delta x} VV \\
\mathrm{H}^1(\Gamma, \mathcal{R}^{\infty}_L(\delta x)^{\psi=1}) @> \iota_{\delta x} >> L\bold{e}_{\delta x} .
\end{CD}
\end{equation*}

Finally, this commutativity follows from the formula
$$\int_{\mathbb{Z}_p}f(y)\partial(\mu)(y)=\int_{\mathbb{Z}_p}yf(y)\mu(y)$$
 for any $f(y)\in \mathrm{LA}(\mathbb{Z}_p, L)$, which finishes to prove 
 the proposition.
 \end{proof}
 
 We next prove the compatibility of $\partial$ with the de Rham $\varepsilon$-isomorphism
 $\varepsilon_{L,\zeta}^{\mathrm{dR}}(\mathcal{R}_L(\delta))$ for de Rham 
 rank one $(\varphi,\Gamma)$-modules $\mathcal{R}_L(\delta)$ 
 under a condition on the Hodge-Tate
 weight of $\mathcal{R}_L(\delta)$ as below.
 \begin{lemma}\label{4.14.5}
 Let $\mathcal{R}_L(\delta)$ be a de Rham $(\varphi,\Gamma)$-module $($here we don't assume 
 that $\delta$ is generic$)$. 
 If the Hodge-Tate weight of $\mathcal{R}_L(\delta)$ is not zero, i.e. we have 
 $\delta=\tilde{\delta} x^k$  such that $k\not=0$, then we have the equality 
 $$\varepsilon_{L,\zeta}^{\mathrm{dR}}(\mathcal{R}_L(\delta x))
 =\partial\circ\varepsilon_{L,\zeta}^{\mathrm{dR}}(\mathcal{R}_L(\delta)).$$

 \end{lemma}
 
 \begin{proof}
Since one has $\bold{D}_{\mathrm{dR}}(\mathcal{R}_L(\delta))
 =\left(L_{\infty}\frac{1}{t^k}\bold{e}_{\delta}\right)^{\Gamma}$ and $\partial(g(t))=\frac{dg(t)}{dt}$ for $g(t)\in L_{\infty}((t))$, 
 the differential operator $\partial$ naturally induces an isomorphism
 $$\partial:\bold{D}_{\mathrm{dR}}(\mathcal{R}_L(\delta))\rightarrow \bold{D}_{\mathrm{dR}}(\mathcal{R}_L(\delta x)): \frac{a}{t^k}\bold{e}_{\delta}\mapsto (-k)\frac{a}{t^{k+1}}\bold{e}_{\delta x}$$
 under the condition $k\not=0$. 
Hence, by definition of 
 of $\varepsilon_{L,\zeta}^{\mathrm{dR}}(M)$ using the isomorphisms $\theta_{L}(M)$ and $\theta_{\mathrm{dR},L}(M,\zeta)$ and the constant $\Gamma_L(M)$ 
 in \S 3.2, it suffices to show the following two equalities:
 \begin{itemize}
 \item[(1)]$\theta_L(\mathcal{R}_L(\delta x))=\partial\circ\theta_L(\mathcal{R}_L(\delta))$,
 \item[(2)]$\Gamma_L(\mathcal{R}_L(\delta))\cdot\partial\circ\theta_{\mathrm{dR},L}(\mathcal{R}_L(\delta),\zeta)=\Gamma_L(\mathcal{R}_L(\delta x))\cdot\theta_{\mathrm{dR},L}(\mathcal{R}_L(\delta x),\zeta)\circ \partial.$
 \end{itemize}
  We first prove the equality (2). Since one has $\Gamma_L(\mathcal{R}_L(\delta))=\Gamma^*(k)$ and 
  $\Gamma_L(\mathcal{R}_L(\delta x))=\Gamma^*(k+1)$, it suffices to show that the following diagram is commutative:
\begin{equation*}
\begin{CD}
\mathcal{L}_{L}(\mathcal{R}_L(\delta))=L\bold{e}_{\delta} @> \Gamma^*(k)\cdot f_{\mathcal{R}_L(\delta),\zeta}>>
\bold{D}_{\mathrm{dR}}(\mathcal{R}_L(\delta))\\
@VV \bold{e}_{\delta}\mapsto -\bold{e}_{\delta x} V @VV \partial V\\
\mathcal{L}_{L}(\mathcal{R}_L(\delta x))=L\bold{e}_{\delta x} @> \Gamma^*(k+1)\cdot f_{\mathcal{R}_L(\delta x),\zeta} >>
\bold{D}_{\mathrm{dR}}(\mathcal{R}_L(\delta x)),
 \end{CD}
 \end{equation*}
where the map $f_{\mathcal{R}_L(\delta'),\zeta}$ (for $\delta'=\delta, \delta x$) is defined in Lemma \ref{3.4}. 
This commutativity is obvious by definition of $f_{\mathcal{R}_L(\delta_0),\zeta}$ since one has 
$\varepsilon_L(W(\mathcal{R}_L(\delta)),\zeta)=\varepsilon_L(W(\mathcal{R}_L(\delta x)),\zeta)$ (since one has a natural 
isomorphism $\bold{D}_{\mathrm{pst}}(\mathcal{R}_L(\delta))\isom \bold{D}_{\mathrm{pst}}(\mathcal{R}_L(\delta x)):\frac{a}{t^k}\bold{e}_{\delta}\mapsto \frac{a}{t^{k+1}}\bold{e}_{\delta x}$) and 
$k\cdot \Gamma^*(k)=\Gamma^*(k+1)$ for $k\not=0$. 
We next show the equality (1). Under the assumption that $k\not=0$,  it is easy to see that 
$\partial$ induces the isomorphisms
$$\bold{D}_{\mathrm{dR}}(\mathcal{R}_L(\delta))\isom \bold{D}_{\mathrm{dR}}(\mathcal{R}_L(\delta x)), \,\,
\bold{D}_{\mathrm{dR}}(\mathcal{R}_L(\delta))^0\isom \bold{D}_{\mathrm{dR}}(\mathcal{R}_L(\delta x))^0$$ and 
$$\bold{D}_{\mathrm{cris}}(\mathcal{R}_L(\delta))\isom \bold{D}_{\mathrm{cris}}(\mathcal{R}_L(\delta x)), \,\,\,\mathrm{H}^i_{\varphi,\gamma}(\mathcal{R}_L(\delta))\isom \mathrm{H}^i_{\varphi,\gamma}(\mathcal{R}_L(\delta x))$$ for any $i=0,1,2$  by Lemma \ref{4.12.5}. Hence, by definition of $\theta'_L(\mathcal{R}_L(\delta))$, it suffices to show 
that the following two diagrams are commutative for $M=\mathcal{R}_L(\delta)$:
\begin{equation}\label{10}
\begin{CD}
\mathrm{H}^0_{\varphi,\gamma}(M)@>>>\bold{D}_{\mathrm{cris}}(M)@>>> \bold{D}_{\mathrm{cris}}(M)\oplus t_M@>>> \mathrm{H}^1_{\varphi,\gamma}(M) \\
@VV \partial V @VV\partial V @VV\partial V@VV \partial V\\
\mathrm{H}^0_{\varphi,\gamma}(M(x))@>>>\bold{D}_{\mathrm{cris}}(M(x))@>>> \bold{D}_{\mathrm{cris}}(M(x))\oplus t_{M(x)}@>>> \mathrm{H}^1_{\varphi,\gamma}(M(x))
\end{CD}
\end{equation}
 and 
\begin{equation}\label{11}
\begin{CD}
\mathrm{H}^1_{\varphi,\gamma}(M)@>>>\bold{D}_{\mathrm{cris}}(M^*)^{\vee}\oplus 
\bold{D}_{\mathrm{dR}}(M)^0@>>> \bold{D}_{\mathrm{cris}}(M)^{\vee}@>>> \mathrm{H}^2_{\varphi,\gamma}(M)\\
@VV \partial V@VV(-\partial ^{\vee})\oplus \partial V @VV -\partial ^{\vee}V @VV \partial V\\
\mathrm{H}^1_{\varphi,\gamma}(M(x))@>>>\bold{D}_{\mathrm{cris}}(M(x)^*)^{\vee}\oplus 
\bold{D}_{\mathrm{dR}}(M(x))^0@>>> \bold{D}_{\mathrm{cris}}(M(x)^*)^{\vee}@>>> \mathrm{H}^2_{\varphi,\gamma}(M(x)),
\end{CD}
\end{equation}
where $\partial ^{\vee}$ is the dual of $\partial:\bold{D}_{\mathrm{cris}}(M(x)^*)=
\bold{D}_{\mathrm{cris}}(\mathcal{R}_L(\delta^{-1}|x|))\isom \bold{D}_{\mathrm{cris}}(\mathcal{R}_L(\delta^{-1}x|x|))=\bold{D}_{\mathrm{cris}}(M^*)$.
For the commutativity of the diagram (\ref{10}), the only non trivial part is the commutativity of 
the diagram
\begin{equation*}
\begin{CD}
\bold{D}_{\mathrm{cris}}(M)\oplus t_M@> \mathrm{exp}_{M,f}\oplus \mathrm{exp}_M >> \mathrm{H}^1_{\varphi,\gamma}(M) \\
@V \partial VV @ V \partial VV \\
\bold{D}_{\mathrm{cris}}(M(x))\oplus t_{M(x)}@> \mathrm{exp}_{M(x),f}\oplus \mathrm{exp}_{M(x)} >> \mathrm{H}^1_{\varphi,\gamma}(M(x)),
\end{CD}
\end{equation*}
but this commutativity easily follows from Proposition \ref{2.23}. 
Using the commutativity of (\ref{10}) for $M=\mathcal{R}_L(\delta^{-1}|x|)$, to prove the commutativity of (\ref{11}), it suffices to show the commutativities of the following diagrams:
\begin{equation}\label{12}
\begin{CD}
\bold{D}_{\mathrm{dR}}(M) @>>> \bold{D}_{\mathrm{dR}}(M^*)^{\vee}\\
@V \partial VV @ V -\partial^{\vee} VV \\
\bold{D}_{\mathrm{dR}}(M(x)) @>>> \bold{D}_{\mathrm{dR}}(M(x)^*)^{\vee}
\end{CD}
\end{equation}
and 
\begin{equation}\label{13}
\begin{CD}
\mathrm{H}^i_{\varphi,\gamma}(M) @>>> \mathrm{H}^{2-i}_{\varphi,\gamma}(M^*)^{\vee}\\
@V \partial VV @ V -\partial^{\vee} VV \\
\mathrm{H}^i_{\varphi,\gamma}(M(x)) @>>> \mathrm{H}^{2-i}_{\varphi,\gamma}(M(x)^*)^{\vee},
\end{CD}
\end{equation}
where the horizontal arrows are isomorphisms obtained by (Tate) duality. 
Since the commutativity of (\ref{12}) is easy to check, here we only prove the commutativity of 
(\ref{13}). Moreover, we only prove it for $i=2$ since other cases are proved in the same way. For $i=2$, it suffices to show the equality 
$$[\partial(f)g\bold{e}_1]=-[f\partial(g)\bold{e}_1]\in \mathrm{H}^2_{\varphi,\gamma}(\mathcal{R}_L(1))$$
for any $[f\bold{e}_{\delta}]\in \mathrm{H}^2_{\varphi,\gamma}(\mathcal{R}_L(\delta))$ and 
$g\bold{e}_{\delta^{-1}|x|}\in \mathrm{H}^0_{\varphi,\gamma}(\mathcal{R}_L(\delta^{-1}|x|))$. 
Since we have $\partial(fg)=\partial(f)g+f\partial(g)$, the equality follows from the fact that we have 
$[\partial(h)\bold{e}_1]=0$ in $\mathrm{H}^{2}_{\varphi,\gamma}(\mathcal{R}_L(1))$ for any $h\in \mathcal{R}_L$.

  \end{proof}
  
  \begin{rem}\label{4.19f}
  Proposition \ref{4.13.5} and Lemma  \ref{4.14.5} and the following proof of Proposition \ref{4.14} 
  should be generalized to a more general setting. Let $M$ be a de Rham $(\varphi,\Gamma)$-module 
  over $\mathcal{R}_L$ of any rank. In $\S$3 of \cite{Na14a}, we developed the theory of Perrin-Riou's 
  big exponential map for de Rham $(\varphi,\Gamma)$-modules, which is a $\mathcal{R}^{\infty}_L(\Gamma)$-linear 
  map $\mathrm{H}^1_{\psi,\gamma}(\bold{Dfm}(M))\rightarrow \mathrm{H}^1_{\psi,\gamma}(\bold{Dfm}(\bold{N}_{\mathrm{rig}}(M)))$ 
  where $\bold{N}_{\mathrm{rig}}(M)\subseteq M[1/t]$ is a de Rham $(\varphi,\Gamma)$-module equipped with a natural action of 
  the differential operator $\partial_M$ defined by Berger. This big exponential map is defined using the operator $\partial_M$. 
  Our generalization of Perrin-Riou's $\delta(V)$-theorem (Theorem 3.21 \cite{Na14a}) states that this map gives an isomorphism 
  $$\mathrm{Exp}_M:\Delta_{\mathcal{R}^{\infty}_L(\Gamma)}(\bold{Dfm}(M))\isom \Delta_{\mathcal{R}^{\infty}_L(\Gamma)}(\bold{Dfm}(\bold{N}_{\mathrm{rig}}(M))).$$
  Therefore, as a generalization of Proposition \ref{4.13.5}, it seems to be natural to conjecture that the conjectural 
  $\varepsilon$-isomorphisms should satisfy 
  $$\varepsilon_{\mathcal{R}^{\infty}_L(\Gamma), \zeta}(\bold{Dfm}(\bold{N}_{\mathrm{rig}}(M)))=
  \mathrm{Exp}_M\circ
  \varepsilon_{\mathcal{R}^{\infty}_L(\Gamma),\zeta}(\bold{Dfm}(M)), $$ 
  which we want to study in future works.
  
  \end{rem}
  Using these results, we prove Proposition \ref{4.14} for general $k\leqq 0$ as follows.
  
  \begin{proof}(of Proposition \ref{4.14} for general $k\leqq 0$)
  
  Let $\delta=\tilde{\delta}x^{k}$ be a generic homomorphism such that $k\leqq 0$. By the arguments before Proposition \ref{4.14}, it suffices to show the 
  equality $\varepsilon_{L,\zeta}(\mathcal{R}_L(\delta))=\varepsilon^{\mathrm{dR}}_{L,\zeta}(\mathcal{R}_L(\delta))$. This equality follows from the equality
  $\varepsilon_{L,\zeta}(\mathcal{R}_L(\tilde{\delta}))=\varepsilon^{\mathrm{dR}}_{L,\zeta}(\mathcal{R}_L(\tilde{\delta}))$ proved in Proposition \ref{4.14} for $k=0$, since we have 
  $$\varepsilon_{L,\zeta}(\mathcal{R}_L(\delta))=\partial^k\circ\varepsilon_{L,\zeta}(\mathcal{R}_L(\tilde{\delta}))\,\,\text{ and }\,\,
  \varepsilon^{\mathrm{dR}}_{L,\zeta}(\mathcal{R}_L(\delta))=\partial^k\circ\varepsilon^{\mathrm{dR}}_{L,\zeta}(\mathcal{R}_L(\tilde{\delta}))$$ by Proposition \ref{4.13.5} and Lemma \ref{4.14.5}.

  \end{proof}

We next consider the case where $k\geqq 1$. To verify the condition (vi), it suffices to show the following proposition.
\begin{prop}\label{4.15}
If $k\geqq 1$, then the following map 
$$\mathrm{H}^1(\Gamma, \mathcal{R}^{\infty}_L(\delta)^{\psi=1})
\isom \mathrm{H}^1_{\psi,\gamma}(\mathcal{R}_L(\delta))
\xrightarrow{\mathrm{exp}^{-1}_{\mathcal{R}_L(\delta)}}
\bold{D}_{\mathrm{dR}}(\mathcal{R}_L(\delta))$$ sends each element 
$[f_{\mu}\bold{e}_{\delta}]\in \mathrm{H}^1(\Gamma,\mathcal{R}^{\infty}_L(\delta)^{\psi=1})$ to

\begin{itemize}
\item[(1)]$(k-1)!\cdot\frac{\delta(-1)}{\varepsilon_L(W(\mathcal{R}_L(\delta)), \zeta)}\cdot
\frac{1}{t^{k}}\cdot
\int_{\mathbb{Z}_p^{\times}}\delta^{-1}(y)\mu(y)\bold{e}_{\delta}$
when $n(\delta)\not= 0$,
\item[(2)]$(k-1)!\cdot\frac{\mathrm{det}_L(1-\varphi|\bold{D}_{\mathrm{cris}}(\mathcal{R}_L(\delta)^*))}{\mathrm{det}_L(1-\varphi|\bold{D}_{\mathrm{cris}}(\mathcal{R}_L(\delta)))}\cdot\frac{\delta(-1)}{t^{k}}\cdot\int_{\mathbb{Z}_p^{\times}}
\delta^{-1}(y)\mu(y)\bold{e}_{\delta}$
when $n(\delta)= 0$.
\end{itemize}

\end{prop}

\begin{proof}
In the same way as the proof of Proposition \ref{4.14}, it suffices to show the proposition for 
$k=1$ (i.e. $\delta=\tilde{\delta} x$) using Proposition \ref{4.13.5} and Lemma \ref{4.14.5}.

Hence, we assume $k=1$. Then, in the similar way as the proof of Proposition \ref{4.14} (for $k=0$), we have the following commutative diagrams
\begin{equation}
\begin{CD}
\mathrm{H}^1_{\psi,\gamma}(\mathcal{R}_L(\tilde{\delta})) @<<< \mathrm{H}^1(\Gamma, 
\mathcal{R}^{\infty}_L(\tilde{\delta} )^{\psi=1}) @> \iota_{\tilde{\delta}} >> L\bold{e}_{\tilde{\delta}}\\
@VV \partial V @VV \partial V @VV \bold{e}_{\tilde{\delta}}\mapsto -\bold{e}_{\delta} V\\
\mathrm{H}^1_{\psi,\gamma}(\mathcal{R}_L(\delta)) @<<< \mathrm{H}^1(\Gamma, 
\mathcal{R}^{\infty}_L(\delta)^{\psi=1}) @> \iota_{\delta} >> L\bold{e}_{\delta}
\end{CD}
\end{equation}
such that all the arrows are isomorphisms by Lemma \ref{4.12.5}. Hence, reducing to the case of $k=0$, 
it suffices to show that 
the following diagram is commutative,
\begin{equation}\label{15}
\begin{CD}
\mathrm{H}^1(\Gamma, 
\mathcal{R}^{\infty}_L(\tilde{\delta})^{\psi=1})@>>>\mathrm{H}^1_{\psi,\gamma}(\mathcal{R}_L(\tilde{\delta}))@> \mathrm{exp}^{*}_{\mathcal{R}_L(\tilde{\delta}^{-1}x|x|)} >>
\bold{D}_{\mathrm{dR}}(\mathcal{R}_L(\tilde{\delta}))=(L_{\infty}\bold{e}_{\tilde{\delta}})^{\Gamma} \\
@V\partial VV@V \partial VV @ VV a\bold{e}_{\tilde{\delta}}\mapsto \frac{a}{t}\bold{e}_{\delta}V \\
\mathrm{H}^1(\Gamma, 
\mathcal{R}^{\infty}_L(\delta)^{\psi=1})@>>>\mathrm{H}^1_{\psi,\gamma}(\mathcal{R}_L(\delta))@< \mathrm{exp}_{\mathcal{R}_L(\delta)} << \bold{D}_{\mathrm{dR}}(\mathcal{R}_L(\delta))=\left(L_{\infty}\frac{1}{t}\bold{e}_{\delta}\right)^{\Gamma}.
\end{CD}
\end{equation}

The following proof of this commutativity is very similar to that of Theorem 3.10  of \cite{Na14a}.
Take $[f\bold{e}_{\tilde{\delta}}]\in \mathrm{H}^1(\Gamma, \mathcal{R}_L^{\infty}(\tilde{\delta})^{\psi=1})
$. If we denote by
$$\alpha\bold{e}_{\tilde{\delta}}:=\mathrm{exp}^*_{\mathcal{R}_L(\tilde{\delta}^{-1}x|x|)}([|\Gamma_{\mathrm{tor}}|
\mathrm{log}_0(\chi(\gamma))p_{\Delta}(f\bold{e}_{\tilde{\delta}}),0])\in \bold{D}_{\mathrm{dR}}(\mathcal{R}_L(\tilde{\delta})) \subseteq \bold{D}_{\mathrm{dif}}(\mathcal{R}_L(\tilde{\delta})),$$
then it suffices to show the equality
$$\mathrm{exp}_{\mathcal{R}_L(\delta)}\left(\frac{\alpha}{t}\bold{e}_{\delta}\right) 
=|\Gamma_{\mathrm{tor}}|\mathrm{log}_0(\chi(\gamma))
[p_{\Delta}(\partial(f)\bold{e}_{\delta} ),0].$$

We prove this equality as follows. First, we have an equality 
$$\frac{|\Gamma_{\mathrm{tor}}|\mathrm{log}_0(\chi(\gamma))}{\mathrm{log}(\chi(\gamma))}
[\iota_n(p_{\Delta}(f\bold{e}_{\tilde{\delta}}))]=[\alpha\bold{e}_{\tilde{\delta}}] \in \mathrm{H}^1_{\psi,\gamma}(\bold{D}^+_{\mathrm{dif}}(\mathcal{R}_L(\tilde{\delta})))$$ 
for a sufficiently large $n\geqq 1$ by the explicit definition of $\mathrm{exp}^*_{\mathcal{R}_L(\tilde{\delta}^{-1}x|x|)}$ (Proposition 2.16 of \cite{Na14a}).
This equality means that there exists  $y_n\in \bold{D}^+_{\mathrm{dif},n}(\mathcal{R}_L(\tilde{\delta}))^{\Delta}$ such that 
$$\frac{|\Gamma_{\mathrm{tor}}|\mathrm{log}_0(\chi(\gamma))}{\mathrm{log}(\chi(\gamma))}
\iota_n(p_{\Delta}(f\bold{e}_{\tilde{\delta}}))-\alpha\bold{e}_{\tilde{\delta}}=
(\gamma-1)y_n.$$
If we set $\nabla_0:=\frac{\mathrm{log}([\gamma])}{\mathrm{log}(\chi(\gamma))}\in \mathcal{R}_L^{\infty}(\Gamma)$ and define
$$\frac{\nabla_0}{\gamma-1}:=
\frac{1}{\mathrm{log}(\chi(\gamma))}\sum_{m\geqq 1}^{\infty}\frac{(-1)^{m-1}([\gamma]-1)^{m-1}}{m}\in \mathcal{R}^{\infty}_L(\Gamma),$$
then we obtain the following equality
\begin{multline}\label{16}
\iota_n\left(\frac{|\Gamma_{\mathrm{tor}}|\mathrm{log}_0(\chi(\gamma))}{\mathrm{log}(\chi(\gamma))}\frac{\nabla_0}{\gamma-1}(p_{\Delta}(f\bold{e}_{\tilde{\delta}}))\right)\\
=\frac{1}{\mathrm{log}(\chi(\gamma))}\alpha\bold{e}_{\tilde{\delta}} +\nabla_0(y_n)\in \frac{1}{\mathrm{log}(\chi(\gamma))}\alpha\bold{e}_{\tilde{\delta}}+t\bold{D}^+_{\mathrm{dif},n}(\mathcal{R}_L(\tilde{\delta})).
\end{multline}
Since we have $f\bold{e}_{\tilde{\delta}}\in \mathcal{R}_L(\tilde{\delta})^{\psi=1}$, we have 
$$(1-\varphi)(p_{\Delta}(f\bold{e}_{\tilde{\delta}}))\in \mathcal{R}_L(\tilde{\delta})^{\Delta,\psi=0}.$$ Hence, there exists $\beta\in \mathcal{R}_L(\tilde{\delta})^{\Delta,\psi=0}$ such that 
$$(1-\varphi)(p_{\Delta}(f\bold{e}_{\tilde{\delta}}))=(\gamma-1)\beta$$ by (for example) Theorem 3.1.1 of \cite{KPX14}. For any $m\geqq n+1$, then we obtain
\begin{multline*}
\iota_{m}\left(\frac{\nabla_0}{\gamma-1}(p_{\Delta}(f\bold{e}_{\tilde{\delta}}))\right)-\iota_{m-1}\left(\frac{\nabla_0}{\gamma-1}(p_{\Delta}(f\bold{e}_{\tilde{\delta}}))\right)\\
=\iota_m\left((1-\varphi)\left(\frac{\nabla_0}{\gamma-1}(p_{\Delta}(f\bold{e}_{\tilde{\delta}}))\right)\right)
=\iota_m\left(\frac{\nabla_0}{\gamma-1}((1-\varphi)(p_{\Delta}(f\bold{e}_{\tilde{\delta}})))\right)\\
=\iota_m\left(\frac{\nabla_0}{\gamma-1}((\gamma-1)\beta)\right)
=\iota_m(\nabla_0(\beta))\in t\bold{D}^+_{\mathrm{dif},m}(\mathcal{R}_L(\tilde{\delta}))
\end{multline*}
since we have $\nabla_0(\mathcal{R}_L(\tilde{\delta}))\subseteq t\mathcal{R}_L(\tilde{\delta})$.
In particular, we obtain
\begin{equation}\label{17}
\iota_m\left(\frac{\nabla_0}{\gamma-1}(p_{\Delta}(f\bold{e}_{\tilde{\delta}}))\right)- \iota_n\left(\frac{\nabla_0}{\gamma-1}(p_{\Delta}(f\bold{e}_{\tilde{\delta}}))\right)\in 
t\bold{D}^+_{\mathrm{dif},m}(\mathcal{R}_L(\tilde{\delta}))
\end{equation}
for any $m\geqq n+1$ by induction.

Since the map $\mathcal{R}_L(\tilde{\delta})\isom \frac{1}{t}\mathcal{R}_L(\delta)
:g\bold{e}_{\tilde{ \delta}}\mapsto \frac{g}{t}\bold{e}_{\delta}$ is an isomorphism of $(\varphi,\Gamma)$-modules, 
the facts (\ref{16}), (\ref{17}) and the explicit definition of the exponential map (Proposition \ref{2.23} (1)) induce 
the following equality
\[
\begin{array}{ll}
\mathrm{exp}_{\mathcal{R}_L(\delta)}\left(\frac{\alpha}{t}\bold{e}_{\delta}\right) 
&=|\Gamma_{\mathrm{tor}}|\mathrm{log}_0(\chi(\gamma))
\left[(\gamma-1)\frac{\nabla_0}{\gamma-1}\left(p_{\Delta}\left(\frac{f}{t}\bold{e}_{\delta} \right)\right),(\psi-1)\frac{\nabla_0}{\gamma-1}\left(p_{\Delta}\left(\frac{f}{t}\bold{e}_{\delta} \right)\right)\right]\\
&=|\Gamma_{\mathrm{tor}}|\mathrm{log}_0(\chi(\gamma))
\left[\nabla_0\left(p_{\Delta}\left(\frac{f}{t}\bold{e}_{\delta} \right)\right),0\right]\\
&=|\Gamma_{\mathrm{tor}}|\mathrm{log}_0(\chi(\gamma))
[p_{\Delta}(\partial(f)\bold{e}_{\delta} ),0],
\end{array}
\]
where the last equality follows from the equality 
$\nabla_0\left(\frac{f}{t}\bold{e}_{\delta}\right)=\partial(f)\bold{e}_{\delta}$ since
we have $\nabla_0\left(\frac{1}{t}\bold{e}_{\delta}\right)=0$ by the assumption $k=1$, from which 
the commutativity of the diagram (\ref{15}) follows.

\end{proof}

As a corollary of Proposition \ref{4.14}, \ref{4.15}, we verify the conditions (iv), (v) by the density argument as follows.

\begin{corollary}\label{4.17}
Let $M$ be a rank one $(\varphi,\Gamma)$-module 
over $\mathcal{R}_A$. Then the isomorphism $\varepsilon_{A,\zeta}(M):\bold{1}_A\isom 
\Delta_A(M)$ which is 
defined in \S4.1 satisfies the conditions (iv), and (v) of Conjecture \ref{3.9}.

\end{corollary}

\begin{proof}
We first verify the conditions (iv). 
By definition of $\varepsilon_{A,\zeta}(\mathcal{R}_A(\delta)\otimes_A\mathcal{L})$, it suffices to verify these conditions for $(\varphi,\Gamma)$-modules of the form 
$M=\mathcal{R}_A(\delta)$ (i.e. $\mathcal{L}=A$) since the general case immediately follows from this case by Lemma \ref{4.6} .
 Then, in the same way as the proof of 
Proposition \ref{4.13.5}, it suffices to verify these conditions for any $\delta=\delta_{\lambda}\delta_0:\mathbb{Q}_p^{\times}\rightarrow L^{\times}$ such that the point $(\delta_0,\lambda)\in X\times \mathbb{G}_m^{\mathrm{an}}$ is contained in the Zariski dense subset $S_1$ of $X\times\mathbb{G}_m^{\mathrm{an}}$ defined by
$$S_1:=\{(\delta_0,\lambda)\in X(L)\times\mathbb{G}_m^{\mathrm{an}}(L)|[L:\mathbb{Q}_p]<\infty, \delta\text{ is generic }, \mathcal{R}_L(\delta) \text{ is de Rham }\}.$$
For such $\delta$, the conditions (iv) follow from Lemma \ref{3.8} since we have $\varepsilon_{L,\zeta}(\mathcal{R}_L(\delta))=\varepsilon_{L,\zeta}^{\mathrm{dR}}(\mathcal{R}_L(\delta))$ by Proposition \ref{4.14} and Proposition \ref{4.15}.

We next verify the condition (v). 
Let $(\Lambda, T)$ be as in Conjecture \ref{3.9} (v). We recall that we defined a canonical isomorphism 
$$\Delta_{\Lambda}(T)\otimes_{\Lambda}A_{\infty}\isom \Delta_{A_{\infty}}(M_{\infty})$$
(see Example \ref{3.3} for definition and notation). Since any continuous map $\Lambda\rightarrow A$ factors through 
$\Lambda\rightarrow A_{\infty}\rightarrow A$, it suffices to show the equality 
\begin{equation}\label{18}
\varepsilon_{\Lambda,\zeta}(T)\otimes\mathrm{id}_{A_{\infty}}=\varepsilon_{A_{\infty},\zeta}(M_{\infty})(:=\varprojlim_n
\varepsilon_{A_n,\zeta}(M_n)).
\end{equation}

Since the condition (v) is local for $\mathrm{Spf}(\Lambda)$, it suffices to verify (v) for $\Lambda$-representations 
of the form $\Lambda(\tilde{\delta})$ for some $\tilde{\delta}:G^{\mathrm{ab}}_{\mathbb{Q}_p}\rightarrow \Lambda^{\times}$.
Let decompose $\delta=\tilde{\delta}\circ \mathrm{rec}_{\mathbb{Q}_p}$ into $\delta=\delta_{\lambda}\delta_0$.
Since $\Lambda/\mathfrak{m}_{\Lambda}$ is a finite ring, there exists 
$k\geqq 1$ such that $\lambda^k\equiv 1$ (mod $\mathfrak{m}_{\Lambda}$).
Then, we can define a continuous $\mathbb{Z}_p$-algebra homomorphism
$\Lambda_k:=\varprojlim_n \mathbb{Z}_p[Y]/(p, (Y^k-1))^n\rightarrow \Lambda:Y\mapsto \lambda$. 
Hence, the $\Lambda$-representation $\Lambda(\tilde{\delta})$ is obtained by a base change of 
the ``universal" $\mathbb{Z}_p[[\Gamma]]\hat{\otimes}_{\mathbb{Z}_p}\Lambda_k$-representation $T_k^{\mathrm{univ}}$ which 
corresponds to the homomorphism $\delta_k^{\mathrm{univ}}:\mathbb{Q}_p^{\times}\rightarrow( \mathbb{Z}_p[[\Gamma]]\hat{\otimes}_{\mathbb{Z}_p}\Lambda_k)^{\times}:p\mapsto 1\hat{\otimes}Y, a\mapsto [\sigma_a^{-1}]\hat{\otimes}1$ 
for $a\in \mathbb{Z}_p^{\times}$. Hence, it suffices to verify the equality (\ref{18}) for this universal one. In this case, since the 
associated rigid space is an admissible open of $X\times \mathbb{G}_m^{\mathrm{an}}$ defined by 
$$Z_k:=\{(\delta_0,\lambda)\in X\times \mathbb{G}_m^{\mathrm{an}}| |\lambda^k-1|<1\},$$
and the associated $(\varphi,\Gamma)$-module is isomorphic to the restriction of the 
universal one $\bold{Dfm}(\mathcal{R}_{\mathbb{G}_m^{\mathrm{an}}}(\delta_Y))$ defined in the proof of 
Proposition \ref{4.13.5}, it suffices to show the equality 
$$\varepsilon_{\mathbb{Z}_p[[\Gamma]]\hat{\otimes}_{\mathbb{Z}_p}\Lambda_k,\zeta}(T^{\mathrm{univ}}_k)\otimes \mathrm{id}_{\Gamma(Z_k,\mathcal{O}_{Z_k})}
=\varepsilon_{\Gamma(Z_k,\mathcal{O}_{Z_k}),\zeta}(\bold{Dfm}(\mathcal{R}_{\mathbb{G}_m^{\mathrm{an}}}(\delta_Y))|_{Z_k}).$$
Since both sides satisfy the condition $(vi)$ for any point $(\delta_0,\lambda)\in Z_k\cap S_1$ by 
the Kato's theorem (\cite{Ka93b}) and by Proposition \ref{4.14}, \ref{4.15}, and since the set $Z_k\cap S_1$ is Zariski dense 
in $Z_k$, the equality (18) follows by the density argument.

\end{proof}

\subsubsection{Verification of the condition (vi):the exceptional case}
Finally, we verify the condition (vi) in the exceptional case, i.e. 
 $\delta=x^{-k}$ or $\delta=x^{k+1}|x|$ for $k\in \mathbb{Z}_{\geqq 0}$. 
 
 We first reduce all the exceptional cases to the case $\delta=x|x|$. 
 \begin{lemma}\label{4.18}
 We assume that the equality 
 $$\varepsilon_{L,\zeta}(\mathcal{R}_L(x|x|))=\varepsilon_{L,\zeta}^{\mathrm{dR}}(\mathcal{R}_L(x|x|))$$ 
 holds. Then the other equalities
  $$\varepsilon_{L,\zeta}(\mathcal{R}_L(\delta))=\varepsilon_{L,\zeta}^{\mathrm{dR}}(\mathcal{R}_L(\delta))$$ 
 also hold for all $\delta=x^{k+1}|x|, x^{-k}$ for $k\geqq 0$.

 \end{lemma}
 \begin{proof}
 The equality for $\delta=x^0$ follows from that for $\delta=x|x|$ by the compatibility of 
 $\varepsilon_{L,\zeta}^{\mathrm{dR}}(-)$ and $\varepsilon_{L,\zeta}(-)$  with the Tate duality, 
 which are proved in Lemma \ref{3.8} and Proposition \ref{4.17}. 
 Then, the equality for $\delta=x^{k+1}|x|$ (respectively, $\delta=x^{-k}$) follows from that for $\delta=x|x|$ (respectively, 
 $\delta=x^0$) by the compatibility of 
 $\varepsilon_{L,\zeta}^{\mathrm{dR}}(-)$ and $\varepsilon_{L,\zeta}(-)$  with $\partial$, which are proved in 
 Lemma \ref{4.14.5} and Proposition \ref{4.13.5}

 \end{proof}
 
 Finally, it remains to show the equality 
 $$\varepsilon_{L,\zeta}(\mathcal{R}_L(1))=\varepsilon_{L.\zeta}^{\mathrm{dR}}(\mathcal{R}_L(1))$$ 
 (we identify $\mathcal{R}_L(x|x|)=\mathcal{R}_L(1):f\bold{e}_{x|x|}\mapsto f\bold{e}_1$).
 Since $\mathcal{R}_L(1)$ is \'etale, this equality immediately follows from the Kato's result
 since we have $\varepsilon_{L,\zeta}(\mathcal{R}_L(1))=\varepsilon_{\mathcal{O}_L,\zeta}(\mathcal{O}_L(1))\otimes\mathrm{id}_L$ under the canonical isomorphism $\Delta_{L}(\mathcal{R}_L(1))\isom \Delta_{\mathcal{O}_L}(\mathcal{O}_L(1))\otimes_{\mathcal{O}_L}L$ by Corollary \ref{4.17}. 
 However, here we give another proof of this equality only using the framework of $(\varphi,\Gamma)$-modules.

In the remaining part of this section, we prove this equality by explicit calculations.
First, it is easy to see that the inclusion 
$$C^{\bullet}_{\psi,\gamma}(L\cdot 1_{\mathbb{Z}_p}\bold{e}_1)\hookrightarrow 
C^{\bullet}_{\psi,\gamma}(\mathrm{LA}(\mathbb{Z}_p,L)(1))$$
induced by the natural inclusion 
$L\cdot 1_{\mathbb{Z}_p}\bold{e}_1\hookrightarrow \mathrm{LA}(\mathbb{Z}_p,L)(1)$ 
(here,  $1_{\mathbb{Z}_p}$ is the constant function on $\mathbb{Z}_p$ with the constant 
value $1$) 
is  quasi-isomorphism. This quasi-isomorphism and the quasi-isomorphism 
$$C^{\bullet}_{\gamma}(\mathcal{R}_L^{\infty}(1)^{\psi=1})\isom C^{\bullet}_{\psi,\gamma}(\mathcal{R}_L^{\infty}(1)),$$ 
and the 
long exact sequence associated to 
the short exact sequence 
$$0\rightarrow \mathcal{R}^{\infty}_L(1)\rightarrow \mathcal{R}_L(1)
\rightarrow \mathrm{LA}(\mathbb{Z}_p, L)(1)\rightarrow 0$$ 
induce the following isomorphisms
$$\alpha_0:\mathrm{H}^0_{\psi,\gamma}(L\cdot 1_{\mathbb{Z}_p}\bold{e}_1)\isom 
\mathrm{H}^1(\Gamma, \mathcal{R}^{\infty}_L(1)^{\psi=1}),$$
\begin{multline*}
\alpha_1:\mathrm{H}^1_{\psi,\gamma}(\mathcal{R}_L(1))\isom 
\mathrm{H}^1_{\psi,\gamma}(L\cdot 1_{\mathbb{Z}_p}\bold{e}_1):\\
[f_1\bold{e}_1,f_2\bold{e}_2]
\mapsto \left(\mathrm{Res}_0\left(f_1\frac{d\pi}{1+\pi}\right)\cdot 1_{\mathbb{Z}_p}\bold{e}_1,\mathrm{Res}_0\left(f_2\frac{d\pi}{1+\pi}\right)\cdot 1_{\mathbb{Z}_p}\bold{e}_1\right),
\end{multline*}
$$\alpha_2:\mathrm{H}^2_{\psi,\gamma}(\mathcal{R}_L(1))
\isom \mathrm{H}^2_{\psi,\gamma}(L\cdot 1_{\mathbb{Z}_p}\bold{e}_1):[f\bold{e}_1]\mapsto 
\mathrm{Res}_0\left(f\frac{d\pi}{1+\pi}\right)\cdot 1_{\mathbb{Z}_p}\bold{e}_1.$$

Therefore, the isomorphism 
$\bar{\theta}_{\zeta}(\mathcal{R}_L(1)): \boxtimes_{i=1}^2\mathrm{Det}_L(\mathrm{H}^i_{\psi,\gamma}(\mathcal{R}_L(1)))^{(-1)^{i+1}}\isom (L(1), 1)$ defined in (\ref{isom}) is the composition of the isomorphisms $\beta_0$, $\beta_1$ and 
$\iota_{x|x|}$.
\begin{multline*}
\boxtimes_{i=1}^2\mathrm{Det}_L(\mathrm{H}^i_{\psi,\gamma}(\mathcal{R}_L(1)))^{(-1)^{i+1}}
\xrightarrow{\beta_0} \boxtimes_{i=0}^2\mathrm{Det}_L(\mathrm{H}^i_{\psi,\gamma}(L\cdot 1_{\mathbb{Z}_p}\bold{e}_1))^{(-1)^{i+1}}\boxtimes
(\mathrm{H}^1(\Gamma,\mathcal{R}^{\infty}_L(1)^{\psi=1}),1)\\
\xrightarrow{\beta_1} (\mathrm{H}^1(\Gamma,\mathcal{R}^{\infty}_L(1)^{\psi=1}),1)
\xrightarrow{\iota_{x|x|}}(L(1),1),
\end{multline*}
where $\beta_0$ is induced by $\alpha_i$ ($i=0,1,2$), and $\beta_1$ is induced by the canonical isomorphism
$$\beta_1:\boxtimes_{i=0}^2\mathrm{Det}_L(\mathrm{H}^i_{\psi,\gamma}(L\cdot 1_{\mathbb{Z}_p}\bold{e}_1))^{(-1)^{i-1}}\isom \bold{1}_L$$ which is the base change by $f_{x|x|}:\mathcal{R}_L^{\infty}(\Gamma)\rightarrow L:[\gamma]\mapsto \chi(\gamma)^{-1}$ of the isomorphism (\ref{27e}) 
for $M=\mathcal{R}_L$.

By definition, the isomorphism $\beta_1$ is explicitly describe as in the following lemma, which easily follows from the definition (hence, we omit the proof).

\begin{lemma}\label{4.19}
If we denote $\tilde{f}_0:=1_{\mathbb{Z}_p}\bold{e}_1$ 
$($resp. $\tilde{f}_{1,1}:=(1_{\mathbb{Z}_p}\bold{e}_1,0)$ and 
$\tilde{f}_{1,2}:=(0,1_{\mathbb{Z}_p}\bold{e}_1)$, resp.  $\tilde{f}_2:=1_{\mathbb{Z}_p}\bold{e}_1$$)$ for 
the basis of  $\mathrm{H}^0_{\psi,\gamma}(L\cdot 1_{\mathbb{Z}_p}\bold{e}_1)$ 
$($resp.  $\mathrm{H}^1_{\psi,\gamma}(L\cdot 1_{\mathbb{Z}_p}\bold{e}_1)$, resp.
$ \mathrm{H}^2_{\psi,\gamma}(L\cdot 1_{\mathbb{Z}_p}\bold{e}_1)$$)$, then 
the canonical trivialization 
$$\beta_1:(\mathrm{H}^0_{\psi,\gamma}(L\cdot 1_{\mathbb{Z}_p}\bold{e}_1),1)^{-1}
\boxtimes (\mathrm{det}_L\mathrm{H}^1_{\psi,\gamma}(L\cdot 1_{\mathbb{Z}_p}\bold{e}_1), 2)
\boxtimes (\mathrm{H}^2_{\psi,\gamma}(L\cdot 1_{\mathbb{Z}_p}\bold{e}_1,1)^{-1}\isom \bold{1}_L$$ satisfies the equality
 $$\beta_1(\tilde{f}^{\vee}_0\otimes (\tilde{f}_{1,1}\wedge \tilde{f}_{1,2})\otimes \tilde{f}^{\vee}_2)=1.$$
 \end{lemma}

\begin{lemma}\label{4.20}
The isomorphism 
$$\mathrm{H}^0_{\psi,\gamma}(L\cdot 1_{\mathbb{Z}_p}\bold{e}_1)
\xrightarrow{\alpha_0}\mathrm{H}^1(\Gamma, \mathcal{R}_L^{\infty}(1)^{\psi=1})
\xrightarrow{\iota_{x|x|}}L\bold{e}_1$$ 
sends the element $\tilde{f}_0$ to $-\bold{e}_1\in L(1)$.

\end{lemma}

\begin{proof}
Since we have $\mathrm{Col}\left(\frac{1+\pi}{\pi}\right)=1_{\mathbb{Z}_p}$ and $\psi\left(\frac{1+\pi}{\pi}\bold{e}_1\right)=
\frac{1+\pi}{\pi}\bold{e}_1$,
we have 
$$\alpha_0(\tilde{f}_0)=\left[\frac{1}{|\Gamma_{\mathrm{tor}}|\mathrm{log}_0(\chi(\gamma))}(\gamma-1)\left(\frac{1+\pi}{\pi}\bold{e}_1\right)\right]$$
by 
definition of the boundary map.

Since we have 
$$(\gamma-1)\left(\frac{1+\pi}{\pi}\bold{e}_1\right)
=\partial\left(\mathrm{log}\left(\frac{\gamma(\pi)}{\pi}\right)\right)\bold{e}_1\text{ and }
\mathrm{log}\left(\frac{\gamma(\pi)}{\pi}\right)\bold{e}_{|x|}\in \mathcal{R}^{\infty}_L(|x|)^{\psi=1},$$  and 
 have the commutative diagram
\begin{equation}
\begin{CD}
\mathrm{H}^1(\Gamma, \mathcal{R}^{\infty}_L(|x|)^{\psi=1})@> \iota_{|x|} >> L\bold{e}_{|x|} \\
@VV \partial V @ VV \bold{e}_{|x|}\mapsto -\bold{e}_1 V \\
\mathrm{H}^1(\Gamma, \mathcal{R}^{\infty}_L(1)^{\psi=1})@> \iota_{x|x|} >> L\bold{e}_1,
\end{CD}
\end{equation}
 we obtain an equality 
\begin{multline*}
\iota_{x|x|}(\alpha_0(\tilde{f}_0))=\frac{1}{|\Gamma_{\mathrm{tor}}|\mathrm{log}_0(\chi(\gamma))}\iota_{x|x|}\left(\left[\partial\left(\mathrm{log}\left(\frac{\gamma(\pi)}{\pi}\right)\right)\bold{e}_{1}\right]\right)\\
=-\frac{1}{|\Gamma_{\mathrm{tor}}|\mathrm{log}_0(\chi(\gamma))}\int_{\mathbb{Z}_p^{\times}}\mu_{\gamma}(y)\bold{e}_1
\end{multline*}
by Lemma \ref{4.13}, where we define $\mu_{\gamma}\in \mathcal{D}(\mathbb{Z}_p, L)$ such that 
$f_{\mu_{\gamma}}(\pi)=\mathrm{log}\left(\frac{\gamma(\pi)}{\pi}\right)$.

We calculate $\int_{\mathbb{Z}_p^{\times}}\mu_{\gamma}(y)$ as follows. 
Since we have $\psi(\mu_{\gamma})=\frac{1}{p}\mu_{\gamma}$, 
we obtain 

$$\int_{p\mathbb{Z}_p}\mu_{\gamma}(y)=\int_{\mathbb{Z}_p}\psi(\mu_{\gamma})(y)\\
=\frac{1}{p}\int_{\mathbb{Z}_p}\mu_{\gamma}(y).$$
Hence, we obtain  
 \[
\begin{array}{ll}
\int_{\mathbb{Z}_p^{\times}}\mu_{\gamma}(y)&=\int_{\mathbb{Z}_p}\mu_{\gamma}(y)-
\int_{p\mathbb{Z}_p}\mu_{\gamma}(y) =\int_{\mathbb{Z}_p}\mu_{\gamma}(y)-\frac{1}{p}\int_{\mathbb{Z}_p}\mu_{\gamma}(y)\\
&=\frac{p-1}{p}\int_{\mathbb{Z}_p}\mu_{\gamma}(y)
=\frac{p-1}{p}\mathrm{log}\left(\frac{\gamma(\pi)}{\pi}\right)|_{\pi=0}=\frac{p-1}{p}\mathrm{log}(\chi(\gamma)).
\end{array}
\]
Hence, we obtain 
$$\iota_{x|x|}(\alpha_0(\tilde{f}_0))=-\frac{\mathrm{log}(\chi(\gamma))}{
|\Gamma_{\mathrm{for}}|\mathrm{log}_0(\chi(\gamma))}\frac{p-1}{p}\bold{e}_1=-\bold{e}_1$$ 
(for any prime $p$), which proves the lemma.

\end{proof}

In the appendix, we define canonical basis $\{f_{1,1},f_{1,2}\}$ of $\mathrm{H}^1_{\psi,\gamma}(\mathcal{R}_L(1))$,
$f_2\in \mathrm{H}^2_{\psi,\gamma}(\mathcal{R}_L(1))$, 
$e_0\in \mathrm{H}^0_{\psi,\gamma}(\mathcal{R}_L)$ and $\{e_{1,1},e_{1,2}\}$ of $\mathrm{H}^1_{\psi,\gamma}(\mathcal{R}_L)$; see the appendix for the definition.

\begin{corollary}\label{4.21}
The isomorphism 
$$\bar{\theta}_{\zeta}(\mathcal{R}_L(1)): (\mathrm{det}_L\mathrm{H}^1_{\psi,\gamma}(
\mathcal{R}_L(1)),2)\boxtimes(\mathrm{H}^2_{\psi,\gamma}(
\mathcal{R}_L(1)),1)^{-1}\isom (L\bold{e}_1,1)$$ 
sends 
the element $(f_{1,1}\wedge f_{1,2})\otimes f_2^{\vee}$ to $-\frac{p-1}{p}\bold{e}_1$.

\end{corollary}

\begin{proof}
By definition, we have 
$$\alpha_{1}(f_{1,1})=\frac{p-1}{p}\mathrm{log}(\chi(\gamma))\tilde{f}_{1,1}, \,\,\,
\alpha_{1}(f_{1,2})=\frac{p-1}{p}\tilde{f}_{1,2}\text{ and }\alpha_2(f_2)=\frac{p-1}{p}\mathrm{log}(\chi(\gamma))\tilde{f}_2.$$
Then, the corollary follows from the previous lemmas.

\end{proof}

Finally, since one has $\Gamma_{L}(\mathcal{R}_L(1))=1$ and 
$\theta_{\mathrm{dR},L}(\mathcal{R}_L(1),\zeta)$ corresponds to 
the isomorphism 
$$\mathcal{L}_{L}(\mathcal{R}_L(1))=L\bold{e}_1\isom \bold{D}_{\mathrm{dR}}(\mathcal{R}_L(1))=\frac{1}{t}L\bold{e}_1:a\bold{e}_1\mapsto \frac{a}{t}\bold{e}_1,$$
 it suffices to show the following lemma.

\begin{lemma}\label{4.22}
The isomorphism 
$$\theta_L(\mathcal{R}_L(1)):(\mathrm{det}_L\mathrm{H}^1_{\psi,\gamma}(\mathcal{R}_L(1)),2)\boxtimes(
\mathrm{H}^2_{\psi,\gamma}(\mathcal{R}_L(1)),1)^{-1}
\isom (\bold{D}_{\mathrm{dR}}(\mathcal{R}_L(1),1)=\left(L\frac{1}{t}\bold{e}_1,1\right)$$ 
sends the element 
$(f_{1,1}\wedge f_{1,2})\otimes f_2^{\vee}$ to $-\frac{p-1}{p t}\bold{e}_1$.

\end{lemma}

\begin{proof}
By definition, the above isomorphism  is the one which is 
naturally induced by the exact sequence
\begin{multline*}
0\rightarrow \bold{D}_{\mathrm{cris}}(\mathcal{R}_L(1)) 
\xrightarrow{(1-\varphi)\oplus \mathrm{can}} \bold{D}_{\mathrm{cris}}(\mathcal{R}_L(1)) 
\oplus \bold{D}_{\mathrm{dR}}(\mathcal{R}_L(1)) \\
\xrightarrow{\mathrm{exp}_{f,\mathcal{R}_L(1)}\oplus \mathrm{exp}_{\mathcal{R}_L(1)}}\mathrm{H}^1_{\psi,\gamma}
(\mathcal{R}_L(1))_{f}\rightarrow 0 
\end{multline*}
and the isomorphisms
$$\mathrm{exp}_{f,\mathcal{R}_L}^{\vee}:\mathrm{H}^1_{\psi,\gamma}(\mathcal{R}_L(1))/
\mathrm{H}^1_{\psi,\gamma}(\mathcal{R}_L(1))_{f}\isom 
\bold{D}_{\mathrm{cris}}(\mathcal{R}_L)^{\vee}$$
and 
$$\bold{D}_{\mathrm{cris}}(\mathcal{R}_L)^{\vee}\isom \mathrm{H}^2_{\psi,\gamma}(\mathcal{R}_L(1))$$ 
which is the dual of the natural isomorphism $\mathrm{H}^0_{\psi,\gamma}(\mathcal{R}_L)\isom \bold{D}_{\mathrm{cris}}(\mathcal{R}_L)$.

We have $\mathrm{exp}_{\mathcal{R}_L(1)}(\frac{1}{t}\bold{e}_1)=f_{1,2}$ 
by the proof of Lemma \ref{5.1}. Since we have 
$$\mathrm{exp}_{f,\mathcal{R}_L}(1)=e_{1,2}$$ 
for $d_0:=1\in L=\bold{D}_{\mathrm{cris}}(\mathcal{R}_L)$ by the explicit definition of $\mathrm{exp}_f$ (Proposition 
\ref{2.23} (2)), and since we have 
$\langle f_{1,1}, e_{1,2}\rangle=1$ by Lemma \ref{5.4}, we obtain that
$$\mathrm{exp}^{\vee}_{f,\mathcal{R}_L}(f_{1,1})=-d_0^{\vee}\in \bold{D}_{\mathrm{cris}}(\mathcal{R}_L)^{\vee}$$
(we should be careful with the sign). Since the natural isomorphism 
$\mathrm{H}^0_{\psi,\gamma}(\mathcal{R}_L)\isom \bold{D}_{\mathrm{cris}}(\mathcal{R}_L)$ 
sends $e_0$ to $d_0\in L=\bold{D}_{\mathrm{cris}}(\mathcal{R}_L)$, we obtain
$$\bold{D}_{\mathrm{cris}}(\mathcal{R}_L)^{\vee}\rightarrow \mathrm{H}^2_{\psi,\gamma}(\mathcal{R}_L(1)):
d_0^{\vee}\mapsto f_2$$
 by Lemma \ref{5.4}. Using these calculations, the 
lemma follows from diagram chase. 
\end{proof}



\section{Appendix: Explicit calculations of $\mathrm{H}^i_{\varphi,\gamma}(\mathcal{R}_L)$ and 
$\mathrm{H}^i_{\varphi,\gamma}(\mathcal{R}_L(1))$}
In this appendix, we compare 
$\mathrm{H}^i(\mathbb{Q}_p, L(k))$ with $\mathrm{H}^i_{\varphi,\gamma}(\mathcal{R}_{L}(k))$ explicitly for $k=0,1$, and define canonical 
basis of $\mathrm{H}^i_{\varphi,\gamma}(\mathcal{R}_{L}(k))$, which are used to compare
$\varepsilon_{L,\zeta}(\mathcal{R}_L(1))$ with $\varepsilon_{L,\zeta}^{\mathrm{dR}}(\mathcal{R}_L(1))$ in Corollary \ref{4.21} and Lemma 
\ref{4.22}. All the results in this appendix seems to be known (see for example \cite{Ben00}), but here we give another proof of these results in the framework of $(\varphi,\Gamma)$-modules over the Robba ring. Of course, we may assume that $L=\mathbb{Q}_p$ by base change.

We first consider $\mathrm{H}^i_{\varphi,\gamma}(\mathcal{R}_{\mathbb{Q}_p})$.
If we identify by 
$$\mathrm{H}^1(\mathbb{Q}_p,\mathbb{Q}_p)=\mathrm{Hom}_{\mathrm{cont}}(\mathrm{G}_{\mathbb{Q}_p}^{\mathrm{ab}},\mathbb{Q}_p)\isom \mathrm{Hom}_{\mathrm{cont}}(\mathbb{Q}_p^{\times},\mathbb{Q}_p):\tau\mapsto \tau\circ\mathrm{rec}_{\mathbb{Q}_p},$$
then this  has a basis $\{[\mathrm{ord}_p], [\mathrm{log}]\}$ defined by 
$$\mathrm{ord}_p:\mathbb{Q}_p^{\times}\rightarrow \mathbb{Q}_p:p\mapsto 1, a 
\mapsto 0\, \text{ for } a\in \mathbb{Z}_p^{\times},$$
$$\mathrm{log}:\mathbb{Q}_p^{\times}\rightarrow \mathbb{Q}_p:p\mapsto 0, 
a\mapsto \mathrm{log}(a)\, \text{ for } a\in \mathbb{Z}_p^{\times}.$$


We define a basis $e_{0}$ of $\mathrm{H}^0_{\varphi,\gamma}(\mathcal{R}_{\mathbb{Q}_p})$ and 
$\{e_{1,1}, e_{1,2}\}$ of $\mathrm{H}^1_{\varphi,\gamma}(\mathcal{R}_{\mathbb{Q}_p})$ by 
$$e_0=1\in \mathcal{R}_{\mathbb{Q}_p},\,\,e_{1,1}:=[\mathrm{log}(\chi(\gamma)),0], \,\,e_{1,2}:=[0,1],$$
which is independent of the choice of $\gamma$, i,.e. is compatible with the comparison isomorphism $\iota_{\gamma,\gamma'}$. 
We can easily check that the canonical isomorphism $\mathrm{H}^1(\mathbb{Q}_p,\mathbb{Q}_p)\isom 
\mathrm{H}^1_{\varphi,\gamma}(\mathcal{R}_{\mathbb{Q}_p})$ 
sends $[\mathrm{log}]$ to $e_{1,1}$ and $[\mathrm{ord}_p]$ to $e_{1,2}$.



We next consider $\mathrm{H}^1_{\varphi,\gamma}(\mathcal{R}_{\mathbb{Q}_p}(1))$.
Let denote by $$\kappa:\mathbb{Q}_p^{\times}\rightarrow \mathrm{H}^1(\mathbb{Q}_p,\mathbb{Q}_p(1))$$ for 
the Kummer map. 
Composing this with the canonical isomorphism 
$\mathrm{H}^1(\mathbb{Q}_p,\mathbb{Q}_p(1))\isom \mathrm{H}^1_{\varphi,\gamma}(\mathcal{R}_{\mathbb{Q}_p}(1))$, we obtain a homomorphism 
$$\kappa_0:\mathbb{Q}_p^{\times}\rightarrow \mathrm{H}^1_{\varphi,\gamma}(\mathcal{R}_{\mathbb{Q}_p}(1)).$$
We define a homomorphism 
\begin{multline*}
\mathrm{H}^1_{\varphi,\gamma}(\mathcal{R}_{\mathbb{Q}_p}(1))\rightarrow \mathbb{Q}_p\oplus \mathbb{Q}_p:\\
[f_1\bold{e}_1,f_2\bold{e}_1]\mapsto \left(\frac{p}{p-1}\cdot\frac{1}{\mathrm{log}(\chi(\gamma))}\cdot \mathrm{Res}_0\left(f_1\frac{d\pi}{1+\pi}\right), -\frac{p}{p-1}\cdot\mathrm{Res}_0\left(f_2\frac{d\pi}{1+\pi}\right)\right),
\end{multline*}
(we note that $\frac{p-1}{p}\cdot\mathrm{log}(\chi(\gamma))=|\Gamma_{\mathrm{tor}}|\cdot\mathrm{log}_0(\chi(\gamma))$),
which is also independent of the choice of $\gamma$, and is  isomorphism. 
Using this isomorphism, we define a basis $\{f_{1,1}, f_{1,2}\}$ of $\mathrm{H}^1_{\varphi,\gamma}(\mathcal{R}_{\mathbb{Q}_p}(1))$ such that $f_{1,1}$ (respectively $f_{1,2}$) corresponds 
to $(1,0)\in L\oplus L$ (respectively $(0,1)$) by this isomorphism. We want to explicitly describe the map 
$\kappa_0$ using this basis. For this, we first prove the following lemma.

\begin{lemma}\label{5.1}
For each $a\in \mathbb{Z}_p^{\times}$, we have 
$\kappa_0(a)=\mathrm{log}(a)\cdot f_{1,2}$.
\end{lemma}

\begin{proof}
By the classical explicit calculation of exponential map, 
we have 
$$\kappa(a)=\mathrm{exp}_{\mathbb{Q}_p(1)}\left(\frac{\mathrm{log}(a)}{t}\bold{e}_1\right).$$ 
Since we have the following commutative diagram 
\begin{equation*}
\begin{CD}
\bold{D}_{\mathrm{dR}}(\mathbb{Q}_p(1))@> \mathrm{exp}_{\mathbb{Q}_p(1)}>> \mathrm{H}^1(\mathbb{Q}_p, \mathbb{Q}_p(1)) \\
@ V \isom VV @ V\isom VV\\
\bold{D}_{\mathrm{dR}}(\mathcal{R}_{\mathbb{Q}_p}(1))
@ > \mathrm{exp}_{\mathcal{R}_{\mathbb{Q}_p}(1)} >> \mathrm{H}^1_{\varphi,\gamma}(\mathcal{R}_{\mathbb{Q}_p}(1))
\end{CD}
\end{equation*}
by Proposition \ref{2.27},
it suffices to show that 
$$\mathrm{exp}_{\mathcal{R}_{\mathbb{Q}_p}(1)}\left(\frac{1}{t}\bold{e}_1\right)=f_{1,2}.$$
We show this equality as follows.
We first take some $f\in (\mathcal{R}^{\infty}_{\mathbb{Q}_p})^{\Delta}$ such that 
$f(\zeta_{p^n}-1)=\frac{1}{p^n}$ for any $n\geqq 0$, which is possible since 
we have an isomorphism $\mathcal{R}^{\infty}_{\mathbb{Q}_p}/t\isom \prod_{n\geqq 0}
\mathbb{Q}_p(\zeta_{p^n}):\overline{f}\mapsto (f(\zeta_{p^n}-1))_{n\geqq 0}$ by Lazard's theorem \cite{La62}.
Then, the element $\frac{f}{t}\bold{e}_1\in \left(\frac{1}{t}\mathcal{R}_{\mathbb{Q}_p}(1)\right)^{\Delta}$ satisfies 
$$\iota_n\left(\frac{f}{t}\bold{e}_1\right)-\frac{1}{t}\bold{e}_1 \in \bold{D}^+_{\mathrm{dif},n}(\mathcal{R}_{\mathbb{Q}_p}(1))$$ for any $n\geqq 1$, since we have 
$$\iota_n\left(\frac{f}{t}\bold{e}_1\right)\equiv p^n\cdot \frac{f(\zeta_{p^n}-1)}{t}\bold{e}_1=
\frac{1}{t}\bold{e}_1\, (\text{ mod } \bold{D}^+_{\mathrm{dif},n}(\mathcal{R}_{\mathbb{Q}_p}(1))).$$ By  the explicit definition of $\mathrm{exp}_{\mathcal{R}_{\mathbb{Q}_p}(1)}$ (Proposition \ref{2.23} (1)), then we 
have 
$$\mathrm{exp}_{\mathcal{R}_{\mathbb{Q}_p}(1)}\left(\frac{1}{t}\bold{e}_1\right)=\left[(\gamma-1)\left(\frac{f}{t}\bold{e}_1\right), (\varphi-1)\left(\frac{f}{t}\bold{e}_1\right)\right]\in\mathrm{H}^1_{\varphi,\gamma}(\mathcal{R}_{\mathbb{Q}_p}(1)).$$
Hence, it suffices to show that 
$$\mathrm{Res}_0\left(\frac{\gamma(f)-f}{t}\cdot\frac{d\pi}{1+\pi}\right)=0$$
and
$$\mathrm{Res}_0\left(\left(\frac{\varphi(f)}{p}-f\right)\cdot\frac{1}{t}\cdot\frac{d\pi}{1+\pi}\right)=-\frac{p-1}{p}.$$
Here, we only calculate $\mathrm{Res}_0\left(\left(\frac{\varphi(f)}{p}-f\right)\cdot\frac{1}{t}\cdot\frac{d\pi}{1+\pi}\right)$ (the 
calculation of $\mathrm{Res}_0\left(\frac{\gamma(f)-f}{t}\cdot\frac{d\pi}{1+\pi}\right)$ is similar). 
By definition of $f$, we have 
 $$\frac{\varphi(f)(\zeta_{p^n}-1)}{p}-f(\zeta_{p^n}-1)=\frac{f(\zeta_{p^{n-1}}-1)}{p}-f(\zeta_{p^n}-1)=\frac{1}{p}\cdot\frac{1}{p^{n-1}}-\frac{1}{p^n}=0$$
 for each $n\geqq 1$.
 Hence, we have $$\left(\frac{\varphi(f)}{p}-f\right)\in \left(\prod^{\infty}_{n\geqq 1}\frac{Q_n(\pi)}{p}\right)\mathcal{R}^{\infty}_{\mathbb{Q}_p}$$ by the theorem of Lazard \cite{La62}, where we define $Q_n(\pi):=\varphi^{n-1}\left(\frac{\varphi(\pi)}{\pi}\right)$ for each $n\geqq 1$. Since we have 
 $t=\pi\prod_{n\geqq 1}\left(\frac{Q_n(\pi)}{p}\right)$, we obtain the equality 
  \begin{multline*}
\mathrm{Res}_0\left(\left(\frac{\varphi(f)}{p}-f\right)\cdot\frac{1}{t}\cdot\frac{d\pi}{1+\pi}\right)=\left(\left(\frac{\varphi(f)}{p}-f\right)\cdot\frac{1}{\prod_{n\geqq 1}^{\infty} \frac{Q_n(\pi)}{p}}\cdot\frac{1}{1+\pi}\right)|_{\pi=0}\\
=\left(\frac{\varphi(f)}{p}-f\right)|_{\pi=0}=\frac{f(0)}{p}-f(0)=-\frac{p-1}{p},
\end{multline*}

where the second equality follows from the fact that $\frac{Q_n(0)}{p}=1$ for $n\geqq 1$, which 
proves the lemma.

\end{proof}

Before  calculating $\kappa_0(p)\in \mathrm{H}^1_{\varphi,\gamma}(\mathcal{R}_{\mathbb{Q}_p}(1))$, we explicitly describe the Tate's trace map 
 in terms of $(\varphi,\Gamma)$-modules.
We note that we normalize the Tate's trace map 
$$\mathrm{H}^2(\mathbb{Q}_p, \mathbb{Q}_p(1))\isom \mathbb{Q}_p$$
such that the cup product pairing
$$\langle , \rangle:\mathrm{H}^1(\mathbb{Q}_p, \mathbb{Q}_p(1))\times \mathrm{H}^1(\mathbb{Q}_p,\mathbb{Q}_p)
\xrightarrow{\cup}\mathrm{H}^2(\mathbb{Q}_p, \mathbb{Q}_p(1))\isom \mathbb{Q}_p$$ 
satisfies that 
$$\langle\kappa(a), [\tau]\rangle=\tau(a)$$ 
for $a\in \mathbb{Q}_p^{\times}$ and $[\tau]\in \mathrm{Hom}(\mathbb{Q}_p^{\times}, 
\mathbb{Q}_p)=\mathrm{H}^1(\mathbb{Q}_p, \mathbb{Q}_p)$ (remark that this normalization coincides with the one used in \S2.4 of \cite{Na14a} and the $(-1)$-times of Kato's one in II. \S 1.4 of \cite{Ka93a}).

\begin{prop}\label{explicitdual}
The map $\iota_{\gamma}:\mathrm{H}^2_{\varphi,\gamma}(\mathcal{R}_{\mathbb{Q}_p}(1))
\isom \mathrm{H}^2(\mathbb{Q}_p, \mathbb{Q}_p(1))\isom \mathbb{Q}_p$ 
which is  the composition of the canonical isomorphism 
$\mathrm{H}^2_{\varphi,\gamma}(\mathcal{R}_{\mathbb{Q}_p}(1))\isom 
\mathrm{H}^2(\mathbb{Q}_p,\mathbb{Q}_p(1))$ with the Tate's trace map is explicitly defined 
by 
$$\iota_{\gamma}([f\bold{e}_1])=\frac{p}{p-1}\cdot \frac{1}{\mathrm{log}(\chi(\gamma))}\mathrm{Res}_0\left(f\frac{d\pi}{1+\pi}\right).$$

\end{prop}

\begin{proof}
Since the map 
$$\iota:\mathrm{H}^2_{\varphi,\gamma}(\mathcal{R}_{\mathbb{Q}_p}(1))\isom
\mathbb{Q}_p: [f\bold{e}_1]\mapsto \mathrm{Res}_0\left(f\frac{d\pi}{1+\pi}\right)$$
is a well-defined isomorphism, there exists unique $\alpha\in \mathbb{Q}_p^{\times}$ such that 
$\iota_{\gamma}=\alpha \cdot\iota$. We calculate $\alpha$ as follows. 

We recall that  the element $[\mathrm{log}(\chi(\gamma)),0]\in \mathrm{H}^1_{\varphi,\gamma}(\mathcal{R}_{\mathbb{Q}_p})$ is the image of 
$[\mathrm{log}]\in \mathrm{H}^1(\mathbb{Q}_p, \mathbb{Q}_p)$ by the comparison isomorphism. 
By the proof of Lemma \ref{5.1}, for each $a\in \mathbb{Z}_p^{\times}$, 
we have 
$$\kappa_0(a)=\mathrm{log}(a)[(\gamma-1)(\frac{f}{t}\bold{e}_1), (\varphi-1)(\frac{f}{t}\bold{e}_1)]
\in \mathrm{H}^1_{\varphi,\gamma}(\mathcal{R}_{\mathbb{Q}_p}(1)),$$ where $f\in \mathcal{R}^{\infty}_{\mathbb{Q}_p}$ is an element defined in the proof of Lemma \ref{5.1}.
Since the cup products are compatible with the comparison isomorphism (see Remark \ref{2.12.111}), then 
we have 

\begin{equation}\label{20}
\iota_{\gamma}(\kappa_0(a)\cup [\mathrm{log}(\chi(\gamma)),0]) 
=\langle \kappa(a),[\mathrm{log}]\rangle=\mathrm{log}(a).
\end{equation}
By definition of the cup product, we have 
 
\begin{multline*}
\kappa_0(a)\cup [\mathrm{log}(\chi(\gamma)),0]=\mathrm{log}(a)\left[(\varphi-1)\left(\frac{f}{t}\bold{e}_1\right)
\otimes \varphi(\mathrm{log}(\chi(\gamma)))\right]\\
=
-\mathrm{log}(a)\mathrm{log}(\chi(\gamma))\left[(\varphi-1)\left(\frac{f}{t}\bold{e}_1\right)\right]
\in \mathrm{H}^2_{\varphi,\gamma}(\mathcal{R}_{\mathbb{Q}_p}(1)).
\end{multline*}
Since we have $\mathrm{Res}_0\left((\varphi-1)(\frac{f}{t})\cdot\frac{d\pi}{1+\pi}\right)=
-\frac{p-1}{p}$ by the proof of Lemma \ref{5.1}, 
we obtain 

\begin{multline*}
\iota_{\gamma}(\kappa_0(a)\cup [\mathrm{log}(\chi(\gamma)),0])=\alpha \cdot\iota(\kappa_0(a)\cup [\mathrm{log}(\chi(\gamma)),0]) \\
=-\alpha\cdot\mathrm{log}(\chi(\gamma))\cdot\mathrm{log}(a)\cdot\iota\left(\left[(\varphi-1)\left(\frac{f}{t}\bold{e}_1\right)\right]\right)
=\alpha\cdot\mathrm{log}(\chi(\gamma))\cdot\mathrm{log}(a)\cdot\frac{p-1}{p}.
\end{multline*}
Comparing this equality with the equality (\ref{20}),  we obtain 
$$\alpha=\frac{p}{p-1}\cdot\frac{1}{\mathrm{log}(\chi(\gamma))},$$
which prove the proposition.

\end{proof}

Finally, we prove the following lemma, which completes the calculation of the map 
$\kappa_0:\mathbb{Q}_p^{\times}\rightarrow \mathbb{Q}_p\oplus\mathbb{Q}_p$.

\begin{lemma}\label{5.3}
$$\kappa_0(p)=f_{1,1}.$$
\end{lemma}

\begin{proof}
Take $f_{1,1}=[f_1\bold{e}_1,f_2\bold{e}_1]\in \mathrm{H}^1_{\varphi,\gamma}
(\mathcal{R}_{\mathbb{Q}_p}(1))$ a representative of $f_{1,1}$. 
By definition of the cup product,  we have

\begin{multline*}
\iota_{\gamma}(f_{1,1}\cup e_{1,1})=\iota_{\gamma}(f_{1,1}\cup [\mathrm{log}(\chi(\gamma)), 0])\\
=-\iota_{\gamma}([f_2\bold{e}_1\otimes \varphi(\mathrm{log}(\chi(\gamma))])
=-\frac{p}{p-1}\mathrm{Res}_0\left(f_2\frac{d\pi}{1+\pi}\right)=0,
\end{multline*}
and 
\begin{multline*}
\iota_{\gamma}(f_{1,1}\cup e_{1,2}) =\iota_{\gamma}(f_{1,1}\cup [0, 1])\\
=\iota_{\gamma}([f_1\bold{e}_1\otimes \gamma(1)])
=\frac{p}{p-1}\cdot\frac{1}{\mathrm{log}(\chi(\gamma))}\cdot\mathrm{Res}_0\left(f_1\frac{d\pi}{1+\pi}\right)=1
\end{multline*}
by Proposition \ref{explicitdual}. 
Since $\kappa(p)\in \mathrm{H}^1(\mathbb{Q}_p, \mathbb{Q}_p(1))$ 
satisfies the similar formulae 
$$\langle\kappa(p), [\mathrm{ord}_p]\rangle=1,\,\, \langle \kappa(p),[\mathrm{log}]\rangle=0,$$ 
we obtain the equality 
$$\kappa_0(p)=f_{1,1}. $$

\end{proof}

Using these lemmas, we obtain the following lemma.
We define the basis $f_2$ of $\mathrm{H}^2_{\varphi,\gamma}(
\mathcal{R}_L(1))$ by $f_2:=\iota_{\gamma}^{-1}(1)$.

\begin{lemma}\label{5.4}
The Tate's duality pairings
$$\langle , \rangle:\mathrm{H}^1_{\varphi,\gamma}(\mathcal{R}_L(1))\times 
\mathrm{H}^1_{\varphi,\gamma}(\mathcal{R}_L)\xrightarrow{\cup}
\mathrm{H}^2_{\varphi,\gamma}(\mathcal{R}_L(1))\xrightarrow{\iota_{\gamma}} L$$ 
and
$$\langle,\rangle:\mathrm{H}^2_{\varphi,\gamma}(\mathcal{R}_L(1))\times 
\mathrm{H}^0_{\varphi,\gamma}(\mathcal{R}_L)\xrightarrow{\cup}
\mathrm{H}^2_{\varphi,\gamma}(\mathcal{R}_L(1))\xrightarrow{\iota_{\gamma}} L$$ 
satisfy the following:

$$\langle f_{1,1}, e_{1,1}\rangle=0, \,\, \langle f_{1,1}, e_{1,2}\rangle=1 $$ 
and 
$$\langle f_{1,2},e_{1,1}\rangle=1, \,\, \langle f_{1,2}, e_{1,2}\rangle=0.$$

$$\langle f_2, e_0\rangle=1.$$

\end{lemma}
\begin{proof}
That we have $\langle f_{1,1},e_{1,1}\rangle=0$ and $\langle f_{1,1},e_{1,2}\rangle=1$ 
is proved in Lemma \ref{5.3}. We prove the formula for $f_{1,2}$. 
By Lemma \ref{5.1}, we have an equality  $f_{1,2}=\frac{1}{\mathrm{log}(a)}
\kappa_0(a)$ for any non-torsion $a\in \mathbb{Z}_p^{\times}$.  
Hence, we obtain 
$$\langle f_{1,2}, e_{1,1}\rangle=\frac{1}{\mathrm{log}(a)}\langle \kappa(a), [\mathrm{log}]\rangle=1$$ 
and 
$$\langle f_{1,2}, e_{1,2}\rangle=\frac{1}{\mathrm{log}(a)}\langle \kappa(a), [\mathrm{ord}_p]\rangle=0$$
by the compatibility of the cup products.
Finally, that $\langle f_2,e_0\rangle=1$ is trivial by definition.

\end{proof}
\subsection*{Acknowledgement}
The author thanks Seidai Yasuda for introducing me to the Kato's global and local 
$\varepsilon$-conjectures. He also thanks Iku Nakamura for constantly encouraging him. 
This work is supported in part by the Grant-in-aid 
(NO. S-23224001) for Scientific Research, JSPS.


\begin{thebibliography}{99}
\bibitem[BelCh09]{BelCh09}
 J.Bella\"iche, G.Chenevier, Families of Galois representations and Selmer groups, 
 Ast\'erisque $\bold{324}$, Soc. Math. France (2009).
\bibitem[Ben00]{Ben00}
D.Benois,  On Iwasawa theory of crystalline representations. Duke Math. J. $\bold{104}$ (2000) 211-267.
\bibitem[BB08]{BB08}
 D.Benois, L.Berger, Th\'eorie d'Iwasawa des repr\'esentations cristallines. II, Comment. Math. Helv. $\bold{83}$ (2008), no. 3, 603-677.

\bibitem[Ber02]{Ber02}
L.Berger, Repr\'esentations $p$-adiques et \'equations diff\'erentielles, Invent. Math. $\bold{148}$ (2002), 219-284.
\bibitem[Ber08a]{Ber08a}
L.Berger, Construction de ($\varphi,\Gamma$)-modules: repr\'esentations $p$-adiques et $B$-paires,  Algebra and Number Theory $\bold{2}$ (2008), no. 1, 91-120.
\bibitem[Ber08b]{Ber08b}
L. Berger, \'Equations diff\'erentielles $p$-adiques et ($\varphi$,N)-modules filtres, Asterisque $\bold{319}$ (2008), 13-38, Repr\'esentations $p$-adiques de groupes $p$-adiques. I. Repr\'esentations galoisiennes et ($\varphi$,N)-modules.
\bibitem[Ber09]{Ber09}
L.Berger, Presque $\mathbb{C}_p$-repr\'esentations et $(\varphi,\Gamma)$-modules, J. Inst. Math. Jussieu $\bold{8}$ (2009), no. 4, 653-668.
\bibitem[BK90]{BK90}
S. Bloch, K.Kato, L-functions and Tamagawa numbers of motives. The Grothendieck Festschrift, Vol. I, 333-400, Progr. Math. $\bold{86}$, Birkh\"auser Boston, Boston, MA 1990.
\bibitem[Ch13]{Ch13}
G.Chenevier, Sur la densit\'e des representations cristallines de $\mathrm{Gal}(\overline{\mathbb{Q}}_p/\mathbb{Q}_p)$, Math. Annalen $\bold{335}$, 1469-1525 (2013).
\bibitem[Co08]{Co08}
P.Colmez, Repr\'esentations triangulines de dimension $2$, Ast\'erisque $\bold{319}$ (2008), 213-258.
\bibitem[Co10]{Co10}
P.Colmez, Repr\'esentations de $\mathrm{GL}_2(\mathbb{Q}_p)$ et $(\varphi,\Gamma)$-modules, Ast\'erisque $\bold{330}$ (2010), 281-509.
\bibitem[Dee01]{Dee01}
J.Dee, $(\varphi,\Gamma)$-modules for families of Galois representations, Journal of Algebra $\bold{235}$
(2001), 636-664.
\bibitem[De73]{De73}
P.Deligne, Les constantes des \'eqations fonctionelles des fonctions $L$, Modular functions of one variable, II (Proc. Internat. Summer School, Univ. Antwerp, Antwerp, 1972), pp. 501-597. Lecture Notes in Math., Vol. $\bold{349}$, Springer, Berlin, 1973.
\bibitem[Em]{Em}
M. Emerton, Local-global compatibility in the $p$-adic Langlands programme for $\mathrm{GL}_{2/\mathbb{Q}}$,preprint.
\bibitem[FP94]{FP94}
J.-M. Fontaine, B.Perrin-Riou, Autour des conjectures de Bloch et Kato; Cohomologie galoisienne et valeurs de fonctions $L$. Motives (Seattle, WA, 1991), 599-706, Proc. Sympos. Pure Math., $\bold{55}$, Part 1, Amer. Math. Soc., Providence, RI, 1994.
\bibitem[FK06]{FK06}
T.Fukaya, K.Kato, A formulation of conjectures on $p$-adic zeta functions in non commutative Iwasawa theory, Proceedings of the St. Petersburg Mathematical Society. Vol. XII (Providence, RI), Amer. Math. Soc. Transl. Ser. 2, vol. $\bold{219}$, Amer. Math. Soc., 2006, pp. 1-85. 





\bibitem[Ka93a]{Ka93a}
K. Kato, Lectures on the approach to Iwasawa theory for Hasse-Weil $L$-functions via $B_{dR}$. Arithmetic algebraic geometry, Lecture Notes in Mathematics $\bold{1553}$, Springer-
Verlag, Berlin, 1993, 50-63.
\bibitem[Ka93b]{Ka93b}
K.Kato, Lectures on the approach to Iwasawa theory for Hasse-Weil $L$-functions via $B_{dR}$. 
Part II. Local main conjecture, unpublished preprint.
\bibitem[KL10]{KL10}
K.Kedlaya, R.Liu, On families of $(\varphi,\Gamma)$-modules, 
Algebra and Number Theory $\bold{4}$ (2010), no7, 943-967.
\bibitem[KPX14]{KPX14}
K.Kedlaya, J.Pottharst, L.Xiao, Cohomology of arithmetic families of $(\varphi,\Gamma)$-modules, 
J. Amer. Math. Soc. $\bold{27}$ (2014), no. 4, 1043-1115.
\bibitem[Ki03]{Ki03}
M.Kisin, Overconvergent modular forms and the Fontaine-Mazur 
conjecture, Invent. Math. $\bold{153}$ (2003), 373-454.
\bibitem[Kis10]{Kis10}
M. Kisin, Deformations of $G_{\mathbb{Q}_p}$ and $GL_2(\mathbb{Q}_p)$ representations, Ast\'erisque $\bold{330}$, 511-528 (2010).
\bibitem[KM76]{KM76}
F. Knudsen, D. Mumford, The projectivity of the moduli space of stable curves I, Math. Scand. $\bold{39}$ (1976), no. 1, 19-55.
\bibitem[La62]{La62}
M. Lazard, Les z\'eros des fonctions analytiques d'une variable sur un corps valu\'e complet. Inst. Hautes \'Etudes Sci. Publ. Math. No. $\bold{14}$ (1962),  47-75.
\bibitem[LVZ14]{LVZ14}
D.Loeffler, O.Venjakob, S.L.Zerbes, Local $\varepsilon$-isomorphism, Kyoto J Math.

\bibitem[Na14a]{Na14a}
K.Nakamura, Iwasawa theory of de Rham $(\varphi,\Gamma)$-modules over the 
Robba ring, J. Inst. Math. Jussieu $\bold{13}$ (2014), Issue 01, no.5, 65-118.

\bibitem[Na14b]{Na14b}
K.Nakamura, Zariski density of crystalline representations for any $p$-adic field, J.Math.Sci.Univ. Tokyo $\bold{21}$ (2014), 79-127.
\bibitem[Na]{Na}
K.Nakamura, Local $\varepsilon$-isomorphisms for rank two $p$-adic representations of 
$\mathrm{Gal}(\overline{\mathbb{Q}}_p/\mathbb{Q}_p)$ and a functional equation of Kato's 
Euler system, in preparation.
\bibitem[Po13a]{Po13a}
J. Pottharst, Analytic families of finite-slope Selmer groups, Algebra and Number Theory $\bold{7}$ (2013), no.7, 1571-1612.
\bibitem[Po13b]{Po13b}
J. Pottharst, Cyclotomic Iwasawa theory of motives, preprint.
\bibitem[ST03]{ST03}
P. Schneider, J. Teitelbaum, Algebras of $p$-adic distributions and admissible representations, Invent. Math. $\bold{153}$ (2003), no.1, 145-196.
\bibitem[Ve13]{Ve13}
O.Venjakob, On Kato's local $\varepsilon$-isomorphism Conjecture for rank one Iwasawa modules, Algebra and Number Theory $\bold{7}$ (2013), no.10, 2369-2416.






\end{thebibliography}
\end{document}